\documentclass[11pt]{article}
\usepackage[margin=1.25in]{geometry}  
\usepackage{geometry}                
\usepackage{amsmath}
\usepackage{mathrsfs}
\usepackage{hyperref}
\usepackage{mathtools}
\usepackage{amsthm}
\usepackage{amssymb}
\usepackage{enumitem}
\usepackage{placeins}
\usepackage{graphicx}
\usepackage{epstopdf}
\usepackage{float}
\usepackage{caption}
\usepackage{subcaption}
\usepackage{xcolor}
\usepackage[square,numbers,sort&compress]{natbib} 
\usepackage{ stmaryrd }

\theoremstyle{plain}
\newtheorem{theorem}{Theorem}[section]
\newtheorem{corollary}[theorem]{Corollary}
 
\newtheorem{prop}[theorem]{Proposition}
\newtheorem{lemma}[theorem]{Lemma}

\theoremstyle{definition}

\newtheorem{remark}[theorem]{Remark}
\newtheorem{example}[theorem]{Example} 

\newcommand{\Z}{\mathbb{Z}}

\newcommand{\N}{\mathbb{N}}
\newcommand{\R}{\mathbb{R}}
\newcommand{\p}{\mathbb{P}}
\newcommand{\E}{\mathbb{E}}

\renewcommand*{\P}{\mathbb{P}}

\newcommand{\indic}{\mathbf{1}}

\newcommand{\fl}[1]{{\left\lfloor #1 \right\rfloor}}
\newcommand{\cl}[1]{{\left\lceil #1 \right\rceil}}
\newcommand{\sset}{\subset}
\newcommand{\lf}{\left}
\newcommand{\rg}{\right}

\newcommand{\mathand}{\;\text{and}\;}

\newcommand{\mathif}{\;\text{if}\;}

\newcommand{\ga}{\gamma}
\newcommand{\Ga}{\Gamma}
\newcommand{\ep}{\epsilon}

\newcommand{\de}{\delta}
\newcommand{\be}{\beta}

\newcommand{\sig}{\sigma}

\newcommand{\la}{\lambda}
\newcommand{\al}{\alpha}

\newcommand{\del}{\partial}

\newcommand{\cA}{\mathcal{A}}
\newcommand{\cB}{\mathcal{B}}
\newcommand{\cC}{\mathcal{C}}

\newcommand{\cF}{\mathcal{F}}
\newcommand{\cG}{\mathcal{G}}

\newcommand{\cP}{\mathcal{P}}

\newcommand{\fA}{\mathfrak{A}}
\newcommand{\fB}{\mathfrak{B}}

\newcommand{\fL}{\mathfrak{L}}
\newcommand{\fM}{\mathfrak{M}}

\newcommand{\sC}{\mathscr{C}}

\newcommand{\T}{\textsf{T}}
\renewcommand{\S}{\textsf{S}}

\usepackage{MnSymbol}

\newcommand{\eqd}{\stackrel{d}{=}}

\newcommand{\cvgd}{\stackrel{d}{\to}}

\newcommand{\X}{\times}

\newcommand{\cvgdown}{\downarrow}

\newcommand{\smin}{\setminus}

\DeclareMathOperator{\Ai}{Ai}
\newcommand{\bv}{\mathbf{v}}
\newcommand{\bw}{\mathbf{w}}
\newcommand{\bx}{\mathbf{x}}
\newcommand{\by}{\mathbf{y}}
\newcommand{\bz}{\mathbf{z}}

\newcommand{\bu}{\mathbf{u}}

\newcommand{\ba}{\mathbf{a}}
\newcommand{\bb}{\mathbf{b}}

\newcommand{\bfe}{\mathbf{e}}
\newcommand{\bm}{\mathbf{m}}
\newcommand{\bn}{\mathbf{n}}

\newcommand{\II}[1]{\llbracket #1 \rrbracket}

\newcommand{\NI}{\operatorname{NI}}
\title{Wiener densities for the Airy line ensemble}
\author{Duncan Dauvergne\footnote{Department of Mathematics, University of Toronto, 40 St. George St. Toronto, ON, Canada, M5S2E4, duncan.dauvergne@utoronto.ca  }}
\begin{document}
	\maketitle
	
	\begin{abstract}
		The parabolic Airy line ensemble $\mathfrak A$ is a central limit object in the KPZ universality class and related areas. On any compact set $K = \{1, \dots, k\} \times [a, a + t]$, the law of the recentered ensemble $\mathfrak A - \mathfrak A(a)$ has a density $X_K$ with respect to the law of $k$ independent Brownian motions. We show that 
		$$
		X_K(f) = \exp \left(-\textsf{S}(f) + o(\textsf{S}(f))\right)
		$$
		where $\textsf{S}$ is an explicit, tractable, non-negative function of $f$. We use this formula to show that $X_K$ is bounded above by a $K$-dependent constant, give a sharp estimate on the size of the set where $X_K < \epsilon$ as $\epsilon \to 0$, and prove a large deviation principle for $\mathfrak A$. We also give density estimates that take into account the relative positions of the Airy lines, and prove sharp two-point tail bounds that are stronger than those for Brownian motion. These estimates are a key input in the classification of geodesic networks in the directed landscape \cite{dauvergne2023geodesic}. The paper is essentially self-contained, requiring only tail bounds on the Airy point process and the Brownian Gibbs property as inputs.

	\end{abstract}


	\section{Introduction}

	The Airy line ensemble is a stationary random sequence of functions $\cA = \{\cA_i : \R\to \R, i \in \N\}$ satisfying $\cA_1 > \cA_2 > \dots$. It was first introduced by Pr\"ahofer and Spohn \cite{prahofer2002scale} via a determinantal formula, see \eqref{E:airy-kernel} below. Pr\"ahofer and Spohn showed that the top line $\cA_1$--known as the Airy (or Airy$_2$) process--describes the scaling limit at a fixed time for a certain one-dimensional random growth model in the KPZ (Kardar-Parisi-Zhang) universality class started from a point. The remaining lines describe the scaling limit of a richer multi-layer random growth model. Since this work, the Airy process and the Airy line ensemble have been shown to be universal limit objects in the KPZ universality class and other related branches of probability. Among other results, we now know:
	\begin{enumerate}[label=(\roman*)]
		\item The Airy process and the Airy line ensemble  appear as one-parameter scaling limits of classical solvable random metric and random growth models in the KPZ universality class (e.g. zero temperature models including exponential/geometric/Poisson/Brownian last passage percolation and tasep, see \cite{johansson2003discrete, prahofer2002scale, CH, dauvergne2019uniform}; positive temperature models including the KPZ equation and the KPZ line ensemble, see \cite{corwin2016kpz, quastel2023convergence, virag2020heat, wu2021tightness}). See \citep{ ferrari2010random, quastel2011introduction, corwin2012kardar, romik2015surprising, borodin2016lectures} for background on the KPZ universality class and related areas. 
		\item The richest scaling limit in the KPZ universality class--the directed landscape--can be described via a last passage problem across the Airy line ensemble, see \cite{DOV}. Moreover, for last passage models, convergence to the Airy line ensemble is the only integrable input required for proving convergence to the directed landscape, see \cite{dauvergne2021scaling}. 
		\item The Airy line ensemble is the limit as $n \to \infty$ at the edge of a system of $n$ non-intersecting Brownian motions, see \cite{adler2005pdes, CH}. This implies that it is the edge limit for the eigenvalue process in Hermitian-matrix valued Brownian motion (Dyson's Brownian motion with $\beta = 2$).
		\item The Airy line ensemble appears as the limit at the boundary between the frozen and liquid regions in general classes of random lozenge tilings, see \cite{aggarwal2021edge}, also \cite{johansson2005arctic, ferrari2003step, okounkov2007random, petrov2014asymptotics} for earlier convergence results and work on related models. The extended Airy kernel has also been found at the more delicate liquid-gas boundary in the two-periodic Aztec diamond \cite{beffara2018airy, beffara2022local}.
	\end{enumerate}

While the determinantal representation of the Airy line ensemble is useful for the definition, it can be difficult to work with in practice and does not give good intuition for the probabilistic structure of Airy line ensemble. It is more useful to view $\cA$ as a system of infinitely many non-intersecting Brownian motions (i.e. a limit of (iii) above).

This idea was made rigorous by Corwin and Hammond \cite{CH} by identifying a \textbf{Brownian Gibbs property} for the Airy line ensemble. To describe this property, we work with the \textbf{parabolic Airy line ensemble} $\fA_1 > \fA_2 > \dots$, where $\fA_i(t) = \cA_i(t) - t^2$. The Brownian Gibbs property states that inside any region $K = \{1, \dots k\} \X [a, b]$, conditionally on all values $\fA_i(t)$ for $(i, t) \notin K$, the parabolic Airy line ensemble on $K$ is simply given by a sequence of $k$ independent Brownian bridges of variance $2$ from $(a, \fA_i(a))$ to $(b, \fA_i(b))$ conditioned so that the entire ensemble remains non-intersecting \footnote{Here we say that a Brownian bridge has variance $v$ if its quadratic variation over any interval $[c, d]$ equals $v(d-c)$.}

This property is extremely useful for understanding the Airy line ensemble. For example, the Brownian Gibbs property is necessary to rigorously show that the determinantal formula of Pr\"ahofer and Spohn defines a system of non-intersecting paths. The Brownian Gibbs property also implies that that recentered Airy lines $\fA_i - \fA_i(0)$ are locally absolutely continuous with respect to Wiener measure of variance $2$. This gives information about the sample paths of $\fA$ that is extremely difficult to access with the determinantal formula alone. 

Given that recentered Airy lines are locally absolutely continuous with respect to Wiener measure, the next natural question is to ask exactly how this occurs. More precisely, we would like to explicitly describe the density (i.e. the Radon-Nikodym derivative) of Airy lines against Wiener measure. This is the goal of the present paper.

To state our main result, for $k \in \N$ and $t > 0$, let $\sC^k_0([0, t])$ denote the space of $k$-tuples of continuous functions $f = (f_1, \dots, f_k), f_i:[0, t] \to \R$ with $f(0) = 0$, equipped with the topology of uniform convergence. Let $\nu_{k, t}$ be the law on $\sC^k_0([0, t])$ of 
\begin{equation}
\label{E:fA1r0}
\fA_i(r) - \fA_i(0),\quad r \in [0, t], i \in \II{1, k},
\end{equation} 
and let $\mu_{k, t}$ be the law on $\sC^k_0([0, t])$ of a $k$-tuple of independent Brownian motions of variance $2$. In \eqref{E:fA1r0} and throughout the paper we write $\II{\ell, k} := [\ell, k] \cap \Z$. Let $X_{k, t}$ be the density of $\nu_{k, t}$ against $\mu_{k, t}$.

\begin{theorem}
	\label{T:tetris-theorem}
	Fix $k \in \N, t \ge 1$. Then there exists a function $\S:\sC^k_0([0, t]) \to [0, \infty)$ such that for $\mu_{k, t}$-a.e.\ $f \in \sC^k_0([0, t])$ we have
	\begin{equation}
	\label{E:Xktf}
	X_{k, t}(f) = \exp (-\S(f) + O_k(t^3 + \sqrt{t} \S(f)^{5/6}) ).
	\end{equation}
\end{theorem}

In Theorem \ref{T:tetris-theorem} and throughout the paper, we write $O_k(g)$ for a term that is bounded in absolute value by $g$ times a $k$-dependent constant. 

The function $\S$ is given as follows. Let $\sC^k([0, t])$ denote the space of $k$-tuples of continuous functions from $[0, t] \to \R$ with the topology of uniform convergence (without a condition on $f(0)$). For $f \in \sC^k([0, t])$, let $\T f \in \sC^k([0, t])$ be the minimal $k$-tuple of functions such that 
$$
\T f_1(r) \ge \T f_2(r) \ge \dots \ge \T f_k(r) \ge 0
$$
for all $r \in [0, t]$ and such that $\T f - \T f (0) = f - f(0)$. Here and throughout we write $\T f_i = (\T f)_i, \T f(0) = (\T f)(0)$, etc. We think of $\T f$ as the ``Tetris of $f$": the result of consecutively dropping the graphs of the functions $f_k, f_{k-1}, \dots, f_1$ onto a flat line, and letting them stack up on top of each other as in the game of Tetris, see Figure \ref{fig:tetris}. Then
$$
\S (f) = \frac{2}{3} \Big(\sum_{i=1}^k [\T f_i(0)]^{3/2} + [\T f_i(t)]^{3/2} \Big).
$$
\begin{remark}[$t$-dependence in Theorem \ref{T:tetris-theorem}] We impose the restriction that $t \ge 1$  throughout the paper in order to simplify our expression of the error term. Naturally, Theorem \ref{T:tetris-theorem} immediately implies that similar bounds hold for $t < 1$ by pushing forward the measures $\mu_{k, 1}, \nu_{k, 1}$ from $\sC^k_0([0, 1])$ to $\sC^k_0([0, t])$ under the projection $f \mapsto f|_{[0, t]}$. Note that \eqref{E:Xktf} also implies estimates on intervals of the form $[a, a + t]$ by the stationarity of $\cA(t) = \fA(t) + t^2$.

The $t^3$ factor in the error is optimal as $t \to \infty$. Because of the stationarity of the Airy line ensemble $\cA(r) = \fA(r) + r^2$ the law $\nu_{k, t}$ concentrates around functions that are relatively close to the parabola $f(r) = -r^2$. On the other hand, the probability that a Brownian motion follows this parabola up to time $t$ is easily checked to be $e^{-\Theta(t^3)}$.
\end{remark}

\begin{figure}[htb]
	\centering
	\begin{subfigure}[t]{2.5cm}
		\centering
		\includegraphics[width=4cm]{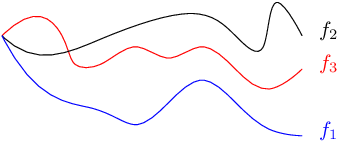}
		\caption{}
	\end{subfigure}
\qquad \qquad \qquad \qquad
	\begin{subfigure}[t]{2.5cm}
		\centering
		\includegraphics[width=4cm]{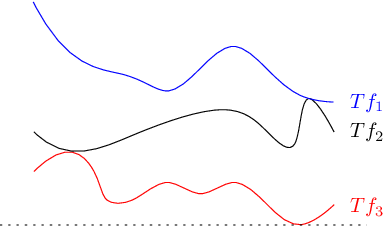}
		\caption{}
	\end{subfigure}
\hspace{5em}
	\caption{A function $f \in \sC^3_0([0, t])$ and its image under the map $\T$.}
	\label{fig:tetris}
\end{figure}

\subsection{More motivation and applications}
\label{S:quantitative}

The Brownian Gibbs property and the local Brownian structure of the Airy line ensemble have been repeatedly exploited to prove results about the Airy process, the Airy line ensemble, and more general limit objects in the KPZ universality class, e.g. see \cite{CH, hammond2017exponents} for some early examples and \cite{gangulyfractal} for a survey with more recent applications. In particular, the construction of the directed landscape \cite{DOV} relies heavily on Brownian Gibbs analysis from \cite{dauvergne2021bulk}. 

While it is quite useful to know that the Airy line ensemble is locally absolutely continuous with respect to Brownian motion, there are many situations when a more quantitative comparison is required. The search for a strong quantitative comparison was first undertaken by Hammond \cite{hammond2016brownian} and then continued by Calvert, Hegde, and Hammond \cite{calvert2019brownian}. The culmination of these two papers is the following estimate from \cite{calvert2019brownian}. For any $k \in \N$, $t > 0$, and any Borel set $A \sset \sC^1_0([0, t])$ we have
\begin{equation}
\label{E:CHH}
\frac{\P(\fA_k(\cdot) - \fA_k(0) \in A)}{\mu_{t}(A)} \le 2 \exp(d_{k, t} \log^{5/6}[\mu_t(A)^{-1}]),
\end{equation}
where $d_{k, t}$ is a $k, t$-dependent constant and $\mu_t = \mu_{1, t}$ is Wiener measure of variance $2$ on $\sC^1_0([0, t])$.

While \eqref{E:CHH} looks like a fairly technical statement, it has proven to be extremely useful. Indeed, since the growth in the right-hand side of \eqref{E:CHH} is slower than any power of $\mu_t(A)$, it shows that events for Brownian motion that occur with probability $\ep$ occur only with probability $\ep^{1-o(1)}$ for the $k$th Airy line. This is a key estimate for ruling out pathological behaviour in the Airy line ensemble, and has been used repeatedly for understanding limiting random geometry in the KPZ universality class, e.g. see \cite{bates2019hausdorff, ganguly2020stability, ganguly2022fractal, basu2019temporal, hammond2017exponents, hammond2020modulus, hammond2019patchwork, ganguly2020geometry, bhatia2022atypical, dauvergne2021disjoint, DSV}.

Theorem \ref{T:tetris-theorem} strengthens the estimate \eqref{E:CHH}. In fact, it implies something more surprising: namely, that the left-hand side of \eqref{E:CHH} is uniformly bounded over all $A$!

\begin{corollary}
	\label{C:bounded-above}
	For every $k \in \N, t \ge 1$, we have that $X_{k, t}(f) \le e^{O_k(t^3)}$ for $\mu_{k, t}$-a.e.\ $f \in \sC^k_0([0, t])$. In other words, $\nu_{k, t}(A) \le e^{O_k(t^3)} \mu_{k, t}(A)$ for any Borel set $A \sset \sC_0^k([0, t])$.
\end{corollary}

It is worth noting that our proof of Theorem \ref{T:tetris-theorem} and Corollary \ref{C:bounded-above} is shorter than the proof of \eqref{E:CHH} from \cite{hammond2016brownian, calvert2019brownian}, and hence gives an easier route to this cornerstone result and its consequences.

Corollary \ref{C:bounded-above} bounds how \textit{large } $\nu_{k, t}(A)$ can be compared to $\mu_{k, t}(A)$. We can also use Theorem \ref{T:tetris-theorem} to get a sharp estimate on how \textit{small } $\nu_{k, t}(A)$ can be compared to $\mu_{k, t}(A)$.
\begin{corollary}
	\label{C:X-bounded-below}
Define the function $\Theta:\R^k \to \R$ by
$$
\Theta(\bx) = \frac{2}{3} \lf(\sum_{i=1}^k |x_i| \rg)^{3/2} + \frac{4}{3} \sum_{j=2}^k \lf(\sum_{i=j}^k |x_i| \rg)^{3/2},
$$
and let $\al_k$ be the maximum of $\Theta$ on the sphere $\|\bx\|_2 = 1$.
Then
\begin{equation}
\label{E:theta-bound}
\inf \Bigg\{\frac{\nu_{k, t}(A)}{\mu_{k, t}(A)} : \mu_{k, t}(A) \ge \ep \Bigg\} = \exp \lf(-\al_k [4t \log(\ep^{-1})]^{3/4} + O_k(t^3 + t \log^{5/8} (\ep^{-1})) \rg).
\end{equation}
\end{corollary}
Note that the decay on the right-hand side of \eqref{E:theta-bound} is slower than any power of $\ep$. Because of this, we expect that Corollary \ref{C:X-bounded-below} could be useful for showing that rare behaviour for Brownian motion also occurs in the Airy line ensemble and the directed landscape, just as \eqref{E:CHH} has been used previously for ruling out certain behaviour in these objects.

For general $k$, the minimum value of $\Theta$ does not have a nice closed form. However, when $k = 1$, then the minimum value is $2/3$ (i.e. $\Theta(x) = 2/3$ when $x = \pm 1$). The proof of Corollary \ref{C:X-bounded-below} also gives a sense of what types of sets will come close to achieving the infimum in \eqref{E:theta-bound}.

\begin{remark}
	\label{R:RN-context}
To give some sense of the fragility and strength of the estimate in Theorem \ref{T:tetris-theorem}, it is interesting to note that Corollary \ref{C:bounded-above} fails if we replace $\fA$ by the stationary Airy line ensemble $\cA(t) = \fA(t) + t^2$ (though Corollary \ref{C:X-bounded-below} will still hold). See Proposition \ref{P:RN-context-example} for details.
\end{remark}

As another application, we can quickly deduce a large deviation principle for $\fA$ by using Corollaries \ref{C:bounded-above} and \ref{C:X-bounded-below} and Schilder's theorem. Note that this does not require the full strength of the above results, only that $X_{k, t}, X_{k, t}^{-1} \in L^p(\mu_{k, t})$ for all $p \in (1, \infty)$.
\begin{corollary}
	\label{C:schilder-Airy}
	For every $\ep > 0$, let $\nu_{k, t}^\ep$ denote the pushforward on $\sC_0^k([0, t])$ of $\nu_{k, t}$ under the map $f \mapsto \ep f$. Also, for $f \in \sC_0^k([0,t])$, define the rate function
	$$
	I(f) = \frac{1}{4} \int_0^t \sum_{i=1}^k |f_i'(s)|^2 ds,
	$$
	if $f$ is absolutely continuous, and $\infty$ otherwise. Then for every open set $G \sset \sC_0^k([0,t])$ we have
	$$
	\liminf_{\ep \cvgdown 0} \ep^2 \log \nu_{k, t}^\ep(G) \ge - \inf_{f \in G} I(f)
	$$
	and for every closed set $F \sset \sC_0^k([0,t])$ we have
	$$
	\limsup_{\ep \cvgdown 0} \ep^2 \log \nu_{k, t}^\ep(F) \le - \inf_{f \in F} I(f).
	$$ 
\end{corollary}

\subsection{A tool for Airy analysis}

While the Brownian Gibbs property is a powerful tool for analyzing the Airy line ensemble, it can nonetheless be difficult to analyze a collection of non-intersecting Brownian bridges on $\II{1, k} \X [0, t]$ conditioned to stay above a possibly arbitrary and unknown boundary. Indeed, at this point there is a fairly substantial bag of tricks designed precisely to handle these issues (e.g. monotonicity ideas \cite{CH}, the bridge representation \cite{dauvergne2021bulk}, the Wiener candidate and jump ensemble methods \cite{hammond2016brownian, calvert2019brownian}, tangent methods \cite{ganguly2022sharp}). 

One of the main contributions of this paper is to construct a line ensemble $\fL^{t, k}$ that is equal to the Airy line ensemble under a certain non-intersection conditioning, but is given by \textit{independent Brownian bridges} on the region $\II{1, k} \X [0, t]$.  The existence of the line ensemble $\fL^{t, k}$ implies Wiener density bounds that are similar to \eqref{E:CHH} and \eqref{C:bounded-above}, but crucially \textit{take into account the locations of the Airy endpoints}.

As our main theorem about the line ensemble $\fL^{t, k}$ has a somewhat technical statement, we first give an example of its use that only concerns a single Airy line.

\begin{example}
\label{Ex:one-line}
Fix $t > 0, k \in \N$. Then the law of process $\fA_k|_{[0, t]}$ has a bounded density against the law of a process $B + L$, where:
\begin{itemize}[nosep]
	\item $B:[0, t] \to \R$ is a Brownian bridge of variance $2$ with $B(0) = B(t) = 0$.
	\item $L:[0, t] \to \R$ is an affine function, independent of $B$, and there are $t, k$-dependent constants $c, d > 0$ such that for all $m > 0$ we have
	\begin{align}
	\label{E:one-pt-intro}
&\P(L(0) \vee L(t) > m) \le e^{- \frac{4}{3} m^{3/2} + c m^{5/4}}, \qquad \P(L(0) \wedge L(t) < - m) \le 2e^{-d m^3}, \\
\label{E:two-pt-intro}
\qquad &\P(|L(0) - L(t)| > m) \le e^{- \frac{m^2}{4t} - \frac{4k-2}{3} m^{3/2} + c m^{5/4} }.
	\end{align}
\end{itemize}
\end{example}

The one-point bounds on $L(0), L(t)$ in Example \ref{Ex:one-line} give the same stretched exponential tail behaviour as for the Tracy-Widom random variable $\fA(0)$. Moreover, the two-point bound on $L(0) - L(t)$ is actually stronger than the same bound for a variance $2$ Brownian motion! We now state the main theorem about the line ensemble $\fL^{t, k}$. 

\begin{theorem}
	\label{T:resampling-candidate}
Fix $t \ge 1, k \in \mathbb N$. Then there exists a random sequence of continuous functions $\fL = \fL^{t, k} = \{\fL^{t, k}_i : \R \to \R, i \in \N\}$ such that the following points hold:
\begin{enumerate}
	\item Almost surely, $\fL$ satisfies $\fL_i(r) > \fL_{i+1}(r)$ for all pairs $(i, r) \notin \II{1, k} \X (0, t)$.
	\item The ensemble $\fL$ has the following Gibbs property. For any set $S = \II{1, \ell} \X [a, b]$, conditional on the values of $\fL_i(r)$ for $(i, r) \notin S$, the distribution of $\fL_i(r), (i,r) \in S$ is given by $\ell$ independent Brownian bridges $B_1, \dots, B_k$ from $(a, \fL_i(a))$ to $(b, \fL_i(b))$ for $i \in \II{1, \ell}$, conditioned on the event $B_i(r) > B_{i+1}(r)$ whenever $(i, r) \notin \II{1, k} \X (0, t)$.
	
	In particular, the law of $\fL$ on $\II{1, k} \X [0, t]$ is simply given by $k$ independent Brownian bridges connecting $(0, \fL_i(0))$ to $(t, \fL_i(t))$.
	\item We have
	 $$
	 \P(\fL_1 > \fL_2 > \dots ) \ge e^{-O_k(t^3)},
	 $$
	 and conditional on the event $\{\fL_1 > \fL_2 > \dots\}$, the ensemble $\fL$ is equal in law to the parabolic Airy line ensemble $\fA$.
	\end{enumerate}
\end{theorem}

The process $L + B$ in Example \ref{Ex:one-line} is given by $\fL^{t, k}_k|_{[0, t]}$. The strong one- and two-point tail bounds in that example follow from more general tail bounds on 
$\fL^{t, k}$ proven in Sections \ref{S:tail-bounds} and \ref{S:upper-bound}. These tail bounds make the line ensemble
$\fL^{t, k}$ useful for analysis in practice.


\begin{remark}
	\label{R:27}
Theorem \ref{T:resampling-candidate} is a key input in the classification of geodesic networks in the directed landscape \cite{dauvergne2023geodesic}. In that work, we use Theorem \ref{T:resampling-candidate} to find an upper bound on the Hausdorff dimension of geodesic $k$-stars in the directed landscape.
\end{remark}

\begin{remark}
\label{R:multiple-patches}
Theorem \ref{T:resampling-candidate} allows us to relax the Brownian Gibbs property for $\fA$ by removing the non-intersection condition on the patch $\II{1, k} \X [0, t]$. We can also give a more general version of Theorem \ref{T:resampling-candidate} that removes the non-intersection condition on $\ell$ disjoint patches $\II{1, k} \X [a_1, a_1 + t], \dots, \II{1, k} \X [a_\ell, a_\ell + t]$ for a vector $\ba = (a_1, \dots, a_\ell)$. If we call the new line ensemble $\fL^\ba$, then 
$$
\P(\fL^\ba_1 > \fL^\ba_2 > \dots) \ge \exp(- e^{O_k (1 + \ell \log \ell)} t^3).
$$
Crucially, this bound does not depend on $\ba$. Moreover, the patches $\fL^\ba|_{\II{1, k} \X [a_i, a_i + t]}$ become asymptotically independent as $\min |a_i - a_{i-1}| \to \infty$. The line ensembles $\fL^\ba$ are useful for studying the Airy process and the Airy line ensemble on large spatial scales. See Theorem \ref{T:resampling-candidate-multiple} for details.
\end{remark}

\subsection{A heuristic for \eqref{E:Xktf}: Brownian motion avoiding a parabola}
	\label{SS:proof-heuristic}
Before giving a brief overview of the proof in Section \ref{SS:outline}, we give a heuristic for formula \eqref{E:Xktf}. For simplicity, we will only consider the case when $k = 1$.
A good starting point for trying to understand the behaviour of the top curve $\fA_1$ is the well-known pair of tail estimates
\begin{equation}
\label{E:Asm-bds}
\P(\fA_1(s) + s^2 > m) \sim e^{- (\frac{4}{3} + o(1)) m^{3/2}} \qquad \text{ and } \qquad \P(\fA_2(s) + s^2 < - m) \sim e^{- c m^{3}}.
\end{equation}
Similar estimates hold for all Airy lines, see Lemmas \ref{L:one-point-upper-bound}, \ref{L:one-point-lower-bound}.
The intuition from these estimates suggests first that $\fA_1$ typically adheres to the parabola $-s^2$, and secondly that it is much easier for $\fA_1$ to move up away from this parabola than it is for $\fA_1$ to push $\fA_2$ down below the parabola. Because of this, a toy model for $\fA_1$ can be given as follows. Consider a sequence of Brownian bridges $B_n$ from $(-n, -n^2)$ to $(n, -n^2)$ conditioned to stay above the parabola $f(s) = -s^2$, and take $n$ to $\infty$ to get a distributional limit $\cB$. The process $\cB$ has been studied in \cite{ferrari2005constrained} and can actually be described as a diffusion with an explicit drift involving Airy functions. However, for this discussion we will stick to softer intuition regarding $\cB$.

\begin{figure}[htb]
	\centering
		\includegraphics[width=10cm]{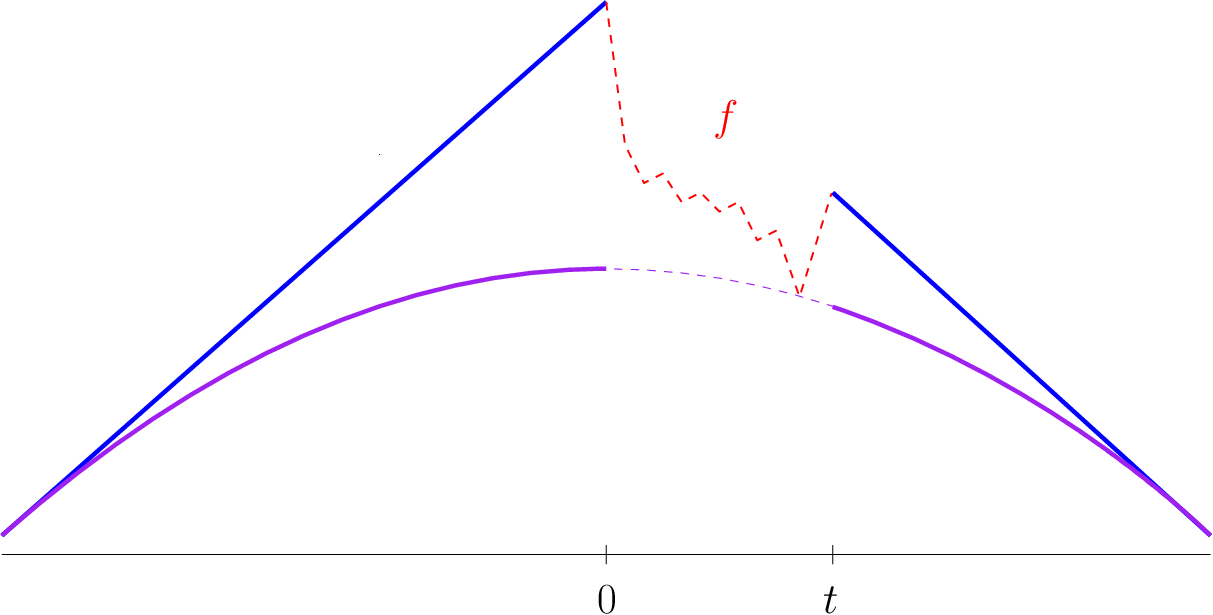}
	\caption{The easiest way for $\cB(r) - \cB(0), r \in [0, t]$ to equal a function $f \in \sC^1_0([0, t])$. The cost to lift $\cB$ far enough above of the parabola to achieve the function $f$ is the difference between the Dirichlet energies of the solid blue and purple curves. This energy difference is $\S(f)$ plus a lower order term in $t$.}
	\label{fig:parabola-avoidance}
\end{figure}

One way to check the reasonability of this toy model is to check the probability that $\cB(0) \sim m$. The best way for $\cB(0)$ to equal $m$ should be if $\cB$ only deviates from the parabola by following straight lines from $(-\la, -\la^2)$ to $(0, m)$ and back to $(\la, -\la^2)$ for some $\la > 0$. Computing the Dirichlet energy of this path against the Dirichlet energy of the parabola, and optimizing over $\la$ yields that our best guess  
for $\P(B(0) \sim m)$ should be $e^{-(\frac{4}{3} + o(1)) m^{3/2}}$, which agrees with the Tracy-Widom upper tail on $\fA_1(0)$. The optimal value is $\la = \sqrt{m}$; with this choice the straight lines from $(-\la, -\la^2)$ and $(\la, -\la^2)$ to $(0, m)$ are \textit{tangent} to the parabola.

Next, for $\cB(r) - \cB(0), r \in [0, t]$ to equal a function $f \in \sC^1_0([0, t])$, we must have $\cB(0) + f(r) \ge - r^2$ for all $r \in [0, t]$. From the definition of the tetris map $\T$, we can observe that
$$
\cB(0) \ge \T f (0) \quad \implies \quad \cB(0) + f(r) \ge - r^2 \; \forall r \in [0, t] \quad \implies \cB(0) \ge \T f (0) - t^2.
$$
Since it is costlier to raise $\cB$ up further away from the parabola, this implies that the optimal way to have $\cB(r) - \cB(0), r \in [0, t]$ equal to $f$ is to let $\cB(0)$ sit in the interval $[ \T f(0) - t^2,  \T f(0)]$ and similarly let $\cB(t)$ sit in the interval $[ \T f(t) - t^2,  \T f(t)]$. The cost of this, relative to the cost of simply having a Brownian motion equal to $f$ can again be computed by comparing the Dirichlet energy of straight lines against a parabola. Assuming $t^2$ is very small relative to the size of $\T f$, this amounts to computing the Dirichlet energy of the straight line from $(-\sqrt{\T f(0)}, - \T f(0))$ to $(0, \T f(0))$  plus the Dirichlet energy of the straight line from $(t, \T f(t))$ to $(\sqrt{\T f(t)}, - \T f(t))$ minus the Dirichlet energy of the parabola from $(-\sqrt{\T f(0)}, - \T f(0))$ to $(\sqrt{\T f(t)}, - \T f(t))$. This is
$$
\frac{2}{3} ([\T f(0)]^{3/2} + [\T f(t)]^{3/2}) + O(t \T f(t)),
$$
which is $\S(f) + O(t \T f(t))$. The assumption that $t^2$ is small relative to the size of $\T f$ guarantees that the error term here is lower order. See Figure \ref{fig:parabola-avoidance} for an illustration. The intuition behind the expression \eqref{E:Xktf} for more than one line is similar, except now we need to stack the $k$ functions $f_1, \dots, f_k$ on top of each other, \textit{and} on top of the parabola in order to achieve the ordering $\fA_1 > \fA_2 > \dots$.

Understanding probabilistic objects built from non-crossing paths via a tangent picture is not a new idea. Indeed, tangent methods were proposed by Colomo and Sportiello \cite{colomo2016arctic} for studying limit shapes of the six-vertex model, and made rigorous by Aggarwal \cite{aggarwal2020arctic}. In the context of the KPZ and Airy line ensembles, tangent methods were used by Ganguly and Hegde \cite{ganguly2022sharp} to pinpoint the upper tail one-point asymptotics in \eqref{E:Asm-bds} and related upper tail two-point asymptotics using only minimal assumptions.

\subsection{Outline of the proof and the paper}
\label{SS:outline}
A priori, it is not clear why the hydrodynamic tangent heuristic should give fine control over the Wiener density. 
As a result, our proof of the upper bound in Theorem \ref{T:tetris-theorem} does not proceed by directly trying to prove this heuristic. 
Rather, we first prove Theorem \ref{T:resampling-candidate}, which implies that the only contribution to the Wiener density $X_{k, t}(f)$ comes from the joint distribution of the endpoints $\fL(0), \fL(t)$ (here $\fL$ is as in Theorem \ref{T:resampling-candidate}). To prove Theorem \ref{T:resampling-candidate}, by the Brownian Gibbs property we can abstractly define a $\sigma$-finite measure $\mu_\fL$ on the space $\sC^\N(\R)$ of sequences of continuous functions $f = \{f_i : \R \to \R, i \in \N\}$ such that if $\mu_\fL(\sC^\N(\R))$ were finite, then after scaling it to a probability measure, it would give the law of a line ensemble $\fL$ satisfying Theorem \ref{T:resampling-candidate}. The key difficulty is in showing that $\mu_\fL(\sC^\N(\R)) < \infty$: here we use a refinement of some ideas from \cite{dauvergne2021bulk} and \cite{hammond2016brownian}, along with new insights about how the Brownian Gibbs property on different domains controls $\mu_\fL(\sC^\N(\R))$. The proof of Theorem \ref{T:resampling-candidate} is contained in Section \ref{S:construction-L}.

In Section \ref{S:tail-bounds}, we analyze the joint distribution of $\fL(0), \fL(t)$ by again appealing to one-point bounds, using monotonicity arguments, and applying an induction argument to pass from control of $(\fL_i(0), \fL_i(t))$ for any fixed $i$ to joint control of $(\fL_i(0), \fL_i(t)), i \in \II{1, k}$. It is only because the upper tail one-point bound in \eqref{E:Asm-bds} sharply matches the bound coming from the tangent heuristic that the tangent picture emerges. In Section \ref{S:upper-bound}, we combine these tail bounds with Theorem \ref{T:resampling-candidate} to establish that LHS$\le$RHS in \eqref{E:Xktf}.

 The lower bound LHS$\ge$RHS is easier, since we only need to work on a part of the probability space where the $(k+1)$st line does not deviate too far from the parabola. It closely resembles a rigorous version of the heuristic discussed in Section \ref{SS:proof-heuristic}, and is shown in Section \ref{S:lower-bound}. A final section (Section \ref{S:proofs-corollaries}) proves the corollaries.
 
 We have striven to keep the paper self-contained, relying only on standard tail bounds on the Airy points $\fA_1(0), \fA_2(0), \dots$, the Brownian Gibbs property, and basic monotonicity properties for non-intersecting Brownian bridges as inputs.
%
%
%

%
%
%
%


\subsection*{Conventions throughout the paper}
In all sections of the paper after Section \ref{S:prelim}, $t \ge 1$ and $k \in \N$ are fixed. We will often suppress dependence of objects on $k$ and $t$ for notational convenience (e.g. we write $\fL$ instead of $\fL^{k, t}$). We write $c, c', c'' > 0$ for large constants and $d, d' > 0$ for small constants. Our constants may change from line to line within proofs and we use a subscript $k$ (i.e. $c_k$) if they depend on $k$. Constants will not depend on any other parameters. 
All Brownian objects in the paper (e.g. Brownian bridge, Wiener measure) will have variance $2$, as is the convention for the Airy line ensemble. For a set $I \sset \R$, we write $I^k_>$ (respectively $I^k_\ge$) for all $\bx \in I^k$ with $x_1 > \dots > x_k$ (respectively $x_1 \ge \dots \ge x_k$). For a sequence of functions $f = (f_1, \dots, f_k)$ and $\ell \le k$, we write $f^\ell = (f_1, \dots, f_\ell)$ for the first $\ell$ coordinates of $f$. Finally, we use the shorthand $\E_\cF(\cdot) = \E(\cdot \mid \cF)$ for conditional expectation with respect to a $\sig$-algebra $\cF$.

	\section{Preliminaries}
	\label{S:prelim}
	
	In this section, we introduce a few basic preliminaries about the Airy line ensemble and other non-intersecting collections of Brownian motions. 
	
	Let $I$ be a closed interval in $\R$, let $k \in \N$, and define $\sC^k(I)$ denote the space of $k$-tuples of continuous functions $f = (f_1, \dots, f_k)$ from $I$ to $\R$ with the \textbf{compact topology} 
	$$
	f_n \to f \quad \text{ if } \quad \sup_{x \in K} |f_{n, i}(x) - f_i(x)| \to 0 \quad \text{   for all compact sets } K \sset I, i \in \II{1, k}.
	$$
	We also let $\sC^\N(I)$ be the space of all sequences of continuous functions $f = (f_1, f_2, \dots)$ from $I$ to $\R$, again equipped the with compact topology (we replace $\II{1, k}$ with $\N$ in the definition above), and we write $\sC(I) = \sC^1(I)$.
	
	We can also think of $f \in \sC^k(I)$ as a function $f:\II{1, k} \X I \to \R$ with the topology of uniform convergence on compact subsets of $\II{1, k} \X I$, and with this in mind, we will often write $f|_U$ for the restriction of $f$ to a subset of $\II{1, k} \X I$. For $f \in \sC^k(I)$ with $\ell \le k$ we write $f^\ell = (f_1, \dots, f_\ell)$ for the first $\ell$ coordinates of $f$.

	The \textbf{parabolic Airy line ensemble} $\fA$ is a random element of $\sC^\N(\R)$ (a.k.a.\ a line ensemble) satisfying 
	$$
	\fA_1 > \fA_2 > \dots
	$$
	almost surely. The process $\cA_i(t) = \fA_i(t) + t^2$ is the usual \textbf{Airy line ensemble} and is stationary under shifts and time reversal. For any $r \in \R$ we have
	\begin{equation}
	\label{E:cAs}
	\cA(\cdot) \eqd \cA(- \; \cdot)\eqd \cA(\cdot + r),
	\end{equation}
	where the equality in distribution is as random element of $\sC^\N(\R)$.
	Throughout the paper we will typically use the parabolic version $\fA$, which has a more natural Gibbs resampling property. The Airy line ensemble $\cA$ is uniquely characterized by a determinantal formula. Indeed, for every finite set of times $\{t_1, \dots, t_k\} \sset \R$, the set of locations set $\{(\cA_i(t_j), t_j) : i \in \N, j \in \II{1, k}\}$ is a determinantal point process on $\R \X \{t_1, \dots, t_k\}$ with with kernel
	\begin{equation}
	\label{E:airy-kernel}
	K((x, s) ; (y, r)) =
	\begin{cases}
	\int_{0}^\infty d \lambda e^{-\lambda(s - r)}\Ai(x + \la)\Ai(y + \la), \qquad & \mathif s \ge t, \\
	-\int_{-\infty}^0 d \lambda e^{-\lambda(s -r)}\Ai(x + \la)\Ai(y + \la), \qquad & \mathif s < r,
	\end{cases}
	\end{equation}
	where $\Ai(\cdot)$ is the Airy function. 
	
	\subsection{Tail bounds}
	The Airy points $\fA_i(0), i \in \N$ satisfy well-known tail bounds. A fairly sharp upper tail bound is a quick consequence of the determinantal formula.
	
	\begin{lemma}
		\label{L:one-point-upper-bound}
		For all $k \in \N$ there exists $c_k > 0$ such that for $\mathbf{m} \in [0, \infty)^k_\ge$  we have
		$$
		\P(\fA_i(0) > m_i \quad \text{ for all } i \in \II{1, k}) \le c_k \exp\Big(-\frac{4}{3} \sum_{i=1}^k m_i^{3/2}\Big).
		$$
	\end{lemma}
	
	\begin{proof}
		This is a simple calculation involving the Airy kernel at a single time. Let $\Lambda = \{\fA_i(0), i \in \N\}.$ Then for $\mathbf{m} \in \R^k_>$ and $\ep > 0$ we have
		$$
		\P(\fA_i(0) \in [m_i- \ep, m_i + \ep] \text{ for all } i \in \II{1, k}) \le \E \lf(\prod_{i=1}^k \#(\Lambda \cap [m_i- \ep, m_i + \ep]) \rg).
		$$
		As long as $\ep > 0$ is small enough, the right hand side above equals
		$$
		\int_{m_k - \ep}^{m_k + \ep} \cdots \int_{m_1 - \ep}^{m_1 + \ep}\det_{1 \le i, j \le n} K(x_i,0; x_j, 0) dx_1 \dots d x_k,
		$$ 
		where $K$ is the Airy kernel. Taking $\ep \to 0$, we see that $(\fA_1(0), \dots, \fA_k(0))$ has a Lebesgue density that at $\bx \in \R^k_\ge$ is bounded above by
		$$
		\det_{1 \le i, j \le n} K(x_i, 0;x_j, 0).
		$$ 
		Using the formula \eqref{E:airy-kernel} and the standard bound $\operatorname{Ai}(x) \le c \exp(-\tfrac{2}{3} x^{3/2})$ for a universal constant $c$ (see formula 10.4.59 in \cite{abramowitz1972handbook}), we have that
		$$
		K((x, 0), (y, 0)) \le c' \exp(-\tfrac{2}{3} (x^{3/2} +y^{3/2})).
		$$
		and so
		$$
		\det_{1 \le i, j \le n} K(x_i, 0;x_j, 0) \le k! c' \exp(-\tfrac{4}{3} \sum_{i=1}^k x_i^{3/2}).
		$$ 
		Integrating this over the set of all $\bx \in \R^k_>$ with $x_i \ge m_i$ for all $i$ yields the result.
	\end{proof}

\begin{lemma}
	\label{L:one-point-lower-bound}
	For every $k \in \N$ there exists $d_k > 0$ such that for all $m > 0$ we have
	\begin{equation}
	\label{E:cAk0}
	\P(\cA_k(0) < - m) \le 2 \exp(-d_k m^{3}).
	\end{equation}

\end{lemma}

\begin{proof}
	For $k = 1$ this is well-known tail bound goes back to \cite{tracy1994level}. For $k > 1$, this is a limiting version of Theorem 3.1 of \cite{dauvergne2021bulk}, equation (5). More precisely, that theorem states that for a sequence of random variables $\cA^n_k(0) \cvgd \cA_k(0)$ as $n \to \infty$, that \eqref{E:cAk0} holds for all large enough $n$.
\end{proof}

Later on in the paper, we will need a quick way to move between two-point tail bounds and a maximum bound on a stochastic process. This is accomplished through the following lemma.

\begin{lemma}
	\label{L:mod-cont}
	Let $t > 0$, $W \in \sC([0, t])$ and suppose that there exist constants $\al, \be > 0$ such that for all $m > 0$ and all $s, s + r \in [0, t]$, we have the estimate
	$$
	\P(|W(s) - W(s+r)| > m \sqrt{r}) \le \al \exp(-\be m^2).
	$$
	Then for $m > 20 \sqrt{\beta}$ we have
	$$
	\P(\sup_{s, s + r \in [0, t]} |W(s) - W(s+r)| \ge m \sqrt{t}) \le 2 \al \exp(-\beta m^2/64).
	$$
\end{lemma}

\begin{proof}[Proof of Lemma \ref{L:mod-cont}]
	The proof is a standard chaining argument. We may assume $W(0) = 0$, and focus on bounding $2\|W\|_\infty$.
	For every $k =0, 1, \dots$, let $G_k$ be the function with $G_k(s) = W(s)$ for $s = it/(2^k)$ for some $i \in \II{0, 2^k}$, and such that $G_k$ is linear on each of the intervals $[it/(2^k), (i+1)t/(2^k)]$. For $k = 0, 1, \dots$, let $H_k = G_k - G_{k-1}$, where $G_{-1}:= 0$. We can write $W = \sum_{k=0}^\infty H_k$. Also, 
	$$
	\|H_k\|_\infty \le \sup_{i \in \II{1, 2^k}} |W(it2^{-k}) - W((i-1)t 2^{-k})|.
	$$ 
	and hence we have the union bound
	$$
	\P(\|H_k\|_\infty > m \sqrt{t} 2^{-k/2} ) \le \sum_{i=1}^{2^k} \P(|W(it2^{-k}) - W((i-1)t 2^{-k})| > m \sqrt{t} 2^{-k/2}) \le 2^k \al \exp(-\be m^2). 
	$$
	Therefore letting
	$$
	M =  \max_{k \ge 0} (t^{-1/2} 2^{k/2}\|H_k\|_\infty - \sqrt{k/\beta}),
	$$
	we have
	$
	\P(M > m) \le 2\al \exp(-\be m^2).
	$
	Next, 
	\begin{align*}
	2\|W\|_\infty \le 2\sum_{k=0}^\infty \|H_k\|_\infty \le 2\sqrt{t} \sum_{k=0}^\infty (M + \sqrt{k/\beta}) 2^{-k/2} \le \sqrt{t} (4M + 10/\sqrt{\beta}),
	\end{align*}
	which yields the result after converting the tail bound on $M$.
\end{proof}
\subsection{Gibbs resampling}
\label{SS:gibbs-resampling}

As discussed in the introduction, the Airy line ensemble satisfies a resampling property called the Brownian Gibbs property. Here we introduce notation that will help us work with the Brownian Gibbs property, and state a few useful lemmas.

First, for an interval $I \sset \R$, an interval $J \sset I$ and an $\R \cup \{-\infty\}$-valued function $g$ whose domain contains $J$, let
$$
\NI(g, J) = \{f \in \sC^k(I) : f_1(s) > f_2(s) > \dots > f_k(s) > g(s) \qquad \text{ for all } s \in J\}.
$$
Note that the value of $k$ and the interval $I$ is not explicit in the notation $\NI(g, J)$. These will always be clear from context (i.e. by specifying the initial space $\sC^k(I)$). We will use the shorthand $\NI(g) = \NI(g, I)$. 
Next, we say that a $k$-tuple of \textbf{independent Brownian bridges} $B = (B_1, \dots, B_k) \in \sC^k([s, t])$ goes from $(s, \bx)$ to $(t, \by)$ for $\bx, \by \in \R^k$ if each of the $B_i$ is a Brownian bridge with $B_i(s) = x, B_i(t) = y$. Note that our bridges will always have variance $2$. We write
\begin{equation*}
\P_{s, t}(\bx, \by, f, J)
\end{equation*}
for the probability that a $k$-tuple of independent Brownian bridges from $(s, \bx)$ to $(t, \by)$ is in the non-intersection set $\NI(f, J)$. Again, we omit $J$ from all notation if $J = [s, t]$. We say that the $k$-tuple of Brownian bridges is \textbf{non-intersecting} if it is conditioned on the event $\NI(-\infty)$.

The Brownian Gibbs property for $\fA$ can now be more formally stated as follows:
\begin{itemize}
	\item[] For a set $S =\II{1, k} \X [a, b]$, let $\cF_{k, a, b}$ denote the $\sig$-algebra generated by $\{\fA_i(x) : (i, x) \notin S\}$.  Then the conditional distribution of $\fA|_S$ given $\cF_{k, a, b}$ is equal to the law of $k$ independent Brownian bridges from $(a, \fA^k(a))$ to $(b, \fA^k(b))$ conditioned on the event $\NI(\fA_{k+1})$.
\end{itemize}
In the statement above, recall that $\fA^k(a) = (\fA_1(a), \dots, \fA_k(a))$ and note that when applying the Gibbs property, the bridges are always drawn independently of the endpoint vectors $\fA^k(a), \fA^k(b)$ and of the lower boundary $\fA_{k+1}$.

We record two key lemmas that will allow us to work with the Brownian Gibbs property. The first is a crucial monotonicity property.

\begin{lemma}
	\label{L:CH-monotonicity}
Let $[s, t], J$ be closed intervals in $\R$ with $J \sset [s, t]$, let $\bx^1 \le \bx^2, \by^1 \le \by^2 \in \R^k_>$ where $\le$ is the coordinate-wise partial order, and let $g_1 \le g_2$ be bounded Borel measurable functions from $[s,t] \to \R \cup \{-\infty\}$. For $i = 1, 2$ be $B^i$ be a $k$-tuple of Brownian bridges from $(s, \bx^i)$ to $(t, \by^i)$, conditioned on the event $\NI(g_i, J)$. Then there exists a coupling such that $B^1_j(r) \le B^2_j(r)$ for all $r \in [s, t], j \in \II{1, k}$.
\end{lemma}

Lemma \ref{L:CH-monotonicity} is proven as Lemmas 2.6 and 2.7 in \cite{CH} in the case when $[s, t] = J$. The proof goes through verbatim in the slight extension when $J$ is a strict subset of $[s, t]$.

The next lemma is a lower bound on the probability of Brownian bridges being non-intersecting. This bound will allow us to compare regular Brownian bridge probabilities with non-intersecting bridge probabilities.

\begin{lemma}
	\label{L:nonint-probability}
	Let $\bx, \by \in \R^k_>$ be vectors with 
	$$
	\min_{i \in \II{1, k-1}} (x_i - x_{i+1}) \wedge (y_i - y_{i+1}) \ge \al
	$$
	for some $\al > 0$. Let $t > 0$. Then setting $\ep = \al^2/t$ we have
	$$
	\P_{0, t}(\bx, \by, - \infty) \ge e^{- k^3 \ep/6}(\ep/2)^{k(k-1)/2}.
	$$
\end{lemma}

This lemma is fairly standard and follows from the Karlin-McGregor formula. For example, it is very similar to Proposition 3.1.1 in \cite{hammond2016brownian}, and we follow the same steps used in the proof of that proposition.

\begin{proof}
First, $\P_{0, t}(\bx, \by, - \infty) \ge \P_{0, t}(\al \iota, \al \iota, -\infty)$ where $\iota = (k, k-1, \dots, 1)$. Therefore letting $\bx = \al \iota$, by the Karlin-McGregor formula, after simplification we can write
	\begin{align*}
	P_{0, t}(\al \iota, \al \iota, -\infty) &= \exp(- \frac{1}{2t} \sum_{i=1}^k x_i^2) \det (e^{x_i x_j/(2t)})_{1 \le i, j \le k}.
	\end{align*}
	See the middle of p. 48 in \cite{hammond2016brownian} for this version of the Karlin-McGregor formula. The determinant is a Vandermonde, and satisfies
	\begin{equation*}
	\det (e^{x_i x_j t^{-1}})_{1 \le i, j \le k} = \prod_{1 \le i < j \le k} (e^{\al^2 j (2t)^{-1}} - e^{\al^2 i (2t)^{-1}}) \ge \prod_{1 \le i < j \le k} \al^2 (2t)^{-1}(j-i) \ge (\al^2 (2t)^{-1})^{k(k-1)/2}. 
	\end{equation*}
	Finally
	$$
	\exp(- (2t)^{-1} \sum_{i=1}^k x_i^2) \ge \exp(- k^3 \al^2 t^{-1}/6),
	$$
	yielding the result.
\end{proof}

	\section{Constructing $\fL^{t, k}$}
	\label{S:construction-L}
	
In this section, we construct the line ensemble $\fL = \fL^{t, k}$ and prove Theorem \ref{T:resampling-candidate}. Throughout the section, we fix $t \ge 1, k \in \N$.

Let $\fA$ be the parabolic Airy line ensemble. We will define a $\sig$-finite measure $\eta = \eta^{t, k}$ on $\sC^\N(\R)$ according to the following procedure. 
\begin{enumerate}
	\item First, define a line ensemble $\fB$ by letting $\fB_i(r) = \fA_i(r)$ for $(i, r) \notin \II{1, k} \X [0, t]$, and let $\fB^k|_{[0, t]}$ be given by $k$ independent Brownian bridges from $(0, \fA^k(0))$ to $(t, \fA^k(t))$. These bridges are independent of each other and of $\fA$. Let $\tilde \eta$ denote the law of $\fB$.
	\item Let $\eta$ be the measure which is absolutely continuous with respect to $\tilde \eta$ with density given by
	\begin{equation}
	\label{E:eta-defn}
	\frac{d \eta}{d \tilde \eta}(f) = \frac{1}{\P_{0, t}(f^k(0), f^k(t), f_{k+1})}
	\end{equation}
	for $f \in \sC^\N(\R)$. This density exists for $\tilde \eta$-a.e.\ $f$ since $\fB^k(0) = \fA^k(0)$, $\fB^k(t) = \fA^k(t)$, $\fB_{k+1} = \fA_{k+1}$ and almost surely, $\fA_1(0) > \dots > \fA_{k+1}(0)$ and $\fA_1(t) > \dots > \fA_{k+1}(t)$.
\end{enumerate}
We would like to say that $\eta$ is a finite measure, and hence can be normalized to a probability measure. This normalized measure will be the law of the line ensemble $\fL = \fL^{t, k}$. Once we know this, Theorem \ref{T:resampling-candidate} will follow easily. However, it is highly non-trivial to show that $\eta$ is finite. The bulk of this section is devoted to proving this. Indeed, by \eqref{E:eta-defn} we can observe that
\begin{equation}
\label{E:CNR}
\eta(\sC^\N(\R)) = \E \lf(\frac{1}{\P_{0, t}(\fA^k(0), \fA^k(t), \fA_{k+1})}\rg)
\end{equation}
and so concretely we just need to show that the \textbf{inverse acceptance probability} 
$$
[\P_{0, t}(\fA^k(0), \fA^k(t), \fA_{k+1})]^{-1}
$$has finite expectation.
In order to do this, we will study alternate constructions of the measure $\eta$. Fix $s > t$, and define a measure $\eta_s = \eta_s^{t, k}$ by the following procedure:  
\begin{enumerate}
	\item Define a line ensemble $\fB_s$ by letting $\fB_{s,i}(r) = \fA_i(r)$ for $(i, r) \notin \II{1, k} \X [-s, s]$, and let $\fB^k_s|_{[-s, s]}$ have the conditional law given $\fA$ outside of $\II{1, k} \X[-s, s]$ of $k$ independent Brownian bridges from $(-s, \fA^k(-s))$ to $(s, \fA^k(s))$ conditioned on the event
	$$
\NI(\fA_{k+1}, [-s, 0] \cup [t, s]).
	$$
	Let $\tilde \eta_s$ denote the law of $\fB_s$.
	\item Let $\eta_s$ be the measure which is absolutely continuous with respect to $\tilde \eta_s$ with density given by
	\begin{equation}
	\label{E:eta-s-defn}
	\frac{d \eta_s}{d \tilde \eta_s}(f) = \frac{\P_{-s, s}(f^k(-s), f^k(s), f_{k+1}, [-s, 0] \cup [t, s])}{\P_{-s, s}(f^k(-s), f^k(s), f_{k+1})}
	\end{equation}
	for $f \in \sC^\N(\R)$.
\end{enumerate}

\begin{lemma}
	\label{L:eta-eta-s}
	For every $s > t$, we have $\eta = \eta_s$, and letting $\cF_s$ be the $\sig$-algebra generated by the Airy line ensemble $\fA$ outside of the set $\II{1, k} \X [-s, s]$, we have
	\begin{equation}
	\label{E:big-eqn}
		\E_{\cF_s} \lf(\frac{1}{\P_{0, t}(\fA^k(0), \fA^k(t), \fA_{k+1})} \rg) = \lf(\frac{1}{\E_{\cF_s} [\P_{0, t}(\fB^k_s(0), \fB^k_s(t), \fA_{k+1})]} \rg).
	\end{equation}
\end{lemma}

Equation \eqref{E:big-eqn} exchanges the expectation of an inverse for the inverse of an expectation, which is much easier to bound. 

The idea that the inverse acceptance probability $[\P_{0, t}(\fA^k(0), \fA^k(t), \fA_{k+1})]^{-1}$ can be controlled by an inverse of a conditional expectation by `stepping back' onto the interval $[-s, s]$ and applying the Brownian Gibbs property on $[-s, s]$ is inspired by the Wiener candidate method of Brownian Gibbs resampling introduced by Hammond \cite{hammond2016brownian}, see also \cite{calvert2019brownian}. 

\begin{proof}
We first compute the density of $\fB$ against $\fB_s$ (more precisely, the density of $\tilde \eta_s$ against $\tilde \eta$). First, letting $U = \II{1, k} \X [-s, s]$ we have that $\fB = \fB_s = \fA$ outside of $U$. Therefore the density of $\fB$ against $\fB_s$ is the same as the density of $\fB|_{U}$ against $\fB_s|_{U}$.

 Conditional on $\cF_s$, all of $\fB_s|_{U}, \fB|_{U},$ and $\fA|_U$ have densities with respect to the law of $k$ independent Brownian bridges from $(-s, \fA^k(-s))$ to $(s, \fA^k(s))$. The density of $\fB_s|_U$ is given by
\begin{equation}
\label{E:indic-fNI}
\frac{\indic(f \in \NI(\fA_{k+1}, [-s, 0] \cup [t, s]))}{\P_{-s, s}(f(-s), f(s), \fA_{k+1}, [-s, 0] \cup [t, s])}
\end{equation}
for $f \in \sC^k([-s, s])$. The density of $\fA|_U$ is 
\begin{equation}
\label{E:indic-f}
\frac{\indic(f \in \NI(\fA_{k+1}, [-s, s]))}{\P_{-s, s}(f(-s), f(s), \fA_{k+1})}
\end{equation}
for $f \in \sC^k([-s, s])$. To find the density of $\fB|_U$, first observe that this density only depends on $f|_{[-s, 0] \cup [t, s]}$, since conditional on $\fB^k|_{[-s, 0] \cup [t, s]}$, the law of $\fB^k|_{[0, t]}$ is that of $k$ independent Brownian bridges. We obtain $\fB$ from $\fA$ by replacing $\fA$ with independent Brownian bridges on $[0, t]$ from $(0, \fA^k(0))$ to $(t, \fA^k(t))$, and so the density of $\fB|_U$ is zero unless $f \in \NI(\fA_{k+1}, [-s, 0] \cup [t, s])$. Moreover, given any such $f$ we can find $\tilde f \in \NI(\fA_{k+1}, [-s, s])$ with $f|_{[-s, 0] \cup [t, s]} = \tilde f|_{[-s, 0] \cup [t, s]}$. The density of $\fB|_U$ at $f$ equals the density of $\fB|_U$ at $\tilde f$.

To find the density of $\fB|_U$ at $\tilde f$, observe that the density of $\fA$ on $\II{1, k} \X [0, t]$ with respect to independent Brownian bridges between these points is given (again by the Brownian Gibbs property) by
\begin{equation}
\label{E:indic-g}
\frac{\indic(g \in \NI(\fA_{k+1}, [0, t]))}{\P_{0, t}(g(0), g(t), \fA_{k+1})}
\end{equation}
for $g \in \sC^k([0, t])$. Therefore the density of $\fB|_U$ at $\tilde f$ (and hence $f$) is simply the product of \eqref{E:indic-f} with the inverse of \eqref{E:indic-g} at $g = \tilde f|_{[0, t]}$, which is non-zero since $\tilde f \in \NI(\fA_{k+1}, [-s, s])$. Putting everything together, we get that the density of $\fB$ at a general $f \in \sC^k([-s, s])$ equals 
\begin{equation}
\label{E:indic-fg}
\frac{\indic(f \in \NI(\fA_{k+1}, [-s, 0] \cup [t, s]) \P_{0, t}(f(0), f(t), \fA_{k+1})}{\P_{-s, s}(f(-s), f(s), \fA_{k+1})}.
\end{equation}
The ratio of \eqref{E:indic-fg} and \eqref{E:indic-fNI} gives the density of $\fB|_U$ against $\fB_s|_U$, conditional on $\cF_s$:
\begin{equation}
\label{E:Pssfss}
\lf(\frac{\P_{-s, s}(\fA^k(-s), \fA^k(s), \fA_{k+1}, [-s, 0] \cup [t, s])}{\P_{-s, s}(\fA^k(-s), \fA^k(s), \fA_{k+1})} \rg) \X \P_{0, t}(f(0), f(t), \fA_{k+1}),
\end{equation}
for $f \in \sC^k([-s, s])$.
Using this, and the definitions \eqref{E:eta-defn}, \eqref{E:eta-s-defn} of $\eta, \eta_s$ from $\tilde \eta, \tilde \eta_s$ gives that $\eta = \eta_s$. For the equality \eqref{E:big-eqn}, using that the first factor in \eqref{E:Pssfss} is $\cF_s$-measurable, we have that
\begin{align*}
	\E_{\cF_s} \lf(\frac{1}{\P_{0, t}(\fA^k(0), \fA^k(t), \fA_{k+1})}  \rg) &= \E_{\cF_s} \lf(\frac{1}{\P_{0, t}(\fB^k(0), \fB^k(t), \fA_{k+1})} \rg) \\
	&=  \frac{\P_{-s, s}(\fA^k(-s), \fA^k(s), \fA_{k+1})}{\P_{-s, s}(\fA^k(-s), \fA^k(s), \fA_{k+1}, [-s, 0] \cup [t, s])},
\end{align*}
which equals $1/\E_{\cF_s} [\P_{0, t}(\fB^{k}_s(0), \fB^{k}_s(t), \fA_{k+1})]$.
\end{proof}

The goal of the next few lemmas is to estimate the right-hand side of \eqref{E:big-eqn}. The first lemma gives a bound on the right-hand side of \eqref{E:big-eqn} in terms of a few simple $\cF_s$-measurable random variables.

\begin{lemma}
	\label{L:RHS-4-estimate}
	Fix $s \ge 2t$ and define $\cF_{s}$-measurable random variables
	\begin{align*}
	D &= D(k, t)= \sqrt{t} + \max_{r, r' \in [0, t]} |\fA_{k+1}(r) - \fA_{k+1}(r')|, \\
M &= M(k, s) = \sqrt{s} + \max_{r, r' \in [-s, s]} |\fA_{k+1}(r) - \fA_{k+1}(r')| + \max_{i \in \II{1, k}} |\fA_i(s) - \fA_i(-s)|.
	\end{align*}
	Then with $\fB_s$ as in the previous lemma, we have
	$$
	\E_{\cF_s} [\P_{0, t}(\fB^{k}_s(0), \fB^{k}_s(t), \fA_{k+1})] \ge \exp\lf(-\frac{c k^3(D^2 + MD)}{s} - c k s \rg).
	$$
\end{lemma}

We separate out one of the main calculations for Lemma \ref{L:RHS-4-estimate}, as it will also be required later in the paper. 

\begin{lemma}
	\label{L:bridge-shift-calc}
	Let $s \ge 2t$, and $\bx, \by \in \R^k_>$. Let $g \in \sC([-s, s])$ be such that $g(-s) < x_k, g(s) < y_k$. 
	
	Let $B$ be a $k$-tuple on independent Brownian bridges from $(-s, \bx)$ to $(s, \by)$, conditioned on the event
	$$
	\NI(g, [-s, 0] \cup [t, s]).
	$$ 
	Define $\iota = (k, \dots, 1), {\bf 1} = (1, 1, \dots, 1) \in \R^k$, and for $\al, \be \ge 0$ define $f^{\al, \be} \in \sC^k([-s, s])$ by letting 
	$$
	f^{\al, \be}(-s) = 0, \quad f^{\al, \be}(0) = f^{\al, \be}(t) = \al \iota + \be {\bf 1}, \quad f^{\al, \be}(s) = 0,
	$$
	and so that $f^{\al, \be}$ is linear on each of the pieces $[-s, 0], [0, t], [t, s]$.
	
	Then for $f \in \NI(g, [-s, 0] \cup [t, s])$ we have that
	\begin{equation}
	\label{E:RN-B-bd}
	\frac{\P(B = f + f^{\al, \be})}{\P(B = f)} \ge \exp \lf(-\frac{k^3(\al + \be)^2}{s} - \frac{k(\al + \be)\sum_{i=1}^k [(f_i(0) - x_i)^+ + (f_i(t) - y_i)^+]}{s} \rg). 
	\end{equation} 
\end{lemma}

The more formal way to interpret \eqref{E:RN-B-bd} is as a bound on the density of the law $\mu_B'$ of $B - f^{\al, \be}$ against the law $\mu_B$ of $B$, on the set $\NI(g, [-s, 0] \cup [t, s])$ where $\mu_B'$ is absolutely continuous with respect to $\mu_B$.

\begin{proof}
	Let $\nu$ be the law of $k$ independent Brownian bridges from $(-s, \bx)$ to $(s, \by)$. Observe that $\mu_B, \mu_B'$ are absolutely continuous with respect to $\nu$ with densities
	\begin{align}
	\nonumber
	&\frac{d \mu_B}{d \nu}(f) = \frac{1}Z \indic(f \in \NI(g, [-s, 0] \cup [t, s])), \\
	\nonumber
	&\frac{d \mu_B'}{d \nu}(f) = \frac{1}{Z} \indic(f + f^{\al, \be} \in \NI(g, [-s, 0] \cup [t, s])) \\
	\label{E:RN-nus}
	&\X\exp \lf(-\frac{2 (f(0) -\bx) \cdot f^{\al, \be}(0) + \|f^{\al, \be}(0)\|^2}{4s} -\frac{2 (f(t) -\by) \cdot f^{\al, \be}(t) + \|f^{\al, \be}(t)\|^2}{4(s-t)}\rg).
	\end{align}
	where $Z = \P_{-s, s}(\bx, \by, g, [-s, 0] \cup [t, s])$ is a normalizing factor. Now, if $f$ is in the set $\NI(g, [-s, 0] \cup [t, s])$, then so is $f + f^{\al, \be}$. Hence the right-hand side of \eqref{E:RN-B-bd} is bounded below by the exponential factor \eqref{E:RN-nus}. We can bound \eqref{E:RN-nus} below by using that $s - t \ge s/2$, and that $0 \le f^{\al, \be} \le k(\al + \be) {\bf 1}$, which yields the desired bound.
\end{proof}

\begin{proof}[Proof of Lemma \ref{L:RHS-4-estimate}]
	All statements in the proof are conditional on $\cF_{s}$. Define the $\cF_s$-measurable vector 
	$$
	\bz = (D + k, D + (k-1), \dots, D + 1)
	$$
and the $\cF_s$-measurable set
\begin{equation}
\label{E:x-y}
O = \{(\bx, \by) \in \R^k_> \X \R^k_> : x_k > \fA_{k+1}(0), y_k > \fA_{k+1}(t) \}.
\end{equation} 
By the definition of $D$, for $(\bx, \by) \in O$ we have 
\begin{equation}
\label{E:P0t}
\P_{0, t}(\bx + \bz, \by + \bz, \fA_{k+1}) \ge \P\Big (\sup_{0 \le r \le t} |B(r)| \le 1/2\Big)^k,
\end{equation}
where $B$ is a Brownian bridge from $(0,0)$ to $(0, t)$. By a standard bound, the right hand side above is bounded below by $e^{-c k t}$. Therefore letting $\mu_\fB$ denote the conditional law of $(\fB^{k}_s(0), \fB^{k}_s(t))$ given $\cF_s$, to complete the proof it suffices to find a set $A \sset O$ such that 
$
\mu_\fB(A + (\bz, \bz))
$
is large. Fix $\Delta > 0$ and let $A_\Delta$ be the $\cF_s$-measurable subset of $(\bx, \by) \in O$  where 
$$
x_i \le \fA_i(-s) + \Delta, \qquad \qquad y_i \le \fA_i(s) + \Delta 
$$
for all $i \in \II{1, k}$. Then by Lemma \ref{L:bridge-shift-calc} with $\al = 1, \be = D$, we have
\begin{align}
\nonumber
&\mu_\fB(A_\Delta+ (\bz, \bz)) \\
\nonumber
&\ge \mu_\fB(A_\Delta) \sup_{(\bx, \by) \in A} \exp \lf(-\frac{k^3(\sqrt{t} + D)^2}{s} - \frac{k(\sqrt{t} + D)\sum_{i=1}^k ((x_i - \fA_i(-s))^+ + (y_i - \fA_i(-s))^+)}{s} \rg) \\
\label{E:mufB}
&\ge \mu_\fB(A_\Delta) \exp \lf(-\frac{4k^3 D^2}{s} - \frac{4k^2 D\Delta}{s} \rg).
\end{align}
In the final line we have used that $\sqrt{t} + D \le 2D$.
It remains to find $\Delta$ where $\mu_\fB(A_\Delta)$ is large.
Define vectors $\bw^s, \bw^{-s}$ where 
$$
w^{\pm s}_i = \fA_i(\pm s) + M + k + 1 - i.
$$
 By Lemma \ref{L:CH-monotonicity}, on the interval $[-s, s]$, the $k$-tuple $(\fB_{s, 1}, \dots, \fB_{s, k})$ is stochastically dominated by $k$ independent Brownian bridges $B = (B_1, \dots, B_k)$ from $(-s, \bw^{-s})$ to $(s, \bw^s)$ conditioned on the event
 $$
 \NI(\fA_{k+1}, [-s, 0] \cup [t, s]).
 $$
 Now, let $L \in \sC^k([-s, s])$ be the function whose $i$th coordinate $L_i$ is the linear function satisfying $L_i(\pm s) = w_i^{\pm s}$. Observe that by the construction of $\bw^{\pm s}$, we have $f \in \NI(\fA_{k+1}, [-s, 0] \cup [t, s])$ for any sequence of bridges $f$ from $(-s, \bw^{-s})$ to $(s, \bw^s)$ with $\| f - L \|_\infty < 1/2$. The probability that this condition holds for a sequence of independent Brownian bridges is bounded below by $\exp(- c k s)$ (similarly to the bound on \eqref{E:P0t}). Therefore
 \begin{align*}
 \P(B_i(r) \le M + k - i + 3/2 + \fA_i(s)& \vee \fA_i(-s) \quad \forall i \in \II{1, k}, r = 0, t) \\
 &\ge \P(\|B - L\|_\infty < 1)  \ge \exp(- c k s).
 \end{align*}
 Observing that 
 $$
 M + k - i + 3/2 + \fA_i(s) \vee \fA_i(-s) - \fA_i(\pm s) \le 2M + 3k
 $$
 for all $i$, we can conclude that 
 $$
 \mu_\fB(A_{2M + 3k}) \ge \P((B(0), B(t)) \in A_{2M + 3k}) \ge \exp(- c k s). 
 $$
 Combining this with the bound on \eqref{E:P0t} and \eqref{E:mufB} and simplifying yields the result.
\end{proof}

It remains to prove tails bounds on $M, D$. The key is a two-point estimate for the Airy line ensemble. In Section \ref{SS:two-time-bound} will prove a highly refined two-point estimate directly for $\fL$, whose proof goes through verbatim for $\fA$. However, the full power of such an estimate is not necessary and we make do with the following lemma, whose proof is much less involved.

\begin{lemma}
	\label{L:two-point-estimate-weak}
		For every $k \in \N$ there exists $c_k > 0$ such that for all $s \in \R, r \in (0, \infty), a > k^2 r^2$ we have
		\begin{equation}
		\label{E:lemma61}
		\begin{split}
		\p(|\fA_k(s) + s^2 - &\fA_k(s + r) - (s + r)^2| > a) \\
		\le &\;c_k\exp\Big(- \frac{a^2}{4r} - \frac{k a^{3/2}}{8} + 3 a^{3/2} + \frac{a}{\sqrt{r}} + 2 k^2 |\log (\sqrt{a}/r)|  \Big).
		\end{split}
		\end{equation}
\end{lemma}

The proof for Lemma \ref{L:two-point-estimate-weak} is similar to the proof of Lemma 6.1 in \cite{dauvergne2021bulk}, except here we keep closer track of the error terms. Somewhat remarkably, this already gives a stronger tail bound than for Brownian motion when $k$ is large enough. We will need the more nuanced method in Section \ref{SS:two-time-bound} to give a stronger tail bound than Brownian motion for all $k$.

\begin{proof}
	By stationarity, it suffices to prove the lemma when $s = 0$. Moreover, since $\fA(\cdot) = \fA(-\; \cdot)$, it suffices to prove the bound with the absolute value removed. For $r > 0$, define $$
	L_r = \fA_k(0) - \fA_k(r) - r^2.
	$$
Let $\la > r$ be a parameter that we will optimize over. Let $\cF_\la$ be the $\sig$-algebra generated by $\fA$ outside of the set $\II{1, k} \X [0, \la]$. Let $\bv = \sqrt{r}(1/k, 2/k \dots, 1)$, and let $B$ be a $k$-tuple of non-intersecting Brownian bridges from $(0, \fA_{k}(0) - \bv)$ to $(\la, \fA_{k}(\la) - \bv)$. By Lemma \ref{L:CH-monotonicity}, given $\cF_\la$, the Airy lines $\fA|_{\II{1, k} \X [0, \la]}$ stochastically dominate $B$, so
\begin{align*}
\P(L_r > a \mid \cF_\la) &\le \P(\fA_k(0) - B_k(r) > r^2 + a \mid \cF_\la) = \P(B_k(0) - B_k(r) > r^2 - \sqrt{r} + a \mid \cF_\la).
\end{align*}
Now let $\tilde B$ be a collection of 
$k$-tuple of independent Brownian bridges from $(0, \fA_k(0) - \bv)$ to $(\la, \fA_k(\la) - \bv)$ without a non-intersection condition. By Lemma \ref{L:nonint-probability} with $t = \la, \al = \sqrt{r}/k$,
\begin{align}
\nonumber
\P(B_k&(0) - B_k(r) > r^2 - \sqrt{r} + a \mid \cF_\la) \\
\nonumber
&\le \exp \lf(-\frac{kr}{6 \la} + \frac{k(k-1)\log (k^2 \la/(2r))}{2}  \rg) \P(\tilde B_k(0) - \tilde B_k(r) > r^2 - \sqrt{r} + a \mid \cF_\la) \\
\label{E:Fla}
&\le \exp \lf(k^2 \log (k \la/r) \rg) \P(\tilde B_k(0) - \tilde B_k(r) > a - \sqrt{r} \mid \cF_\la).
\end{align}
Given $\cF_\la$, the increment $\tilde B_k(0) - \tilde B_k(r)$ is normal with variance $2r(1-r/\la)$ and mean $r L_\la/\la + r \la$. 

Therefore we have 
\begin{align}
\nonumber
\P(\tilde B_k(0) - \tilde B_k(r) > a - \sqrt{r} \mid \cF_\la) &\le c \exp\lf(- \frac{(a - \sqrt{r} - r L_\la/\la - r \la)^2}{4(1-r/\la)r} \rg) \\
\nonumber
&\le \exp\lf(- \frac{(a^2 - 2a(\sqrt{r} + r L_\la/\la + r\la))(1 + r/\la)}{4r} \rg) \\
\label{E:Lla}
&\le \exp\lf(- \frac{a^2}{4r} - \frac{a^2}{4\la} + \frac{a}{\sqrt{r}} + a (L_\la)^+/\la + a \la\rg).
\end{align}
In the final inequality, we have used that $1 + r/\la \le 2$.
Now, by Lemmas \ref{L:one-point-upper-bound} and \ref{L:one-point-lower-bound}, for all $m > 0$,
\begin{align*}
\P(L_\la > m) &\le \P(\fA_k(0) > m - m^{2/3}) + \cP(\fA_k(0) < -m^{2/3}) \\
&\le c_k e^{-\frac{4k}{3}(m - m^{2/3})^{3/2}} + 2 e^{- d_k m^2}.
\end{align*}
Therefore there exists $c_k' > 0$ such that $L_\la$ is stochastically dominated by a random variable with Lebesgue density bounded above by $c_k'e^{-\frac{4k-1}{3}m^{3/2}}$ and so
$$
\E e^{a L_\la/\la} \le c_k' \int_0^\infty e^{a m/\la -\frac{4k-1}{3}m^{3/2}} dm.
$$
For any $\al, \be > 0$, if $f(x) = \al x - \be x^{3/2}$ and $u = 4 \al^2/(9 \be^2)$ is the argmax of $f$, then
$$
\int_0^\infty e^{f(x)} dx \le e^{f(u)}(2 + \int_{-\infty}^\infty e^{- |f'(u + 1) \wedge f'(u-1)| x} dx) \le \frac{4 e^{f(u)}}{|f'(u + 1) \wedge f'(u-1)|} \le \frac{4 e^{f(u)}}{\al}.
$$
Using this inequality and simplifying the result gives that
$
\E e^{a L_\la/\la} \le c_k'' (\la/a) e^{4 a^3 \la^{-3}/(3(4k-1)^2)},
$
and so the expectation of \eqref{E:Lla} is bounded above by
\begin{equation}
\label{E:a24t}
c_k'' \exp\lf(- \frac{a^2}{4r} - \frac{a^2}{4\la} + \frac{a}{\sqrt{r}} + \frac{4 a^3}{3\la^3(4k-1)^2} + a \la + \log(\la/a)\rg).
\end{equation}
A quick inspection reveals that the minimum value of \eqref{E:a24t} is taken when $\la = O(k^{-1} \sqrt{a})$. Plugging in $\la = \sqrt{a}k^{-1}$ (which is bounded below by $r$ since $a > k^2 r^2$) and simplifying gives that \eqref{E:a24t} is bounded above by
$$
c_k''\exp\lf(- \frac{a^2}{4r} - \frac{k a^{3/2}}{8} + 3 a^{3/2} + \frac{a}{\sqrt{r}} \rg).
$$
Combining this with \eqref{E:Fla} yields the result.
\end{proof}

We can deduce tail bounds on $M$ and $D$ immediately from Lemma \ref{L:two-point-estimate-weak}.

	\begin{lemma}
	\label{L:MD-bounds}
	For every $k \in \N$, there exists $c_k > 0$ such that the following bounds hold for all $t \ge 1, s \ge 2t$. Letting $D = D (k, t)$ and $M = M(k, s)$ we have
	\begin{align*}
	\P(D > a) \le \exp\Big(c_k t^3 - \frac{a^2}{320 t}\Big), \qquad 
	\P(M > a) \le \exp\Big(c_k s^3 - \frac{a^2}{320 s}\Big).
	\end{align*}
\end{lemma}

\begin{proof}
To prove the bound on $D$, we apply Lemma \ref{L:mod-cont} to the process $\fA_k(r), r \in [0, t]$. Set $\cA_i(s) = \fA_i(s) + s^2$. For $r < r + \ep \in [0, t]$ and $a > c_k' t^{3/2}$, we have the estimate
\begin{align*}
\P(|\fA_k(r) - \fA_k(r + \ep)| > a\sqrt{\ep}) \le \P(|\cA_k(r) - \cA_k(r + \ep)| > a\sqrt{\ep} + t \ep) \le c_k \exp(-a^2/5).
\end{align*}
Here we use Lemma \ref{L:two-point-estimate-weak} in the second bound, along with the lower bound on $a$ to simplify the result. This estimate ensures that $\fA$ on $[0, t]$ satisfies Lemma \ref{L:mod-cont} with $\beta = 1/5$ and $\al = \exp(c_k t^3)$; the large value of $\al$ deals with the case when $a \le c_k' t^{3/2}$. Appealing to that lemma yields the result. The proof of the bound on $M$ is essentially the same.
\end{proof}

We can finally give a concrete bound on $\E [\P_{0, t}(\fB^{k}_s(0), \fB^{k}_s(t), \fA_{k+1} \mid \cF_{s})]$, which in turn shows that the measure $\eta$ is finite.

\begin{corollary}
	\label{C:bounded-expectation}
	There exists $d_k \in (0, 1)$ such that for all $s \ge 2 t, \ep > 0$ we have
	\begin{equation}
	\label{E:m-bound}
	\E_{\cF_s} [\P_{0, t}(\fB^k_s(0), \fB^k_s(t), \fA_{k+1})] \le \ep
	\end{equation}
	with probability at most 
	$$
	e^{O_k(s^3)} \exp \lf(- d_k \al \sqrt{s/t}\rg),
	$$
	where $\alpha = \log(\ep^{-1})$. In particular, when $s = 4d_k^{-2}t$, this bound simplifies to 
	$e^{O_k(t^3)} \ep^2$, and so 
	\begin{equation}
	\label{E:eta-CNR}
	\eta(\sC^\N(\R)) = \E \lf(\frac{1}{\E_{\cF_s} [\P_{0, t}(\fB^k_s(0), \fB^k_s(t), \fA_{k+1})]}\rg) \le e^{O_k(t^3)}.
	\end{equation}
\end{corollary}

\begin{proof}
Without loss of generality, we may assume $\al \ge c_k s^{5/2} t^{1/2}$ for some large constant $c_k$, since otherwise the bound holds trivially. Let $D = D(k, t), M = M(k, s)$ and write $D' = t^{-1/2} D - c_k' t^{3/2}, M' = s^{-1/2} M - c_k' s^{3/2}$. Then by Lemma \ref{L:MD-bounds}, for $X = D' \vee M'$ and $a > 0$ we have
$
\P(X > a) \le e^{- d_k' a^2}.
$
Moreover, by Lemma \ref{L:RHS-4-estimate}, we have
\begin{align*}
	\E_{\cF_s} [\P_{0, t}(\fB^k_s(0), \fB^k_s(t), \fA_{k+1})] &\ge \exp\lf(-c_k' (X^2 \sqrt{t/s} + s \sqrt{t} X + s t^2) \rg).
\end{align*}
Therefore the probability of the event \eqref{E:m-bound} is at most
\begin{align*}
\P(c_k' (X^2 \sqrt{t/s} + s \sqrt{t} X + s t^2) \ge \al) &\le \P(X(X + s^{3/2}) \ge c_k^{\prime -1} \sqrt{s/t} \al/2) \\
&\le \P(X^2 \ge c_k^{-1} \sqrt{s/t} \al/4) + \P(X \ge c_k^{-1} \al/(4 s \sqrt{t})) \\
&\le 2 \exp \lf(- d_k (\al \sqrt{s/t}) \wedge (\al^2/(s^2 t))\rg) \\
&= 2 \exp \lf(- d_k \al \sqrt{s/t}\rg)
\end{align*}
where we have used the lower bound on $\al$ and the tail bound on $X$ to achieve the bound.
The equality in \eqref{E:eta-CNR} follows from \eqref{E:CNR} and Lemma \ref{L:eta-eta-s}, and the inequality is immediate from the simplified form of the bound on \eqref{E:m-bound} when $s = 4 d_k^{-2} t$. 
\end{proof}

By Corollary \ref{C:bounded-expectation}, $\eta$ is a finite measure, and so the normalized measure $\eta' = \tfrac{1}{\eta(\sC^\N(\R))}\eta$ is a probability measure. Let $\fL = \fL^{t, k}$ denote a line ensemble whose law is equal to $\eta'$. It is easy to check that $\fL$ satisfies Theorem \ref{T:resampling-candidate}.

\begin{proof}[Proof of Theorem \ref{T:resampling-candidate}]
Part $1$ is immediate from the definition of $\eta$. Next, by the definition \eqref{E:eta-defn}, restricting $\eta$ to the set 
$$
\sC^\N_>(\R) = \{f \in \sC^\N(\R) : f_1 > f_2 > \dots\}
$$
gives a probability measure. By the Brownian Gibbs property, this measure is the law of $\fA$. Therefore conditioning $\fL$ to lie in $\sC^\N_>(\R)$ gives a parabolic Airy line ensemble. The bound $\P(\fL \in \sC^\N_>(\R)) \ge e^{-O_k(t^3)}$ follows from the bound on $\eta(\sC^\N(\R))$ in Corollary \ref{C:bounded-expectation}. This gives Theorem \ref{T:resampling-candidate}.3.

For part $2$, first observe that if $\fL$ satisfies the desired Gibbs property on all sets $S = \II{1, \ell} \X [a, b]$ where $a < 0 < t < b$ and $\ell > k$, then it satisfies the Gibbs property everywhere. Fix such a set $S$. Recalling the definition of $\fB$ from the construction of $\eta$, and let $\fA$ be coupled to $\fB$ as in that construction. We first compute the conditional density of $\P_{\sig(\fB|_{S^c})}(\fB|_S \in \cdot)$ against the law of $\ell$ independent Brownian bridges $B$ from $(a, \fB^\ell(a))$ to $(b, \fB^\ell(b))$.
Let $S_1 = \II{1, k} \X [0, t]$ and $S_2 = S \smin S_1$. First, $\fB|_{S_2} = \fA|_{S_2}$ so by the Brownian Gibbs property for $\fA$ we have 
\begin{align}
\label{E:B-fB}
\frac{\P_{\sig(\fB|_{S^c})}(\fB|_{S_2} = f)}{\P_{\sig(\fB|_{S^c})}(B|_{S_2} = f)} = \frac{\indic(f_i(r) > f_{i+1}(r) \text{ for all } (i, r) \in S_2)\P_{0, t}(f^k(0), f^k(t), f_{k+1})}{Z}.
\end{align}
where in the indicator above we say that $f_{\ell+1} = \fB_{\ell+1}$, and $Z$ is the $\sig(\fB|_{S^c})$-measurable normalization factor $\P_{a, b}(\fB^\ell(a), \fB^\ell(b), \fB_{k+1})$.
 Next, given $\fB|_{S_2 \cup S^c}$, the distribution of $\fB|_{S_1}$ is that of $k$ independent Brownian bridges from $(0, \fB^k(0))$ to $(t, \fB^k(t))$. As the same is true of the unconditioned bridges $B$, the density
$$
\frac{\P_{\sig(\fB|_{S^c})}(\fB|_{S} = f)}{\P_{\sig(\fB|_{S^c})}(B|_{S} = f)} 
$$
is the same as the right-hand side of \eqref{E:B-fB}. Therefore by \eqref{E:eta-defn}, if $B'$ consists of $\ell$ independent Brownian bridges $B$ from $(a, \fL^\ell(a))$ to $(b, \fL^\ell(b))$ then
\begin{align*}
\frac{\P_{\sig(\fL|_{S^c})}(\fL|_{S} = f)}{\P_{\sig(\fL|_{S^c})}(B|_{S} = f)}  = \frac{\indic(f_i(r) > f_{i+1}(r) \text{ for all } (i, r) \in S_2)}{Z'},
\end{align*}
where $Z'$ is a new $\sig(\fL|_{S^c})$- measurable normalization factor. This is the desired Gibbs property.
\end{proof}

\begin{remark}
It is interesting to note that Theorem \ref{T:resampling-candidate} already implies that on the set $\II{1, k} \X[0, t]$, if $L \in \sC^k([0, t])$ is the function whose coordinates $L_i$ are linear functions given by the conditions $L_i(0) = \fA_i(0), L_i(t) = \fA_i(t)$, then $[\fA - L]|_{\II{1, k} \X [0, t]}$ has a uniformly bounded density against the law of $k$ independent Brownian bridges.
\end{remark}

\subsection{The line ensembles $\fL^{t, k, \ba}$}
Next, we construct the more general line ensembles introduced in Remark \ref{R:multiple-patches}. 

\begin{theorem}
	\label{T:resampling-candidate-multiple}
	Fix $t \ge 1, k, \ell \in \mathbb N$ and $\ba \in \R^\ell$ such that $a_j + t < a_{j+1}$ for all $j \in \II{1, \ell - 1}$. Define $U(\ba) = \bigcup_{j=1}^\ell (\II{1, k} \X (a_j, a_j + t))$.
	
	Then there exists a random sequence of continuous functions $\fL^\ba = \fL^{t, k, \ba} = \{\fL^\ba_i : \R \to \R, i \in \N\}$ such that the following points hold:
	\begin{enumerate}
		\item Almost surely, $\fL^\ba$ satisfies $\fL^\ba_i(r) > \fL^\ba_{i+1}(r)$ for all pairs $(i, r) \notin U(\ba)$.
		\item The ensemble $\fL^\ba$ has the following Gibbs property. Let $m \in \N$ and $a < b \in \R$ and set $S = \II{1, m} \X [a, b]$. Then conditional on the values of $\fL^\ba_i(r)$ for $(i, r) \notin S$, the distribution of $\fL^\ba_i(r), (i,r) \in S$ is given by $m$ independent Brownian bridges $B_1, \dots, B_m$ from $(a, \fL^\ba_i(a))$ to $(b, \fL^\ba_i(b))$ for $i \in \II{1, m}$, conditioned on the event $B_i(r) > B_{i+1}(r)$ whenever $(i, r) \notin U(\ba)$.
		\item We have
		\begin{equation}
		\label{E:fLa1}
\P(\fL^\ba_1 > \fL^\ba_2 > \dots ) \ge \exp(-e^{O_k (1 + \ell \log \ell)} t^3)),
		\end{equation}
		and conditional on the event $\{\fL^\ba_1 > \fL^\ba_2 > \dots\}$, the ensemble $\fL^\ba$ is equal in law to the parabolic Airy line ensemble $\fA$.
	\end{enumerate}
\end{theorem}

To prove Theorem \ref{T:resampling-candidate-multiple} we require a multi-interval extension of Lemma \ref{L:eta-eta-s} and Corollary \ref{C:bounded-expectation}.
Setting some notation, for a disjoint union of compact intervals $A = \bigcup_{i=1}^m [c_i, d_i]$, let $\cF_A$ be the $\sig$-algebra generated by $\{\fA_i(x) : (i, x) \notin \II{1, k} \times A\}$. Also, for a disjoint union of compact intervals $B = \sum_{i=1}^m [a_i, b_i]$ with $[a_i, b_i] \subset [c_i, d_i]$ for all $i$, we define a line ensemble $\fB_{B, A}$ analogously to how we defined $\fB_s$ prior to Lemma \ref{L:eta-eta-s}. 

More precisely, let $\fB_{B, A, i}(r) = \fA_i(r)$ for $(i, r) \notin \II{1, k} \times A$, and let $\fB_{B, A}^k|_{[c_i, d_i]}, i = 1, \dots, m$ have the conditional law given $\fA$ outside of $\II{1, k} \X A$ of $m$ independent collections of $k$ independent Brownian bridges from $(c_i, \fA^k(c_i))$ to $(d_i, \fA^k(d_i))$ conditioned on the event
$$
\NI(\fA_{k+1}, A \setminus B).
$$
In particular $\fB_{B, B}$ is simply given by $\fA$ off of $\II{1, k} \times B$, connected by independent Brownian bridges on $\II{1, k} \times B$.
For a line ensemble $f$ it will also be useful to introduce the random variable
$$
Z(f, B) := \prod_{i=1}^m \P_{a_i, b_i}(f^k(a_i), f^k(b_i), \fA_{k+1}).
$$
\begin{lemma}
	\label{L:eta-eta-multi}
For any $A, B$ as above, we have
	\begin{equation}
		\label{E:big-eqn-multi}
		\E_{\cF_A} \frac{1}{Z(\fB_{B, B}, B)} = \frac{1}{\E_{\cF_A} Z(\fB_{B, A}, B)}.
	\end{equation}
Moreover, there exists a $k$-dependent constant $d_k' > 0$ such that if $T = \max_i (d_i - c_i)$, and $\de \in (0, d_k' m^{-2})$ is such that $a_i = c_i/2 + d_i/2, b_i = c_i(1/2 - \de) + d_i(1/2 + \de)$ for all $i$, then
\begin{equation}
	\label{E:PEFA}
\E [Z(\fB_{B, B}, B)^{-1}] = \E [\E_{\cF_A} Z(\fB_{B, A}, B)]^{-1} \le e^{O_k(T^3)}.
\end{equation}
\end{lemma}
\begin{proof}
	We first compute the conditional density given $\cF_A$ of $\fB_{B, B}^k|_A$ against $\fB_{B, A}^k|_A$. Conditional on $\cF_A$, all of $\fB_{B, A}^k|_A, \fB_{B, B}^k|_A, \fA^k|_A$ have densities with respect to the law of $m$ independent collections of $k$ independent Brownian bridges from $(c_i, \fA^k(c_i))$ to $(d_i, \fA^k(d_i)), i = 1, \dots, m$. The density of $\fB_{B, A}^k|_A$ at a continuous function $f:A \to \R^k$ is proportional to
	\begin{equation}
		\label{E:ffff}
		\indic(f \in \NI(\fA_{k+1}, A \smin B)),
	\end{equation}
	whereas the density of $\fA|_U$ is proportional to
	\begin{equation}
		\label{E:squig-1}
		\indic(f \in \NI(\fA_{k+1}, A)).
	\end{equation}
	We obtain $\fB|_{B, B}$ from $\fA$ by replacing $\fA$ with independent Brownian bridges on $\II{1, k} \times B$. The density of $\fA$ on $\II{1, k} \times B$ with respect to independent Brownian bridges between the boundary points equals
	\begin{equation}
		\label{E:indic-g-squig}
		\frac{\indic(g \in \NI(\fA_{k+1}, B))}{Z(g, B)}
	\end{equation}
	for $g:B \to \R^k$. Therefore the density of $\fB_{B, B}^k|_A$ is proportional to the product of \eqref{E:squig-1} with the inverse of \eqref{E:indic-g-squig} at the function $g = f|_B$ as long as \eqref{E:indic-g-squig} is non-zero. This is
	\begin{equation}
		\label{E:indic-fg-fg}
		\indic(f \in \NI(\fA_{k+1}, A \smin B)) Z(f, B).
	\end{equation}
	The same formula holds when \eqref{E:indic-g-squig} evaluated at $f|_B$ equals $0$ since the density of $\fB_{B, B}^k|_A$ against independent Brownian bridges only depends on $f|_{A \smin B}$, as in the proof of Lemma \ref{L:eta-eta-s}.
Examining \eqref{E:indic-fg-fg} and \eqref{E:ffff}, we see that conditional on $\cF_A$, $\fB_{B, B}|_U$ has a density against $\fB_{B, A}|_U$ which is proportional to $Z(f, B)$. Including the normalization constant, we see that this conditional density equals
$$
\frac{Z(f, B)}{\E_{\cF_A} Z(\fB_{B, A}, B)}.
$$
Equation \eqref{E:big-eqn-multi} then follows from a change-of-variables. 

The first equality in \eqref{E:PEFA} is immediate from \eqref{E:big-eqn-multi}. For the second equality, observe that 
\begin{align*}
\E_{\cF_A} Z(\fB_{B, A}, B) &= \E_{\cF_A} \prod_{i=1}^m \P_{a_i, b_i} (\fB_{B, A}^k(a_i), \fB_{B, A}^k(b_i), \fA_{k+1}) \\
&=\prod_{i=1}^m \E_{\cF_A} \P_{a_i, b_i} (\fB_{B, A}^k(a_i), \fB_{B, A}^k(b_i), \fA_{k+1}) \\
&=\prod_{i=1}^m \E_{\cF_{[c_i, d_i]}} \P_{a_i, b_i} (\fB_{[a_i, b_i], [c_i, d_i]}^k(a_i), \fB_{[a_i, b_i], [c_i, d_i]}^k(b_i), \fA_{k+1}).
\end{align*}
Here the second and third equalities use that conditional on $\cF_A$ the functions $\fB_{B, A}^k|_{[a_i, b_i]}$ are independent, and individually equal in law to $\fB_{[a_i, b_i], [c_i, d_i]}^k|_{[a_i, b_i]}$ conditional on the finer $\sig$-algebra $\cF_{[c_i, d_i]}$. Now, by stationarity of $\fA$ and the relationships between $a_i, b_i, c_i, d_i$,  each of the random variables in the above product can be controlled by Corollary \ref{C:bounded-expectation}. Applying this corollary gives 
\begin{align*}
\P(\E_{\cF_A} Z(\fB_{B, A}, B) < \ep) &\le \sum_{i=1}^m \P(\E_{\cF_{[a_i, b_i]}} \P_{a_i, b_i} (\fB_{[a_i, b_i], [c_i, d_i]}(a_i), \fB_{[a_i, b_i], [c_i, d_i]}(b_i), \fA_{k+1}) < \ep^{1/m}) \\
&\le m e^{O_k(T^3)} \exp \lf(- d_k \log (\ep^{-1/m})/\sqrt{\de} \rg) \le e^{O_k(T^3)} \ep^2.
\end{align*}
Here the final bound uses that $\de < d_k' m^{-2}$. The bound \eqref{E:PEFA} follows.
\end{proof}

We also require a simple covering lemma.

\begin{lemma}
\label{L:covering-lemma}
Fix $t > 0, \ell \in \N, \de \in (0, 1/4)$ and consider $\ba \in \R^\ell$ with $b_j := a_j + t < a_{j+1}$ for all $j \in \II{1, \ell - 1}$. Then there exists a collection of disjoint intervals $\mathcal C = \{[c_1, d_1], \dots, [c_m, d_m]\}$ such that
\begin{itemize}[nosep]
	\item $\max_i d_i - c_i \le (5/\de)^\ell t$.
	\item For all $j \in \II{1, \ell}$, $[a_j, b_j] \subset [c_i/2 + d_i/2, c_i(1/2-\de) + d_i(1/2 + \de)]$ for some $i \in \II{1, m}$.
\end{itemize} 
We call $\cC$ a $\de$-\textbf{admissible cover} for $\ba$.
\end{lemma}

\begin{proof}
We prove the lemma by induction. In the base case when $\ell = 1$, we set $c_1 = a_1 - t/(2\de), d_1 = a_1 + t/(2\de)$. Now fix $\ell \ge 2$, and suppose the lemma holds for all $j \le \ell$. Suppose first that there exists some $j \in \II{1, \ell - 1}$ such that $a_{j+1} - a_j > 2(5/\de)^{j \vee (\ell-j)} t$. In this case, consider admissible covers $\cC_1, \cC_2$ of $(a_1, \dots, a_j)$ and $(a_{j+1}, \dots, a_\ell)$. By possibly removing redundant intervals, we may also assume that every $[c, d] \in \cC_1 \cup \cC_2$ contains at least one interval $[a_j, b_j]$. Then for any $[c, d] \in \cC_1, [c', d'] \in \cC_2$ we have
$$
d < a_j + (5/\de)^j t < a_{j+1} - (5/\de)^{\ell-j} t < c',
$$
where the first and third inequalities above use the first two bullet points in the lemma. Hence the intervals in $\cC = \cC_1 \cup \cC_2$ are all disjoint, and hence $\cC$ is an admissible cover for $\ba$. If $a_{j+1} - a_j \le 2(5/\de)^{j \vee (\ell-j)} t$ for all $j$, then 
$$
b_\ell - a_1 \le t + 2t \sum_{j=1}^{\ell-1} (5/\de)^{j \vee (\ell-j)} \le \de (5/\de)^j t,
$$
and so the single interval $[a_1 - (5/\de)^\ell t/2, a_1 + (5/\de)^\ell t / 2]$ gives an admissible cover.
\end{proof}

\begin{proof}[Proof of Theorem \ref{T:resampling-candidate-multiple}]
First, by shift invariance of $t \mapsto \fA(t) + t^2$ we may assume $a_1 = 0$. 
In this case, we can construct $\fL^{\ba}$ by taking the line ensemble $\fL^{a_\ell + t, k}$ from Theorem \ref{T:resampling-candidate} and conditioning on the following avoidance event $N_1$:
\begin{equation}
\label{E:ka}
\fL^{a_\ell + t, k}_i(r) > \fL^{a_\ell + t, k}_{i+1}(r) \qquad  \text{ for all } (i, r) \notin U(\ba).
\end{equation}
This line ensemble satisfies properties $1, 2$ by virtue of the corresponding properties for $\fL^{a_\ell + t, k}$. It remains to obtain the lower bound required for property $3$.
For this, let $\de = d_k' \ell^{-2}$, where $d_k'$ is a small $k$-dependent constant, and using Lemma \ref{L:covering-lemma} let $\mathcal C = \{[c_1, d_1], \dots, [c_m, d_m]\}$ be a $\de$-admissible cover for $\ba$. Define disjoint intervals $[a_1, b_1'], \dots, [a_m', b_m']$ where $a_i = (c_i + d_i)/2$ and $b_i' = a_i + \de(c_i - d_i)$. Then every interval $[a_i, a_i + t]$ is contained in one of the intervals $[a_i', b_i']$, each of which is a subset of $[c_i, d_i]$. Let $A = \bigcup_{i=1}^m [c_i, d_i]$ and $B = \bigcup_{i=1}^m [a_i', b_i']$. 

Now define a line ensemble $\hat \fL^\ba$ by taking the line ensemble $\fL^{a_\ell + t, k}$ and conditioning on the weaker avoidance event $N_2$:
$$
\fL^{k, a_\ell + t}_i(r) > \fL^{k, a_\ell + t}_{i+1}(r) \qquad  \text{ for all } (i, r) \notin \II{1, k} \times B.
$$
Since $N_1 \subset N_2$ and the full avoidance event in \eqref{E:fLa1} is a subset of $N_1$, we have that
\begin{equation}
\label{E:new-31}
\P(\fL^\ba_1 > \fL^\ba_2 > \dots ) \ge \P(\hat \fL^\ba_1 > \hat \fL^\ba_2 > \dots),
\end{equation}
so it suffices to lower bound the right-hand side above.

Using the notation introduced prior to Lemma \ref{L:eta-eta-multi}, let $\tilde \eta$ denote the law of $\fB_{B, B}$, and let $\hat \eta$ denote the law of $\hat \fL^\ba$. The relationship between the measures $\tilde \eta, \hat \eta$ introduced here is almost the same as the relationship between the measures $\tilde \eta, \eta$ introduced prior to Lemma \ref{L:eta-eta-s}, except that here we have already normalized $\hat \eta$ to a probability measure. Indeed, $\hat \eta$ is absolutely continuous with respect to $\tilde \eta$ with density given by
\begin{equation*}
\left.
\frac{d \hat \eta}{d \tilde \eta}(f) = \frac{1}{Z(f, B)} \middle/ \E \frac{1}{Z(\fB_{B, B}, B)} \right.
\end{equation*}
On the other hand, the probability on the right-hand side of \eqref{E:new-31} equals $\E Z(\hat \fL^\ba, B)$, which by the above formula equals $1/\E Z(\fB_{B, B}, B)^{-1}$. By Lemma \ref{L:eta-eta-multi}, this is bounded below by $e^{-O_k(T^3)},
$
where $T = \max_i d_i - c_i$. Finally, the first bullet in Lemma \ref{L:covering-lemma} and the definition of $\de$ gives that $T \le t e^{O_k(1 + \ell \log \ell)}$, yielding the result.
\end{proof}

The claimed asymptotic independence in Remark \ref{R:multiple-patches} of the patches $\fL^\ba|_{\II{1, k} \X [a_i, a_i + t]}$ as $\min |a_i - a_{i-1}| \to \infty$ follows from the strong mixing of the Airy line ensemble shown in \cite{corwin2014ergodicity}.

\section{Tail bounds for $\fL^{t, k}$}
\label{S:tail-bounds}

In this section, we prove one- and two-time tail bounds on $\fL = \fL^{t, k}$ (including the bounds in Example \ref{Ex:one-line}). For readability, we have not aimed to prove tail bounds on the more general ensembles $\fL^{t, k, \ba}$ alongside the tail bounds for $\fL$, though these can be obtained with similar methods.

As in the previous section, we fix $k \in \N$ and $t \ge 1$ throughout the section. These bounds will be used in Section \ref{S:upper-bound} to move from Theorem \ref{T:resampling-candidate} to the upper bound in Theorem \ref{T:tetris-theorem}.

\subsection{The one-time bound}
\label{SS:one-time-bound}
The single-time bound is easier, so we start there. The idea is to first prove tail bounds on the line ensembles $\fB_s$ introduced for Lemma \ref{L:eta-eta-s}, and then deduce tail bounds on $\fL$ through the definition \eqref{E:eta-s-defn} and the bounds in Corollary \ref{C:bounded-expectation}. 

\begin{lemma}
	\label{L:one-point-bound-fB}
	Let $\bm \in [0, \infty)^k_\ge$ and define the set
	$$
	A(\bm, r) = \{f \in \sC^\N(\R) : f_i(r) > m_i + r^2 \quad \text{ for all } i \in \II{1, k} \}.
	$$
	Then for $r \notin (0, t)$ we have
	$$
	\P(\fB_s \in A(\bm, r)) \le \exp\Big(-\frac{4}{3} \sum_{i=1}^k m_i^{3/2} + c_k \log (\|\bm\|_2 + s)\Big).
	$$
\end{lemma}

\begin{proof}
	First, the bound is immediate from Lemma \ref{L:one-point-upper-bound} when $|r| \ge s$ since $\fA(r) = \fB_s(r)$ when $r \ge s$. 
The case when $|r| < s$ is more involved.
	Let 
	$$
	S = \II{1, k} \X [-s, s], \qquad S_1 = \II{1, k} \X ([-s, 0] \cup [t, s]), \qquad S_2 = \II{1, k} \X [0, t].
	$$
Define a line ensemble $\fM_s$ by letting $\fM_{s,i}(r) = \fA_i(r)$ for $(i, r) \notin S$, and let $\fM_{s}|_{S}$ have the conditional law given $\fA|_{S^c}$ of $k$ independent Brownian bridges from $(-s, \fA^k(-s))$ to $(s, \fA^k(s))$ conditioned on the event
$$
\NI(\fA_{k+1}, [-s, 0] \cup [t, s]) \cap \NI(-\infty, [0, t]).
$$
In words, $\fM_s$ is defined similarly to $\fB_s$, except with the extra condition that $\fM_{s,1} > \dots > \fM_{s,k}$ on the interval $[0, t]$. Then by Lemma \ref{L:CH-monotonicity}, $\fM_s$ is stochastically dominated by $\fA$, so by Lemma \ref{L:one-point-upper-bound}, we have
\begin{equation}
\label{E:fMI}
\P(\fM_i(r) + r^2 > m_i \quad \text{ for all } i \in \II{1, k}) \le c_k \exp\Big(-\frac{4}{3} \sum_{i=1}^k m_i^{3/2}\Big)
\end{equation}
for all $r \in \R$, $\bm \in [0, \infty)_>^k$. Our goal is to effectively compare $\fM^k_s|_{S_2}$ to $\fB^k_s|_{S_2}$ conditionally on the $\sig$-algebra $\cF_s := \sig(\fA|_{(\II{1, k} \X [-s, s])^c})$ in order to apply this bound. By Lemma \ref{L:bridge-shift-calc} (and with notation as in that lemma), for any $g \in \NI(\fA_{k+1}, [-s, 0] \cup [t, s])$ we have
$$
\frac{\P_{\cF_s}(\fB_s|_{S_2} = [g + f^{\ep, 0}]|_{S_2} )}{\P_{\cF_s}(\fB_s|_{S_2} = g|_{S_2})} \ge \exp \lf(-\frac{k^3\ep^2}{s} - \frac{k^2 \ep (g_1(0) \vee g_1(t) - \fA_{k+1}(-s) \wedge \fA_{k+1}(s))^+}{s} \rg). 
$$
On the other hand, the conditional laws of $\fB_s|_{S_2}$ and $\fM_s|_{S_2}$ given $\cF_s$ are mutually absolutely continuous with density
\begin{align*}
\frac{\P_{\cF_s}(\fM_s|_{S_2} = g)}{\P_{\cF_s}(\fB_s|_{S_2} = g)} &= \frac{\P_{0, t}(g(0), g(t), -\infty) \P_{\cF_s}(\fA^k|_{[-s, s]} \in \NI(\fA_{k+1}, [-s, 0] \cup [t, s]))}{\P_{\cF_s}(\fA^k|_{[-s, s]} \in \NI(\fA_{k+1}, [-s, 0] \cup [t, s]) \cap \NI(-\infty, [0, t]))} \\
&\ge \P_{0, t}(g(0), g(t), -\infty).
\end{align*}
In particular, if $g_i(r) \ge g_{i+1}(r) + \ep$ for all $i \in \II{1, k-1}, r \in  \{0, t\}$ for some $\ep < 1/2$, then by Lemma \ref{L:nonint-probability}, this is bounded below by
$
e^{-c_k \log (t/\ep)}.
$
Putting this together with the previous bound implies that for $g \in \NI(\fA_{k+1}, [-s, 0] \cup [t, s])$, setting
$$
Y_g = 2s + (g_1(0) \vee g_1(t) - \fA_{k+1}(-s) \wedge \fA_{k+1}(s))^+,
$$ 
and $\ep = s/Y_g$, we have
\begin{equation}
\label{E:fMs}
\frac{\P_{\cF_s}(\fM_s|_{S_2} = [g + f^{\ep, 0}]|_{S_2} )}{\P_{\cF_s}(\fB_s|_{S_2} = g|_{S_2})} \ge \exp\lf(-c_k \log Y_g \rg).
\end{equation}
Therefore for every $\bm \in [0, \infty)^k_\ge, r \in [-s, 0] \cup [t, s],$ and $\la > 0$ we have
$$
\P_{\cF_s}(\fB_s \in A(\bm, r)) \le e^{c_k \log \la} \P_{\cF_s}(\fM_s \in A(\bm, r)) + \P_{\cF_s}(Y_\fB > \la).
$$
Upon taking expectations of both sides, we get that
$$
\P (\fB \in A(\bm, r)) \le e^{c_k \log \la} \P(\fM \in A(\bm, r)) + \P(Y_\fB > \la).
$$ 
It remains to bound $\P(Y_\fB > \la)$ and then optimize over $\la$. Define vectors $\bw^s, \bw^{-s}$ where 
$$
w^{\pm s}_i = \fA_i(\pm s) + \sup_{r \in [-s, s]} |\fA_{k+1}(r) - \fA_{k+1}(s)| + (k + 1 - i)s.
$$
By Lemma \ref{L:CH-monotonicity}, on the interval $[-s, s]$ $\fB^k_s$ is stochastically dominated by $k$ independent Brownian bridges $B = (B_1, \dots, B_k)$ from $(-s, \bw^{-s})$ to $(s, \bw^s)$ conditioned on the event
$$
\NI(\fA_{k+1}, [-s, 0] \cup [t, s]).
$$
Now, by the construction of $\bw^{\pm s}$, the probability that $k$ independent Brownian bridges $B = (B_1, \dots, B_k)$ from $(-s, \bw^{-s})$ to $(s, \bw^s)$ satisfy $B \in \NI(\fA_{k+1}, [-s, 0] \cup [t, s])$ is bounded below by $e^{-c_k}$. Therefore
\begin{align*}
\P(Y_\fB > \la) &\le \P(Y_B > \la) \\
&\le \P(\fA_1(s) \vee \fA_1(-s) - \fA_{k+1}(-s) \wedge \fA_{k+1}(s) + \sup_{r \in [-s, s]} |\fA_{k+1}(r) - \fA_{k+1}(s)| \ge \la/2 - 3k s) \\
&+ e^{c_k} \P(\sup_{r \in [-s, s]} W(s) > \la/2),
\end{align*}
where $W$ is an independent Brownian bridge from $(-s, 0)$ to $(-s, 0)$. We can bound the first of these terms with Lemma \ref{L:one-point-upper-bound}, Lemma \ref{L:one-point-lower-bound}, and Lemma \ref{L:MD-bounds}. The second of these terms can be bounded with a standard Brownian bridge bound. Altogether, we get that both of terms are of size $e^{-d_k \la^{3/2}}$ as long as $\la > c_k s^3$. Therefore setting $\la = c_k (s^3 + \|\bm\|_2)$ and applying the bound \eqref{E:fMI}, we get  
$$
e^{c_k \log \la} \P(\fM \in A(\bm, r)) + \P(M_\fB > \la) \le \exp\Big(c_k \log (\|\bm\|_2 + s)  -\frac{4}{3} \sum_{i=1}^k m_i^{3/2}\Big).
$$
which gives the desired result.
\end{proof}

We can translate Lemma \ref{L:one-point-bound-fB} into a one-time bound on $\fL$.

\begin{lemma}
\label{L:L-bound} For any $r \notin [0, t]$ and $\bm \in [0, \infty)^k_\ge$ we have
$$
\P(\fL \in A(\bm, r)) \le \exp\Big(-\frac{4}{3} \sum_{i=1}^k m_i^{3/2} + c_k \sqrt{t} \Big(\sum_{i=1}^k m_i^{3/2}\Big)^{5/6}\Big).
$$
	\end{lemma}

\begin{proof}
	First, we may assume $\sum_{i=1}^k m_i^{3/2} > c_k' t^3$ for a large $c_k'$ since otherwise the bound is trivial. We can choose $c_k'$ arbitrarily large to wash out other constants in the proof; this only affects the constant $c_k$ in the lemma statement.
We use the notation introduced prior to Lemma \ref{L:eta-eta-s}. For every $s > t$, we have that
\begin{equation}
	\label{E:L-A-M}
\P(\fL \in A(\bm, r)) = \frac{\eta_s(A(\bm, r))}{\eta(\sC^\N(\R))} = \frac{1}{\eta(\sC^\N(\R))} \int_{A(\bm, r)} \frac{d \eta_s}{d \tilde \eta_s}(f) d \tilde \eta_s(f).
\end{equation}
Now, the density $\frac{d \eta_s}{d \tilde \eta_s}$ can be thought of as an $\R$-valued random variable whose domain is the Borel space $\sC^\N(\R)$ with probability measure $\tilde \eta_s$. When thought of this way, it is equal in distribution to
$$
X :=\frac{1}{\E_{\cF_s} [\P_{0, t}(\fB_s^k(0), \fB_s^k(t), \fA_{k+1})]}. 
$$
We use this along with Corollary \ref{C:bounded-expectation} to bound \eqref{E:L-A-M}. Let $S(y) = \P(X_s > y)$ be the survival function for $X$, set $S^{-1}(x) = \inf \{y \in \R: S(y) < x\}$, and let $\beta = \eta_s (A(\bm, r))$. Let $\tilde S(y) = \exp(c_k s^3 - d_k \log y \sqrt{s/t})$ be the upper bound for $S$ from Corollary \ref{C:bounded-expectation}. Then
\begin{align*}
\int_{A(\bm, r)} \frac{d \eta_s}{d \tilde \eta_s}(f) d \tilde \eta_s(f) &\le \int_{S^{-1}(\beta)}^\infty S(y) dy \le \int_{\tilde S^{-1}(\beta)}^\infty \tilde S(y) dy \\
&= \exp(c_k d_k^{-1} s^{5/2} t^{1/2})\frac{\beta^{1 - d_k^{-1}\sqrt{t/s}}}{1 - d_k\sqrt{s/t}}
\end{align*}
Note that the final equality is an explicit computation, i.e. the constants $c_k, d_k$ that appear are those in the definition of $\tilde S$.
%
%
%
%
Therefore by the bound on $\beta$ in Lemma \ref{L:one-point-bound-fB}, we have that 
$$
\P(\fL \in A(\bm, r)) \le \exp\Bigg(c_k d_k^{-1} s^{5/2} \sqrt{t} -(1 -  d_k^{-1} \sqrt{t/s}) \Bigg[ \frac{4}{3} \sum_{i=1}^k m_i^{3/2} - c_k \log (\|\bm\|_2 + s) \Bigg] \Bigg).
$$
We can now minimize over $s$. This occurs roughly when $s = (\sum_{i=1}^k m_i^{3/2})^{1/3}$ (which is greater than $t$ by our assumption), which gives the desired bound after simplification, assuming $c_k'$ was chosen sufficiently large.
\end{proof}

\subsection{The two-time bound}
\label{SS:two-time-bound}

The proof of a sharp two-time bound is a much more refined version of the argument used Lemma \ref{L:two-point-estimate-weak} that takes into account the parabola. The argument is also applied to the ensemble $\fL$ rather than directly to $\fA$.  The first step is to achieve a lower bound on the probability that $\fL_{k+1}$ is small. This is achieved via the following lemma.

\begin{lemma}
	\label{L:RN-lines-bd}
Let $S = \II{k+1, \infty} \X \R$. Then the law of $\fL|_S$ is absolutely continuous with respect to the law of $\fA|_S$ and for any Borel set $A$ with $\P(\fA|_S \in A) = \ep > 0$ have
$$
\P(\fL|_S \in A) \le \ep e^{c_k \sqrt{t} (\log^{5/6}(\ep^{-1})))}
$$
\end{lemma}

\begin{proof}
First, we may assume $\log^{1/3}(\ep^{-1}) > c_k t$ for a large constant $c_k$, since otherwise the bound is trivial.
Letting $\fB = \fB_s$ for some $s > c_k t$ we have $\fB_s|_S = \fA|_S$. Moreover, as in the proof of Lemma \ref{L:L-bound} we have
$$
\P(\fL|_S \in A) \le [\P(\fB_s \in A)]^{1 - c_k' \sqrt{t/s}} e^{c_k s^{5/2} \sqrt{t}},
$$
and minimizing at $s = c_k (\log \P(\fB_s \in A)^{-1})^{1/3}$ (which is greater than $c_k t$ by our lower bound on $\ep$), gives the result.
\end{proof}

\begin{corollary}
	\label{C:RN-lines-mod-cont}
Fix an interval $[-\la, \la]$ for $\la \ge 2$. Then for $m > c_k t$ we have
$$
\P(\min_{r \in [-\la, \la]} \fL_{k+1}(r) + r^2 \le - m) \le \la e^{-d_k m^3}.
$$

\end{corollary}

\begin{proof}
Let $\tilde \fA(r) = \fA(r) + r^2$. First, by Lemma \ref{L:one-point-lower-bound} and a union bound we have that 
$$
\P(\min_{r \in [-\la - \ep, \la], r \in \ep \Z} \tilde \fA_{k+1}(r) \le - m + 1)  \le 2 \la \ep^{-1} e^{-d_k m^3},
$$
and by Lemma \ref{L:two-point-estimate-weak}, Lemma \ref{L:mod-cont}, and a union bound
$$
\P(\min_{r \in [-\la - \ep, \la], r \in \ep \Z, \de \in [0, \ep]} |\tilde \fA_{k+1}(r + \de) - \tilde \fA_{k+1}(r)| \le 1)  \le c_k \la \ep^{-1} e^{-d_k/\ep}.
$$
Combining these bounds with $\ep = m^{-3}$ gives that
$$
\P(\min_{r \in [-\la, \la]} \fA_{k+1}(r) + r^2 \le - m) \le c_k \la m^6 e^{-d_k m^3} \le \la e^{-d_k' m^3}.
$$
Applying Lemma \ref{L:RN-lines-bd} with $\ep = \la e^{-d_k' m^3}$ and simplifying yields the result.
\end{proof}

Lemma \ref{L:L-bound} and Corollary \ref{C:RN-lines-mod-cont} give the two bounds in \eqref{E:one-pt-intro}.

One of the main technical ingredients we need to achieve a two-point bound is a probability estimate for a Brownian bridge avoiding a parabola. 

\begin{lemma}
	\label{L:parabola-lemma-A}
	Fix  $\la \ge 1$, and for $x, y \ge 0$ with $\sqrt{x} < \la/2, \sqrt{y} < \la/2$, define 
	\begin{align*}
E(x, y, \la) &= \frac{2}{3} x^{3/2} + (\la - \sqrt{y})^2 \sqrt{y} + \frac{1}{3} (\la - \sqrt{y})^3, \quad \mathand\\
J(x, y, \la) &= E(x, y, \la) - \frac{(x - y + \la^2)^2}{4\la}.
	\end{align*}
	
	Next let $\al \in \R$, and let $f(s) = - s^2 + \al s$.
	Then we have the following bounds.
	\begin{enumerate}
		\item For $\la > 1$, any $x, y > 0$ with $\sqrt{x} < \la/2, \sqrt{y} < \la/2$, and any $t \in [0, \sqrt{x} + 1]$, we have
		$$
		\P_{0, \la}(x, y + f(\la), f, [t, \la]) \le \exp(-J(x, y, \la) +c\la \log (1 + \la)).
		$$
		\item Let $\la > 1, x, y > 0$ with $\sqrt{x} < \la/2, \sqrt{y} < \la/2$, and let $t \in (0, \sqrt{x} + 1], a \in (0, x]$.  If $B$ is a Brownian bridge from $(0, x)$ to $(\la, f(\la) + y)$, then
		\begin{align*}
		&\P(B(s) > f(s) \text{ for all } s \in [t, \la], B(t) \le x - a) \le \\
		&\exp \Big(-J(x, y, \la) -\frac{a^2}{4t} + \frac{2}{3} x^{3/2} - \frac{2}{3}(x - a)^{3/2} + c(t^3 + tx + \la \log(1 + \la)) \Big).
		\end{align*}
		\item Let $\bx, \by \in [1,\infty)^k_>$ with $x_i - x_{i+1} \ge 1$ and $y_i - y_{i+1} \ge 1$ for all $i \in \II{1, k-1}$, and suppose that $\sqrt{x_1} < \la/2, \sqrt{y_1} < \la/2$. Then
		$$
		\P_{0, \la}(\bx, \by + f(\la), f) \ge \exp\Big(-\sum_{i=1}^k J(x_i, y_i, \la) -c k^2\la\Big).
		$$
	\end{enumerate}
\end{lemma}

The proof of Lemma \ref{L:parabola-lemma-A} is essentially a standard Dirichlet energy computation. We sketch the idea here, and leave the details to the appendix. 

First, we can set $\al = 0$, as Brownian bridge law commutes with affine shifts. This simplifies the computations. Next, the cost for a Brownian bridge $B$ from $(0,x)$ to $(\la, - \la^2 + y)$ to roughly follow a path $g:[0, \la] \to \R$ with $g(0) = x, g(\la) = - \la^2 + y$ is the exponentiated Dirichlet energy 
$$
\exp \lf(-E(g) + \frac{(x - y + \la^2)^2}{4\la} \rg), \qquad \text{ where } \quad E(g) = \frac{1}{4} \int_0^\la |g'(s)|^2 dt.
$$
This follows, for example, from the Cameron-Martin theorem.
The factor of $1/4$ here (as opposed to the usual $1/2$) is due to the fact that we are working with variance $2$ bridges, and the added factor of $(x - y + \la^2)^2/(4\la)$ comes from conditioning a Brownian motion to end at the point $y - \la^2$ to give rise to a Brownian bridge. Hence we should look for a path $g$ of minimal Dirichlet energy that stays above the function $f(s) = -s^2$. This is the minimal concave function $g = g_{x, y}$ from $(0,x)$ to $(\la, - \la^2 + y)$ that dominates $f$. Since $f$ is just a parabola, the function $g$ is easily computed, and is determined by the following constraints:
 \begin{itemize}[nosep]
		\item $g(0) = 0, g(\la) = -\la^2 + y$, and there are values $0 \le a < b \le \la$ such that $g$ is linear on the intervals $[0, a]$ and $[b, \la]$. 
		\item $g(z) = f(z)$ for all $z \in (a, b)$.
		\item $g$ is differentiable at $a$ and $b$. This follows from forcing $g$ to be both concave and to dominate $f$.
	\end{itemize}
We can compute that $a = \sqrt{x}, b = \la - \sqrt{y}$, and that $E(g) = E(x, y, \la)$. This gives rise to the bound in Lemma \ref{L:parabola-lemma-A}.1, and the bound in part $3$; the cost of having the $k$ bridges stay ordered is a lower order term. The bound in part $2$ follows similarly to part $1$, by computing the path of minimal Dirichlet energy from $(0, x)$ to $(\la, - \la^2 + y)$ that dips below $x-a$ at time $t$.

We can now state the first version of the two-time bound for $\fL$.

\begin{prop}
	\label{P:two-point-bound-A}
	Let $\bm \in \R^k_\ge$, let $\ba \in [0, \infty)^k$ and let $\sig \in \{0, t\}^k$. Define $\hat \sig \in \{0, t\}^k$ to be the unique vector with $\hat \sig_i \ne \sig_i$ for all $i$. Next, define
	$$
	A(\bm, \ba, \sig) = \{f \in \sC^k(\R) : f(\sig_i) \ge m_i, \quad f_i(\sig_i) - f_i(\hat \sig_i) \ge a_i \text{ for all } i \in \II{1, k} \}.
	$$
	Then
	\begin{equation}
	\label{E:fLKT}
	\P(\fL^k \in A(\bm, \ba, \sig)) \le \exp \lf( -\sum_{i=1}^k \lf( \frac{a_i^2}{4t} + \frac{2}{3} (m_i^+)^{3/2} + \frac{2}{3} [(m_i - a_i)^+]^{3/2} \rg) + E(\ba, \bm, t) \rg),
	\end{equation}
	where
	$$
	E(\ba, \bm, t) = O_k(t^3 + \sqrt{t}\|\bm\|_2^{5/4} + t \|\ba\|_2 + t^{-1/6}\|\ba\|^{2/3}_2 \|\bm\|_2^{1/2}). 
	$$

\end{prop}

The proof of Proposition \ref{P:two-point-bound-A} is the most technical proof in the paper, and it is based on a backwards induction on $\ell \in \II{1, k}$. The next lemma gives the key inductive step.

\begin{lemma}
	\label{L:inductive-step}
In the setup of Proposition \ref{P:two-point-bound-A}, for $\ell \in \II{1, k + 1}$ define the sets
\begin{align*}
	A_\ell(\bm, \ba, \sig) = \{f \in \sC^k(\R) : f(\sig_i) \ge m_i, \quad f_i(\sig_i) - f_i(\hat \sig_i) \ge a_i \text{ for all } i \in \II{\ell, k} \}.
\end{align*}
and for $r = 0, t$ define
\begin{align*}
D^r_\ell(\bm) &= \{f \in \sC^k(\R) : f_i(r) \ge m_i \text{ for all } i \in \II{1, \ell} \}.
\end{align*}
Note that with these definitions, $A_{k+1}(\bm, \ba, \sig), D^0_0(\bm),$ and $D^t_0(\bm)$ are the whole space.
Then for all $\ell \in \II{1, k + 1}$ and $r = 0, t$ we have
\begin{equation}
\label{E:split-ma}
\begin{split}
&\P(\fL^k \in A_\ell(\bm, \ba, \sig) \cap D_{\ell - 1}^r(\bm)) \\
&\le \exp \lf( -\sum_{i=\ell}^k \lf( \frac{a_i^2}{4t} + \frac{2}{3} (m_i^+)^{3/2} + \frac{2}{3} [(m_i - a_i)^+]^{3/2} \rg) -\sum_{i=1}^{\ell - 1} \frac{4}{3} (m_i^+)^{3/2} + E(\ba, \bm, t) \rg).
\end{split}
\end{equation}
\end{lemma}

Proposition \ref{P:two-point-bound-A} is immediate from Lemma \ref{L:inductive-step} since 
$$
A_1(\bm, \ba, \sig) \cap D_{0}^r(\bm) = A(\bm, \ba, \sig).
$$
\begin{proof}
The proof is based on a downwards induction on $\ell$. In the base case when $\ell = k + 1$, since $A_{k+1}(\bm, \ba, \sig)$ is the whole space, the bound is equivalent to the claim that 
$$
\P(D^r_k(\bm)) \le \exp \lf( -\sum_{i=1}^{k} \frac{4}{3} (m_i^+)^{3/2} + E(\ba, \bm, t) \rg)
$$
for $r = 0, t$. This follows from Lemma \ref{L:L-bound}. Now assume $\ell \le k$, and that the lemma holds for $\ell +1$ for any choice of $\ba, \bm, \sig, r$. Without loss of generality, we assume that $\sig_\ell = 0$; the case when $\sig_\ell = t$ follows from the distributional equality
$$
\fL(r) + r^2 \eqd \fL(t -r) + (t-r)^2
$$ 
as functions in $\sC^\N(\R)$. This equality follows from the same result for $\fA$. There are two cases here, depending on whether $r = 0$ or $r = t$. We start with the simpler case when $r = 0$.

	\textbf{Step 1: Setting up a Gibbs resampling.} \qquad
	Let $\be_k > 0$ be a large $k$-dependent constant and set
	$$
	\ga = \ga(\ba, \bm, t) = \be_k (t^{3/2} + t^{-1/2} \|\ba\|_2 + \|\bm\|^{3/4}_2).
	$$
	How large we need to take $\be_k$ will be made clear in the proof.
		Let $\cF$ denote the $\sig$-algebra generated by $\fL$ outside of the set $S = \II{1, \ell} \X [0, \ga^{2/3}]$. The basic idea of the induction is a Gibbs resampling on the set $S$. To see how the structure will work, observe that
	\begin{align*}
	A_\ell(\bm, \ba, \sig) \cap D_{\ell - 1}^0(\bm) = A_{\ell + 1}(\bm, \ba, \sig) \cap D_{\ell}^0(\bm) \cap \{\fL_\ell(0) - \fL_\ell(t) > a_\ell\}.
	\end{align*}
	The set
	$
	A_{\ell + 1}(\bm, \ba, \sig) \cap D_{\ell}^0(\bm)
	$
	is $\cF$-measurable, so for any set $E$ we have
	\begin{align}
	\nonumber
\P_\cF(A_\ell&(\bm, \ba, \sig) \cap D_{\ell - 1}^0(\bm)) \\
	\nonumber
&= \indic(A_{\ell + 1}(\bm, \ba, \sig) \cap D_{\ell}^0(\bm)) \P_\cF(\fL_\ell(0) - \fL_\ell(t) > a_\ell) \\
\label{E:indic-A}
&\le \indic(A_{\ell + 1}(\bm, \ba, \sig) \cap D_{\ell}^0(\bm) \cap E) \P_\cF(\fL_\ell(0) - \fL_\ell(t) > a_\ell) + \indic(E^c).
	\end{align}
Now, for every $n = 0, 1, \dots$, let 
$$
\bm_n = (m_1 \vee (m_\ell + n), \dots, m_{\ell - 1} \vee (m_\ell + n), m_\ell + n, m_{\ell + 1}, \dots, m_k),
$$
and let $D_n = D_{\ell}^0(\bm_{n-1}) \smin D_{\ell}^0(\bm_n)$ so that $D_\ell^0(\bm) = \bigcup_{n=1}^\infty D_n$.
We can write
\begin{equation*}
\begin{split}
&\indic(A_{\ell + 1}(\bm, \ba, \sig) \cap D_{\ell}^0(\bm) \cap E) \P_\cF(\fL_\ell(0) - \fL_\ell(t) > a_\ell) \\
&\le \sum_{n=1}^\infty \indic(A_{\ell + 1}(\bm, \ba, \sig) \cap D_n) \| \P_\cF(\fL_\ell(0) - \fL_\ell(t) > a_\ell) \indic(E \cap D_n)\|_\infty,
\end{split}
\end{equation*}
Taking expectations in \eqref{E:indic-A} and applying the inductive hypothesis we get
\begin{equation}
	\label{E:Al-big}
	\begin{split}
&\P(A_\ell(\bm, \ba, \sig) \cap D_{\ell - 1}^0(\bm)) \le \P(E^c) + \sum_{n=1}^\infty \|\P_\cF(\fL_\ell(0) - \fL_\ell(t) > a_\ell) \indic(E \cap D_n)\|_\infty \X \\
&\X \exp \lf( -\sum_{i=\ell + 1}^k \lf( \frac{a_i^2}{4t} + \frac{2}{3} (m_i^+)^{3/2} + \frac{2}{3} [(m_i - a_i)^+]^{3/2} \rg) - \frac{4}{3} [(m_\ell + n - 1)^+]^{3/2} -\sum_{i=1}^{\ell - 1} \frac{4}{3} (m_i^+)^{3/2} +  E(\ba, \bm_n, t) \rg).
\end{split}
\end{equation}
To complete the proof we will find a set $E$ such that $\P(E^c)$ is bounded above by the right-hand side of \eqref{E:split-ma}, such that
\begin{equation}
	\label{E:Pfell-0-times}
\|\P_\cF(\fL_\ell(0) - \fL_\ell(t) > a_\ell) \indic(E \cap D_n)\|_\infty = 0, \qquad \text{for} \quad n \notin [a_\ell - m_\ell - 1 - \sqrt{\beta_k} \ga^{2/3}, \ga^{4/3}],
\end{equation}
and 
\begin{equation}
\label{E:PfLell-infty}
\begin{split}
\|\P_\cF&(\fL_\ell(0) - \fL_\ell(t) > a_\ell) \indic(E \cap D_n)\|_\infty \\
&\le \exp \Big( -\frac{a_\ell^2}{4t} + \frac{2}{3} [(m_\ell + n-1)^+]^{3/2} - \frac{2}{3} [(m_\ell - a_\ell)^+]^{3/2} + E(\ba, \bm_n, t)\Big).
\end{split}
\end{equation}
To see why these bounds complete the proof, observe that by \eqref{E:Pfell-0-times}, the sum in \eqref{E:Al-big} has at most $\ga^{4/3}$ non-zero terms, so it suffices to bound each of these terms by the right-hand side of \eqref{E:split-ma} since $\log \ga^{4/3} \le E(\ba, \bm, t)$. Second, if either $n \le c_k m_\ell$ or $n \le c_k t^2$, then $E(\ba, \bm_n, t) \le E(\ba, \bm, t)$ and so by \eqref{E:PfLell-infty} the $n$th summand in \eqref{E:Al-big} is bounded by the right-hand side of \eqref{E:split-ma}. If both $n \ge c_k |m_\ell|$ and $n \ge c_k t^2$, then by the lower bound on $n$ in \eqref{E:Pfell-0-times} we also have that $n \ge d_k a_\ell$ for some small constant $d_k$. These constraints together imply that $E(\ba, \bm_n, t) = E(\ba, \bm, t) + O_k(\sqrt{t}n^{5/4})$. On the other hand, using again that $n \ge c_k |m_\ell|$, we have that 
$$
- \frac{2}{3} [(m_\ell + n - 1)^+]^{3/2} \le - \frac{2}{3} (m_\ell^+)^{3/2} - \frac{1}{2}n^{3/2}.
$$
Since $O_k(\sqrt{t}n^{5/4}) - \tfrac{1}2 n^{3/2} \le c_k' t^3 \le E(\ba, \bm, t)$, the $n$th summand in \eqref{E:Al-big} is bounded by the right-hand side of \eqref{E:split-ma} in this case as well.

\textbf{Step 2: Defining $E$ and comparing to a simpler ensemble.} \qquad Let $E_r$ be the set where:
	\begin{itemize}[nosep]
		\item For all $s \in [-\ga^{2/3}, \ga^{2/3}]$, we have 
		$
		\fL_{\ell+1}(s) + s^2 > - \sqrt{\be_k} \ga^{2/3}.
		$
		\item We have the inequalities
		\begin{align*}
		- 0.9 \sqrt{\be_k} \ga^{2/3} &\le \fL_\ell(r) + r^2 \le \fL_1(r) + r^2 \le \frac{1}{10}\ga^{4/3} \qquad \text{ and } \\
		- 0.9 \sqrt{\be_k} \ga^{2/3} &\le \fL_\ell(\pm \ga^{2/3}) + \ga^{4/3} \le \fL_1(\pm \ga^{-2/3}) + \ga^{4/3} \le \frac{1}{10}\ga^{4/3},
		\end{align*}
	\end{itemize}
and set $E = E_0$ (in the final step of the proof we will also use the set $E_t$). 
	By Lemma \ref{L:L-bound}, Corollary \ref{C:RN-lines-mod-cont} and a union bound, as long as the constant $\be_k$ is large enough we have that 
	\begin{equation}
	\label{E:PEC}
		\P(E^c) \le e^{- \ga^2},
	\end{equation}
	which is bounded above by the right-hand side of \eqref{E:split-ma} for $\be_k$ large enough. Moreover, \eqref{E:Pfell-0-times} follows from the constraints imposed on $\fL_{\ell + 1}, \fL_\ell$ on $E$.
	
	Now let $\bar \iota = (0, 1, \dots, \ell-1)$, and let 
	$B = (B_1, \dots, B_\ell)$ be a $k$-tuple of independent Brownian bridges from $(0, \fL^\ell(0) - \bar \iota)$ to $(\ga^{2/3}, \fL^\ell(\ga^{2/3}) - \bar \iota)$. Let $f:[0, \ga^{2/3}] \to \R$ be given by $f(s) = -s^2 - \sqrt{\be_k} \ga^{2/3}$, and let $\tilde B$ denote $B$ conditioned on the non-intersection event $\NI(f; [t, \ga^{2/3}])$.
	
	By Lemma \ref{L:CH-monotonicity} and the Gibbs property for $\fL$ on the set $S$, on the event $E$ the bottom line $\fL_\ell|_{[0, \ga^{2/3}]}$ stochastically dominates $\tilde B_\ell$. Therefore on $E$: 
	\begin{align*}
	\P_\cF(\fL_\ell(0) - \fL_\ell(t) > a_\ell) &\le \P_\cF(\fL_\ell(0) - \tilde B_\ell(t) > a_\ell) \\
	&=\P_\cF(\tilde B_\ell(0) - \tilde B_\ell(t) > a_\ell - (\ell - 1)).
	\end{align*}
	Now,
	\begin{equation}
	\label{E:B-frac}
	\P_\cF(\tilde B_\ell(0) - \tilde B_\ell(t) > a_\ell - (\ell - 1)) = \frac{\P_\cF(B \in \NI(f; [t, \ga^{2/3}]), B_\ell(t) < B_\ell(0) - a_\ell + (\ell - 1))}{\P_{0, \ga^{2/3}}(B(0),  B(\ga^{2/3}), f; [t, \ga^{2/3}])},
	\end{equation}
	and we can bound the numerator above and denominator below on the event $E_0$ using Lemma \ref{L:parabola-lemma-A}.\\
	
	\textbf{Step 3: Bounding \eqref{E:B-frac}.} \qquad Starting with the denominator in \eqref{E:B-frac}, we have that
	\begin{align*}
\P_{0, \ga^{2/3}}(B(0),  B(\ga^{2/3}), f; [t, \ga^{2/3}]) &\ge \P_{0, \ga^{2/3}}(B(0), B(\ga^{2/3}), f) \\
&=\P_{0, \ga^{2/3}}(B(0) + \sqrt{\be_k} \ga^{2/3}, B(\ga^{2/3}) + \sqrt{\be_k} \ga^{2/3},  g),
	\end{align*}
	where $g(s) = -s^2$.
Moreover, on $E$ we have that 
$$
\max (\fL_1(0) + \sqrt{\be_k} \ga^{2/3}, \fL_1(\ga^{2/3}) + \sqrt{\be_k} \ga^{2/3} - g(\ga^{2/3})) \le \frac{1}{5} \ga^{4/3}
$$
as long as $\be_k$ was chosen large enough. Similarly, 
$$
\fL_\ell(0) + \sqrt{\be_k} \ga^{2/3} \ge 2k, \quad \fL_\ell(\ga^{2/3}) + \sqrt{\be_k} \ga^{2/3} - g(\ga^{2/3})) \ge 2k.
$$ 
Therefore we are in the setting of Lemma \ref{L:parabola-lemma-A}.3, and so
\begin{align*}
&\P_{0, \ga^{2/3}}(B(0) + \sqrt{\be_k} \ga^{2/3}, B(\ga^{2/3}) + \sqrt{\be_k} \ga^{2/3},  g) \\
&\ge \exp\Big(-\sum_{i=1}^\ell J(B_i(0) + \sqrt{\be_k} \ga^{2/3}, B_i(\ga^{2/3}) + \sqrt{\be_k} \ga^{2/3} + \ga^{4/3}, \ga^{2/3}) -c \ell^2 \ga^{2/3} \Big).
\end{align*}
Turning to the numerator, we have that
\begin{align*}
\P_\cF(B \in \NI(f; [t, \ga^{2/3}]), &B_\ell(t) < B_\ell(0) - a_\ell + (\ell - 1)) \\
&\le \P_\cF(B_\ell \in \NI(f; [t, \ga^{2/3}]), B_\ell(t) < B_\ell(0) - a_\ell + (\ell - 1))
\\
&\X\prod_{i=1}^{\ell - 1} \P_{0, \ga^{2/3}}(B_i(0) + \sqrt{\be_k} \ga^{2/3}, B_i(\ga^{2/3}) + \sqrt{\be_k} \ga^{2/3},  g; [t, \ga^{2/3}])
\end{align*}
The product of $\ell - 1$ terms above can be bounded above by Lemma \ref{L:parabola-lemma-A}.1. Note that $t \le 0.1 \sqrt{\be_k} \ga^{2/3} \le B_i(0) + \sqrt{\be_k} \ga^{2/3}$ by the definition of $\ga$ as long as $\be_k$ is large enough. The remaining $\ell$-term can be bounded using Lemma \ref{L:parabola-lemma-A}.2. Together we get that the right-hand side above is bounded above by
\begin{align*}
\exp\Big(-&\sum_{i=1}^\ell J(B_i(0) + \sqrt{\be_k} \ga^{2/3}, B_i(\ga^{2/3}) + \sqrt{\be_k} \ga^{2/3} + \ga^{4/3}, \ga^{2/3}) +c_k (t\ga^{2/3} \log(\ga) + t^3 + t |B_\ell(0)|)\\
& - \frac{(a_\ell - \ell + 1)^2}{4t} + \frac{2}{3}(B_\ell(0) + \sqrt{\be_k}\ga^{2/3})^{3/2} - \frac{2}{3}(B_\ell(0) + \sqrt{\be_k}\ga^{2/3} - a_\ell + \ell - 1)^{3/2}
\Big).
\end{align*}
Putting together the two bounds in this section, and using that $B_\ell(0) = \fL_\ell(0) - (\ell -1)$ and that $\ga \gg k$ gives that \eqref{E:B-frac} is bounded above by
\begin{equation}
\label{E:halfway-house}
\begin{split}
\exp \Big(- \frac{a_\ell^2}{4t}& + \frac{2}{3}[\fL_\ell(0)^+]^{3/2} - \frac{2}{3}[(\fL_\ell(0) - a_\ell)^+]^{3/2} \\
+O_k&(t\ga^{2/3} \log (\ga) + t^3 + t |\fL_\ell(0)| + \frac{a_\ell}{t} + \ga^{2/3} \sqrt{|\fL_\ell(0)|} + \ga)\Big).
\end{split}
\end{equation}
Now, on the set $E \cap D_n$, we know that $\fL_\ell(0) \in [m_\ell + n - 1, m_\ell + n]$. Applying this, and simplifying the error, we get that \eqref{E:halfway-house} is bounded above by
\begin{align*}
\exp &\Big(- \frac{a_\ell^2}{4t} + \frac{2}{3}[(m_\ell + n-1)^+]^{3/2} - \frac{2}{3}[(m_\ell + n - a_\ell)^+]^{3/2}
+E(\ba, \bm_n, t)\Big),
\end{align*}
implying the desired bound in \eqref{E:PfLell-infty}.

\textbf{Step 4: Completing the proof when $\sig_\ell = 0, r = t$.}
First, with $E_t$ as defined in Step 2 above, by noting that $\P (E_t^c)$ satisfies the same bound as $\P(E^c)$ in \eqref{E:PEC}, it is enough to bound
$$
\P (A_\ell(\bm, \ba, \sig) \cap D_{\ell - 1}^t(\bm) \cap E_t).
$$
Our strategy here is to bootstrap from the previous case, by using that the Gibbs property implies that it is unlikely for the lines in $\fL$ to be large at time $t$ without being large at time $0$.  Let $\cG$ be the $\sig$-algebra generated by $\fL$ outside of the set $\II{1, \ell - 1} \X [-\ga^{2/3}, t]$, and suppose that we can find $\bm' \le \bm$ with $\P_\cG(D_{\ell - 1}^0(\bm')) \ge 1/2$ almost surely on the set $D_{\ell - 1}^t(\bm) \cap E_t$. Then
\begin{align*}
\P(A_\ell(\bm, \ba, \sig) \cap D_{\ell - 1}^t(\bm) \cap E_t) &\le 2 \E[\P_\cG(D_{\ell - 1}^0(\bm')) \indic(A_\ell(\bm, \ba, \sig) \cap D_{\ell - 1}^t(\bm) \cap E_t)] \\
&\le 2 \E[\P_\cG(D_{\ell - 1}^0(\bm')) \indic(A_\ell(\bm', \ba, \sig)] \\
&= 2 \P(D_{\ell - 1}^0(\bm') \cap A_\ell(\bm', \ba, \sig)).
\end{align*}
Here the second inequality uses that $A_\ell(\bm, \ba, \sig)  \subset A_\ell(\bm', \ba, \sig)$ since $\bm' \le \bm$, and the final equality uses that $A_\ell(\bm', \ba, \sig)$ is $\cG$-measurable.

Now, on $D_{\ell - 1}^t(\bm) \cap E_t$ and conditional on $\cG$, by Lemma \ref{L:CH-monotonicity} the vector $\fL^k(0)$ stochastically dominates an $(\ell - 1)$-tuple of Brownian bridges $B$ from $(- 2\ga^{4/3} - 1, \dots, - 2 \ga^{4/3} - \ell)$ at time $-\ga^{2/3}$ to $(m_1 -1, \dots, m_\ell - \ell)$ at time $t$, conditioned not to intersect each other. The probability of non-intersection is bounded below by
$
\exp(- ck \log \ga)
$
by Lemma \ref{L:nonint-probability},
so
$
\P(\fL^k(0) \ge \bm') \ge 1/2
$
for any $\bm'$ for which $
\P(B(0) \ge \bm') \ge 1 - e^{-c k \log \ga}/2.
$
This holds with $\bm' = \bm - c_k \sqrt{t} \ga^{2/3}$, and so
\begin{align*}
\P(D_{\ell - 1}^t(t) \cap A_\ell(\bm, \ba, \sig) \cap E_t) \le 2 \P(D_{\ell - 1}^0(\bm - c_k \sqrt{t} \ga^{2/3}) \cap A_\ell(\bm - c_k \sqrt{t} \ga^{2/3}, \ba, \sig)),
\end{align*}
which satisfies the same upper bound as $\P(D_{\ell - 1}^0(\bm) \cap A_\ell(\bm, \ba, \sig))$, up to a term that gets folded into the error.
\end{proof}

Proposition \ref{P:two-point-bound-A} implies the following simpler two-point bound, of which inequality \eqref{E:two-pt-intro} is a special case.

\begin{corollary}
	\label{C:two-point-bound}
	Let $\ba \in [0, \infty)^k$ with $\|\ba\|_2 \ge c_k t^2$ and let $\sig \in \{0, t\}^k$. Let $\hat \sig \in \{0, t\}^k$ be the unique vector with $\hat \sig_i \ne \sig_i$ for all $i$. Next, let $f(\ba, \sig) \in \sC^k_0([0, t])$ be the unique $k$-tuple of linear functions satisfying $f_i(\sig_i) - f_i(\hat \sig_i) = a_i$, and recall the definition of the map $\S$ introduced for Theorem \ref{T:tetris-theorem}. Then 
\begin{equation}
\label{E:fLKTT}
\begin{split}
&\P(\fL_i(\sig_i) - \fL_i(\hat \sig_i) \ge a_i \text{ for all } i \in \II{1, k}) \\
&\le \exp \lf( - \S (f(\ba, \sig)) -\sum_{i=1}^k \frac{a_i^2}{4t} + O_k(\sqrt{t}\|\ba\|_2^{5/4})\rg)
\end{split}
\end{equation}
\end{corollary}

\begin{proof}
Without loss of generality, assume $\sig_k = 0$. Then using the notation $\T$ for the tetris map introduced after Theorem \ref{T:tetris-theorem} we have that
$$
\fL_k(t) \ge - \al \qquad \implies \qquad \fL(0) \ge \T f(0) - \al, \qquad \fL(t) \ge \T f(t) - \al,
$$
since the vectors $\fL^k(0), \fL^k(t)$ are ordered.
Therefore for any $\al > 0$, with notation as in Proposition \ref{P:two-point-bound-A} we have
$$
\P(\fL_i(\sig_i) - \fL_i(\hat \sig_i) \ge a_i \forall i \in \II{1, k}) \le \P(\fL \in A( \T f(0) \vee \T f(t) - \al, \ba, \sig)) + \P(\fL_k(t) \le - \al).
$$
Setting $\al = c_k(t^2 + t^{-1/3} \|\ba\|_2^{2/3})$, by Corollary \ref{C:RN-lines-mod-cont}, we have that $\P(\fL_k(t) \le - \al)$ is bounded above by the right-hand side of \eqref{E:fLKTT} for large enough $c_k$. This uses that $T f_1 \le c_k \|\ba \|_2$ and that $\|\ba\|_2 \ge t^2$. For this choice of $\al$, the left-hand side is similarly bounded via Proposition \ref{P:two-point-bound-A}. We have simplified the error by using that $T f_1 \le c_k \|\ba \|_2$ and that $\|\ba\|_2 \ge t^2$.
\end{proof}

	\section{The upper bound in Theorem \ref{T:tetris-theorem}}
	\label{S:upper-bound}
	In this section we prove the upper bound in Theorem \ref{T:tetris-theorem}. Throughout the section we fix $k \in \N, t \ge 1$. For this, it is enough to prove that the upper bound holds for $\fA^k|_{[1, t + 1]} - \fA^k(1)$, since in law $\fA^k|_{[1, t + 1]} - \fA^k(1)$ and $\fA^k|_{[0, t]} - \fA^k(0)$ only differ by a linear shift which can be wrapped into the error. Moreover, we can work with $\fL = \fL^{t + 2, k}$ rather than $\fA$, since the conditional probability that $\fL$ is non-intersecting--$e^{-O_k(t^3)}$--can again be wrapped into the error term in Theorem \ref{T:tetris-theorem}. The key step is to convert Proposition \ref{P:two-point-bound-A} into a density estimate.
	
	For this lemma, we define a $\sig$-finite measure $\mu$ on $X = \R^k \X \R^k \X \sC^k([1, t + 1])$ as follows. First let $\mu'$ denote the product measure on $X$ given by Lebesgue measure on $\R^k \X \R^k$ and the law of $k$ independent Brownian bridges from $(1, 0)$ to $(t + 1, 0)$ on $\sC^k([1, t + 1])$. Let $\mu$ denote the pushforward of $\mu'$ under the map
	$$
	(\bx, \by, f) \mapsto (\bx, \by, f + L_{\bx, \by}),
	$$
	where $L_{\bx, \by}$ is the function which is linear in each coordinate and satisfies $L(1) = \bx, L(t + 1) = \by)$. Then the law of $(\fA^k(1), \fA^k(t + 1), \fA^k|_{[1, t + 1]})$ is absolutely continuous with respect to $\mu$. The next lemma bounds its density.
	\begin{lemma}
		\label{L:density-estimate}	
	Let $Y = Y(\bx, \by, g)$ denote the density on $\R^k \X \R^k \X \sC^k([1, t+1])$ of $(\fA^k(1), \fA^k(t + 1), \fA^k|_{[1, t + 1]})$ against the measure $\mu$. Then we have the following bounds:
	\begin{enumerate}
		\item For all $(\bx, \by, g) \in \R^k \X \R^k \X \sC^k([1, t +1])$ we have 
		$$
		Y(\bx, \by, g) \le \exp(-d_k |\inf g_k|^3 \indic(\inf g_k \le -c_k t^2) + c_k t^3).
		$$
		\item For all $(\bx, \by, g) \in \R^k \X \R^k \X \sC^k([1, t+ 1])$ we have
		$$
		Y(\bx, \by, g) \le \exp\lf(- \lf(\frac{\|\bx - \by\|^2_2}{4t} +\sum_{i=1}^k \frac{2}{3} [(x_i^+)^{3/2} + (y_i^+)^{3/2}] \rg) + O_k(t^3 + \sqrt{t}(\|\bx\|_2^{5/4} + \|\by\|_2^{5/4}) )\rg).
		$$
	\end{enumerate}
	\end{lemma}

\begin{proof}
	Throughout, we work with $\fL$ conditional on the event $\fL_1 > \dots > \fL_k > \fL_{k+1}$, which is equal in law to $\fA$. Note that this event has probability $e^{-O_k(t^3)}$.
	
	\textbf{Bound 1:} \qquad	First, since $\fL^k|_{[0, t + 2]}$ given $\fL^k(0), \fL^k(t+2)$ is simply given by $k$ independent Brownian bridges, the Lebesgue density of $(\fL^k(1), \fL^k(t + 1))$ is bounded above by a constant $c_k > 0$. Moreover, $\fL^k|_{[1, t + 1]}$ has the law of $k$ independent Brownian bridges from $(1, \fL^k(1))$ to $(t + 1, \fL^k(t + 1))$ given $(\fL^k(1), \fL^k(t + 1))$. Therefore the density of $(\fL^k(1), \fL^k(t + 1), \fL^k|_{[1, t + 1]})$ against $\mu$ is uniformly bounded by a constant $c_k$, and so $Y$ is a.s.\ by $e^{O_k(t^3)}$. 
	
	Note that this density bound also holds a.s.\ if we condition on any information about $\fL$ outside of $\II{1, k} \X [0, t + 2]$. In particular, if we let $Y_\al$ be the conditional density of $(\fA^k(1), \fA^k(t + 1), \fA^k|_{[1, t + 1]})$ against $\mu$ on the event $A_\al = \{\inf \fL_{k+1}|_{[0, t + 2]} \le - \al\}$ and $Z_\al$ the conditional density on $A_\al^c$ we have
\begin{align*}
Y(\bx, \by, g) &= Y_\al(\bx, \by, g)\P(A_\al) + Z_\al(\bx, \by, g)\P(A^c_\al) \\
&\le e^{c_k t^3} \P(A_\al) + Z_\al(\bx, \by, g).
\end{align*}
If $-\al > \inf g_k$, then $Z_\al(\bx, \by, g) = 0$ by the ordering of the lines in $\fA$. Moreover, $\P(A_\al) \le (t + 2) e^{-d_k \al^3}$ by Corollary \ref{C:RN-lines-mod-cont} for $\al \ge c_k t^2$, which yields the desired bound when $\inf g_k \le-c_k t^2$.

	\textbf{Bound 2:} \qquad Let $f$ be the Lebesgue density of $(\fL^k(1), \fL^k(t + 1))$. It is enough to prove the second bound with $Y(\bx, \by, g)$ replaced by $f(\bx, \by)$, since the $e^{O(t^3)}$ cost of conditioning $\fL$ to be non-intersecting folds into the error, and $\fL^k|_{[1, t + 1]}$ has the law of $k$ independent Brownian bridges from $(1, \fL^k(1))$ to $(t + 1, \fL^k(t + 1))$ given $(\fL^k(1), \fL^k(t + 1))$.	
	
	For $\bm, \bn \in \Z^k_\ge$, let $f(\bx, \by, \bm, \bn)$ be the conditional Lebesgue density of $(\fL^k(1), \fL^k(t + 1))$ at $(\bx, \by)$ on the event
$$
C(\bm, \bn) = \{\fl{\fL^k(0)} = \bm, \fl{\fL^k(t + 2)} = \bn\}.
$$
Since $\fL^k|_{[0, t + 2]}$ is simply a $k$-tuple of Brownian bridges from $\fL^k(0)$ to $\fL^k(t + 2)$, we have that  $f(\bx, \by, \bm, \bn)$ is bounded above by
$$
c_k \exp \lf(-\frac{\|\bx - \by\|^2_2}{4t} - \frac{\|\bx - \bm\|^2_2}{4} - \frac{\|\by - \bn\|^2_2}{4} + \frac{\|\bm - \bn\|^2_2}{4(t + 2)} + \frac{\sum_{i=1}^k|x_i - m_i| + |y_i - n_i| + |n_i - m_i|}{4} \rg).
$$
The final error term comes from the fact that $\fL^k(0) \in [\bm, \bm + 1]$, and $\fL^k(t) \in [\bn, \bn + 1]$. Next, from Proposition \ref{P:two-point-bound-A}, we have the bound
\begin{align*}
\P(C(\bm, \bn)) \le \exp \lf( - \frac{\|\bm - \bn\|^2_2}{4(t + 2)} -\sum_{i=1}^k \lf(\frac{2}{3} [(m_i^+)^{3/2} + (n_i^+)^{3/2}] \rg) + O_k(\sqrt{t}(\|\bm\|_2^{5/4} + \|\bn\|_2^{5/4}) + t^3) \rg)
\end{align*}
and by Corollary \ref{C:RN-lines-mod-cont} we have
$
\P(\bigcup_{m_k \wedge n_k \le -\al} C(\bm, \bn)) \le 2\exp (-d_k \al^3)
$
for $\al > c_k' t^2$. 
Therefore for $\al > c_k' t^2$, we have
\begin{align}
\label{E:exp-frack}
f(\bx, \by) \le c_k \exp (-d_k \al^3) +&\sum_{\substack{\bm, \bn \in \Z^k_\ge, \\ n_k \wedge m_k \ge -\al}} \exp(-D(\bn, \bm))
\end{align}
where 
$$
D(\bn, \bm) = \frac{\|\bx - \by\|^2}{4t} + O_k(t^3 + \|\bx\|_2 + \|\by\|_2) + \sum_{i=1}^k D_{x_i}(m_i) + D_{y_i}(n_i)
$$
and
\begin{align*}
D_{z}(\ell) = \frac{(z-\ell)^2}{4} + \frac{2}{3}(\ell^+)^{3/2} - c_k\sqrt{t}|\ell|^{5/4}.
\end{align*}
We set $\al = c_k'(t^2 + \|\bx\|_2^{3/4} + \|\by\|_2^{3/4})$ so that
\begin{equation}
\label{E:dkal-bd}
2\exp (-d_k \al^3) \le \exp\lf(-100 D(\bx, \by) \rg).
\end{equation}
Now,
\begin{align}
\label{E:Dxynm}
-D(\bn, \bm) \le - \frac{1}{3} (m_1 \vee n_1)^{3/2} - 100 D(\bx, \by)
\end{align} 
whenever $m_1 \vee n_1 > (\|\bx|_2 + \|\by\|_2 + 2)^{c_k'}$. All but $(2t^2 + \|\bx\|_2 + \|\by\|_2)^{c_k'}$ terms in the sum in \eqref{E:exp-frack} satisfy this bound. Therefore by \eqref{E:exp-frack}, \eqref{E:dkal-bd}, and \eqref{E:Dxynm} we have
\begin{align*}
f(\bx, \by) &\le \exp\lf(-100 D(\bx, \by) + c_k' \log(\|\bx|_2 + \|\by\|_2 + 2)\rg) \\
&+ \max_{\bm, \bn \in \Z^k_\ge, n_k \wedge m_k \ge -\al} \exp(-D(\bn, \bm) + c_k' \log(\|\bx|_2 + \|\by\|_2 + 2)) \\
&\le \max_{\bm, \bn \in \Z^k_\ge, n_k \wedge m_k \ge -\al} \exp(-D(\bn, \bm) + O_k(\log (\|\bx|_2 + \|\by\|_2 + 2)).
\end{align*}
The $O_k(\log (\|\bx|_2 + \|\by\|_2 + 2)$-term above is lower order compared with the $O_k(t^3 + \|\bx\|_2 + \|\by\|_2)$-term appearing in $D$, so we may drop it from the above expression. It just remains to bound $\max_{\bm, \bn \in \Z^k_\ge, n_k \wedge m_k \ge -\al} D(\bm, \bn)$ below. Indeed, the bound in the lemma will follow if we can show that for every $z \in \R$ we have
\begin{equation}
\label{E:42-Dz}
D_z(\ell) \ge \frac{2}{3}(z^+)^{3/2} + O_k(t^3 + \|\bx\|_2 + \|\by\|_2 + \sqrt{t}|z|^{5/4})
\end{equation}
whenever $\ell \ge - \al$. First, for $\ell \in [-\al, 0]$ we have
\begin{equation}
\label{E:bound-1-Dz}
D_z(\ell) \ge \frac{(z^+)^2}{4} + c_k \sqrt{t}|\al|^{5/4} = \frac{(z^+)^2}{4} + O_k(t^3 + \|\bx\|_2 + \|\by\|_2),
\end{equation}
which is bounded below by the right-hand side of \eqref{E:42-Dz} for any $z$.
Next, for $\ell \ge 0$ we always have the bound
\begin{equation*}
\label{E:ge-0}
D_z(\ell) \ge \frac{2}{3}\ell^{3/2} - c_k\sqrt{t}\ell^{5/4} = O_k(t^3) 
\end{equation*}
which is bounded below by the right-hand side of \eqref{E:42-Dz} when $z \le c_k'' t^2$. Finally, for $z \ge c_k'' t^2$ and $\ell \ge 0$, we can differentiate $D_z$ to get
\begin{align*}
D_z'(\ell) = \frac{2(z -\ell)}{t} - \sqrt{\ell} + \frac{5}{4} c_k \sqrt{t} \ell^{1/4}
\end{align*}
As long as $c_k''$ was chosen large enough relative to $c_k$, $D_z'$ has all its roots in the region where $|\ell - z| \le 20 t \sqrt{z}$, and hence the minimum value of $D_z$ over $\ell \ge 0$ is taken in this interval. Finally, for $|\ell - z| \le 20 t \sqrt{z}$ and $z \ge c_k'' t^2$ we have the estimate
$$
D_z(\ell) \ge \frac{2}{3}z^{3/2} - 20 t z - 2c_k \sqrt{t} |z|^{5/4},
$$
which completes the proof of \eqref{E:42-Dz} when $z \ge c_k'' t^2$.
\end{proof}
	
	We are now ready to prove the upper bound in Theorem \ref{T:tetris-theorem}.
\begin{corollary}
	\label{C:Theorem1-upper}
	With notation as in Theorem \ref{T:tetris-theorem},
	for $\mu_{k, t}$-a.e.\ $f \in \sC^k_0([0, t])$ we have that
	\begin{equation}
	\label{E:Xktf-body}
		X_{k, t}(f) = \exp (-\S(f) + O_k(t^3 + \sqrt{t} \S(f)^{5/6})).
	\end{equation}
\end{corollary}
\begin{proof}
With notation as in Lemma \ref{L:density-estimate} we can write
\begin{equation}
\label{E:Xkt-integrand}
X_{k, t}(f) = \int Y(\bx, \bx + f(t), f) \sqrt{4 \pi t}\exp \lf(\frac{\|f(t)\|_2^2}{4t} \rg) d \bx.
\end{equation}
Observe that $Y(\bx, \bx + f(t), f) = 0$ unless $f_1 + x_1 > f_2 + x_2 > \dots > f_k + x_k$. Equivalently, we require
\begin{equation}
\label{E:xk-condition}
\bx - x_k{\bf 1}\ge \T f(0) - \T f_k(0) {\bf 1}
\end{equation}
where ${\bf 1} = (1, \dots, 1)$. here the functional $\T$ is the Tetris map from Theorem \ref{T:tetris-theorem}.
Next, by Lemma \ref{L:density-estimate}.2 we can see that the integrand in \eqref{E:Xkt-integrand} is bounded above by
\begin{equation}
\label{E:simple-integrand}
\exp\lf(- \frac{2}{3} \sum_{i=1}^k [(x_i^+)^{3/2} + ([x_i + f_i(t)]^+)^{3/2}]+ O_k(t^3 + \sqrt{t}(\|\bx\|_2^{5/4} + \|f(t)\|_2^{5/4})) \rg).
\end{equation}
Also, by Lemma \ref{L:density-estimate}.1, for $x_k - \T f_k(0) \le - c_k t^2$, the integrand is bounded above by
$$
\exp\Big(- d_k|x_k - \T f_k(0)|^3 + \|f(t)\|_2^2 + c_k t^3\Big),
$$
Together with \eqref{E:simple-integrand}, this bounds implies that the contribution to the integral from the region where either $x_1 \ge c_k(t^2 + \S(f)^{2/3})$ or $x_k \le \T f_k(0) - c_k' (t^2 + \S(f)^{1/3})$ is at most $e^{-\S(f)}$ which is bounded above by the right-hand side of \eqref{E:Xktf-body}.

The size of the remaining region is $\exp(O_k(\log (\S(f)+2))$, which folds in to the error, so it enough to simply maximize \eqref{E:simple-integrand} subject to the constraint \eqref{E:xk-condition} on the set where $x_1 \ge c_k(t^2 + \S(f)^{2/3})$ or $x_k \le \T f_k(0) - c_k' (t^2 + \S(f)^{1/3})$. Observe that on this set, the error term in \eqref{E:simple-integrand} is at most $O_k(t^3 + \sqrt{t}\S(f)^{5/6})$ so it suffices to understand the main term. As the main term $\frac{2}{3} \sum_{i=1}^k [(x_i^+)^{3/2} + ([x_i + f_i(t)]^+)^{3/2}]$ is decreasing in each $x_i$, the minimal value occurs at $\bx = \T f(0) - c_k' (t^2 + \S(f)^{1/3})$. Plugging in this value for $\bx$ yields 
$$
\S(f) - O_k((t^2 + \S(f)^{1/3})\sqrt{t^2 + \S(f)^{2/3}}) = \S(f) - O_k(t^3 + \sqrt{t}\S(f)^{5/6}),
$$
as desired.	
\end{proof}
		\section{The lower bound in Theorem \ref{T:tetris-theorem}}
		\label{S:lower-bound}

		In this section we complete the proof of Theorem \ref{T:tetris-theorem}. For this section $k \in \N, t \ge 1,$ and $f \in \sC^k([0, t])$ are fixed throughout and we use the notation $\S, \T$ from Theorem \ref{T:tetris-theorem}.
		The proof of the lower bound for Theorem \ref{T:tetris-theorem} involves a Brownian Gibbs resampling on a staircase-shaped region of the form
		$$
		U = \bigcup_{i=1}^k \{i\} \X [x_i^-, x_i^+],
		$$
		where $\bx^-, \bx^+ \in \R^k$, and as in Figure \ref{fig:parabola-avoidance-multiple} we roughly take $x_i^- \sim -\sqrt{\T f_i(0)}$ and $x_i^+ \sim t + \sqrt{\T f_i(t) + t^2}$. For technical reasons, we need to space the values $x_i^\pm$ out slightly. Define
		\begin{equation}
		\label{E:xi-defn}
		x^-_i = -\sqrt{\T f_i(0)} - (k + 2 - i), \qquad x^+_i = \sqrt{\T f_i(t) + t^2} + t + (k + 2 - i),
		\end{equation}
		With these definitions, we have the separation
		$$
		x^-_i + 1 \le x^-_{i+1}, \qquad x^+_{i+1} + 1 \le x^+_{i},
		$$
		and $x_i^- \le -2$, $x_i^+ \ge 2 t$ for all $i$.
		
		One major simplification when applying a Gibbs resampling argument for the lower bound is that we can work only with a part of space where $\fA$ behaves nicely on the outer boundary of $U$. The next lemma sets up this nice boundary behaviour for $\fA$. For this lemma and throughout this section, let 
		$$
		E_{t, f} = 2t^2 + \T f_1(0) + \T f_1(t).
		$$
	The errors in the section will be expressed in terms of $E_{t, f}$.
	
		\begin{figure}
		\centering
		\includegraphics[width=10cm]{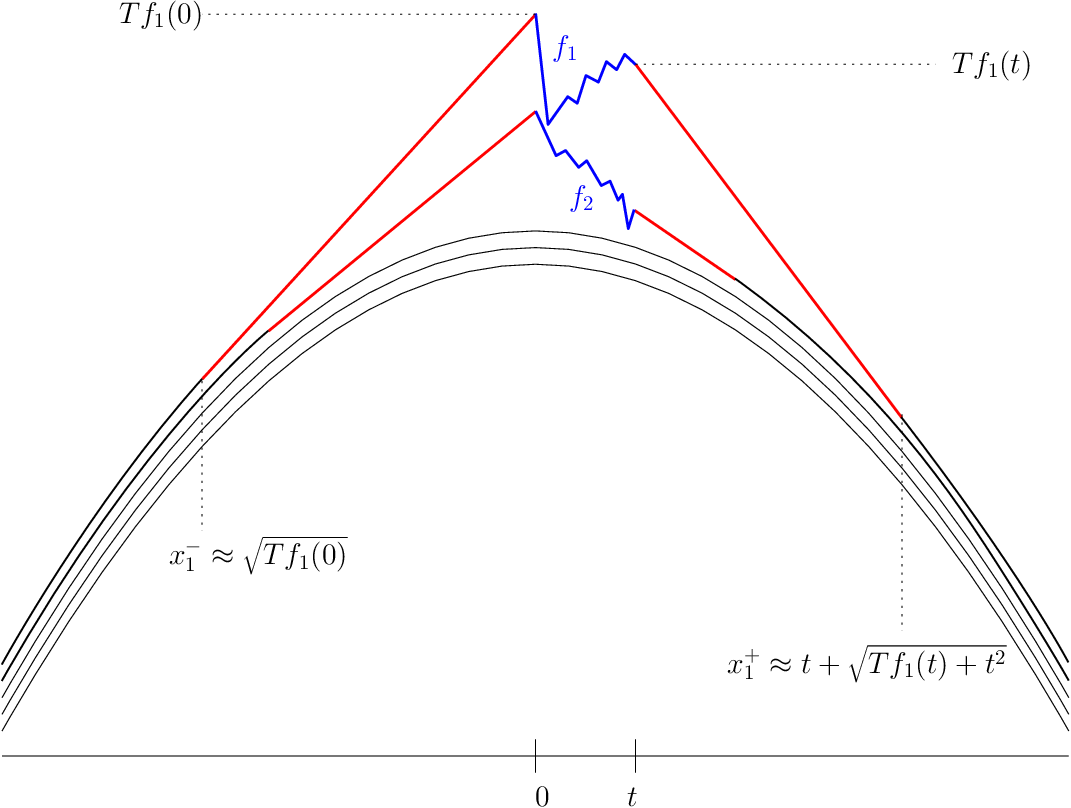}
		\caption{The setup for the lower bound in Theorem \ref{T:tetris-theorem}. The goal is to apply a Brownian Gibbs resampling to create the above configuration, where the red lines are tangent to the parabola. The start and end locations of the red lines for the top curve are indicated. The cost $\S (f)$ is (up to a lower order term) the Dirichlet energy difference between the red lines and the corresponding parabolic segments.}
		\label{fig:parabola-avoidance-multiple}
	\end{figure}

\begin{lemma}
\label{L:good-event-lemma}
Let $A_m$ be the event where
$$
|\fA_{i}(y) + y^2| \le m \log^{2/3}(3 k E_{t, f}) \quad \text{ for all } \quad y \in [- 3 k E_{t, f}, 3 k E_{t, f}], i \in \II{1, k + 1}.
$$
Also, for each $i \in \II{1, k}$, let $C_\ep^{i, -}$ be the set where 
$$
\P_{x_i^-, x_i^- + 1}(\fA_i(x_i^-), \fA_i(x_i^- + 1), \fA_{i+1}) \ge \ep.
$$
In words, $C_\ep^{i, -}$ is the set where the probability that an independent Brownian bridge $B$ from $(x_i^-, \fA_i(x_i^-))$ to $(x_i^- + 1, \fA_i(x_i^- + 1))$ stays above $\fA_{i+1}$ is at least $\ep$. Similarly let $C_\ep^{i, +}$ be the set where
$$
\P_{x_i^+ - 1, x_i^-}(\fA_i(x_i^+ - 1), \fA_i(x_i^+), \fA_{i+1}) \ge \ep.
$$
Then there exists a $k$-dependent constant $m_k > 0$ such that for $m \ge m_k$ we have
$$
\P(D_m
) \ge 1/2, \qquad \text{ where } \qquad D_m = A_m \cap \bigcap_{i=1}^k C_{1/m}^{i, -} \cap C_{1/m}^{i, +}
$$
\end{lemma}

Note the implicit dependence on $t, f$ in the events in the lemma; the claim is that we can take $m_k$ independent of $t, f$.

\begin{proof}
By a union bound, it is enough to show that
\begin{equation*}
\lim_{m \to \infty} \inf_{t \ge 1, f} \P A_m = 1, \qquad \lim_{\ep \to 0} \inf_{t \ge 1, f} \P C_\ep^{i, *} = 1 \quad \text{ for all } i \in \II{1, k}, * \in \{+, -\}.
\end{equation*}
By Lemmas \ref{L:one-point-upper-bound} and \ref{L:one-point-lower-bound} and a union bound, we have
$$
|\fA_{j}(i) + i^2| \le m \log^{2/3}(3 k E_{t, f})
$$
for all integers $i \in \Z\cap [- 3 k E_{t, f} - 1, 3 k E_{t, f} + 1]$ and all $j \in \II{1, k+1}$ with probability tending to $1$ with $m$, uniformly in $t$ and $f$. 
 Next, for all $i \in \Z\cap  [- 3 k E_{t, f} - 1, 3 k E_{t, f}]$ and $j \in \II{1, k}$, by Lemma \ref{L:two-point-estimate-weak} each of the processes $\fA_j(s) + s^2, s \in [i, i + 1]$ satisfies the conditions of the modulus of continuity Lemma \ref{L:mod-cont} with common $\be = 1/5$ and $\al = c_k$. Therefore by another union bound, $\P A_m \to 1$ as $m \to \infty$ uniformly in $t$ and $f$.

Next, since $t \mapsto \fA(t) + t^2$ is stationary, the probabilities $\P C^{i, \pm}_\ep$ do not depend on the exact values of $x_i^\pm$, and hence they do not depend on $f, t$. Moreover, the strict ordering of the Airy line ensemble implies that for every $i, \ep$, we have $\P C^{i, *}_\ep \to 1$ as $\ep \to 0$, as desired. 
\end{proof}

One of the main difficulties in applying a Gibbs resampling to make rigorous the picture in Figure \ref{fig:parabola-avoidance-multiple} is dealing with non-intersection costs that come from an $O(1)$ region near the boundary of $U$. We can handle this on the event $D_m$ with the following lemma.

Setting some notation,  let $\bx \in \R^k_\ge, \by \in \R^k_\le$ and define $I_k = [x_k, y_k], I_j = [x_j, y_j] \smin I_{j+1}$ for $j < k$. Consider a $k$-tuple of functions $h = (h_2, \dots, h_{k+1})$ such that the domain of $f_i$ contains $I_{i-1}$ for all $i$. Let 
$$
\NI(\bx, \by, h) = \{g \in \prod_{j=1}^k \sC([x_j, y_j]) : g_1(r) > \dots > g_j(r) > h_{j+1}(r) \text{ for all } r \in I_j, j \in \II{1, k} \},
$$ 
and let $\P_{\bx, \by}(\bu, \bv, h)$ be the probability that $k$ independent Brownian bridges from $(x_i, u_i)$ to $(y_i, v_i)$ lie in the set $\NI(\bx, \by, f)$.
Finally we set
$$
\zeta = \zeta_m = m \log^{2/3}(3 k E_{t, f}).
$$
\begin{lemma}
	\label{L:avoidance-lines-lemma}
	Define $\fA^k(\bx^\pm) \in \R^k_>$ by
	$$
	\fA^k(\bx^\pm) = (\fA_1(x_1^\pm), \dots, \fA_k(x_k^\pm)), 
	$$
	and let $\iota = (k, k-1, \dots, 1)$. Then there exists a $k$-dependent constant $m_k' \ge 0$ such that for $m \ge m_k'$, on $D_m$ we have the following almost sure lower bounds on non-intersection probabilities:
	\begin{equation}
	\label{E:0-event}
	\P_{\bx^-, 0}(\fA^k(\bx^-), \T f(0) + 2\zeta_m\iota, \fA|_{U^c}) \ge e^{-c_k m^2 \sqrt{E_{t, f}}}.
	\end{equation}
	and
	\begin{equation}
	\label{E:t-event}
	\P_{t, \bx^+}(\T f(t) + 2\zeta_m\iota, \fA^k(\bx^+), \fA|_{U^c}) \ge e^{-c_k m^2 \sqrt{E_{t, f}}}.
	\end{equation}
\end{lemma}
\begin{proof}	
\textbf{Step 1: Defining events for the proof of \eqref{E:0-event}.} \qquad Define
$$
\ell_i(s) = 2\sqrt{\T f_i(0)}s + \T f_i(0) + 2 (k-i + 1)\zeta_m.
$$ 	
Observe that for all $i \in \II{1, k}, s \in \R$ we have 
\begin{equation}
\label{E:first-li-constraint}
\ell_i(s) \ge - s^2 + 2 \zeta_m.
\end{equation}
Moreover, for $i \in \II{1, k-1}$ we have
\begin{equation}
\label{E:second-li-constraint}
\ell_i(s) \ge \ell_{i+1}(s) + 2\zeta_m, \qquad s \in [x_{i+1}^- + 1, 0].
\end{equation}
The event $D_m$ implies that for all $(i, s) \in \II{1, k + 1} \X [x_1^-, x_1^+]$ we have $\fA_i(s) \le - s^2 + \zeta_m$. Therefore if $B$ is a $k$-tuple independent Brownian bridges from $(x_i^-, \fA^k(x_i^-))$ to $(0, \T f(0) + 2 \zeta_m \iota)$, then $B \in \NI(\bx^-, 0, \fA|_{U^c})$ as long as the following three events hold for all $i \in \II{1, k}$ and $m$ is sufficiently large.
\begin{enumerate}
	\item $G_{1, i}$: \quad We have $B_i(x_i^- + 1) \in [\ell_i(x_i^- + 1) - 1, \ell_i(x_i^- + 1)]$.
	\item $G_{2, i}$: \quad For all $s \in [x_i^- + 1, 0]$ we have $|B_i(s) - \ell_i(s)| \le 2$.
	\item $G_{3, i}$: \quad For all $s \in [x_i^-, x_i^- + 1]$ we have $\fA_{i+1}(s) < B_i(s) < \ell_{i-1}(s) - \zeta_m$. Here the upper inequality is dropped if $i=1$.
\end{enumerate}
Therefore conditional on $\fA|_U$, we have that
\begin{align*}
\P_{\bx^-, 0}(\fA^k(\bx^-), \T f(0) + 2\zeta_m\iota, \fA|_{U^c}) &\ge \P(\bigcap_{i=1}^k G_{1, i} \cap G_{2, i} \cap G_{3, i}) \\
&\ge \prod_{i=1}^k  \P(G_{1, i}) \P(G_{2, i} \mid G_{1, i}) \P(G_{3, i} \mid G_{1, i}).
\end{align*}
Here the final inequality uses independence properties of the Brownian bridges $B_i$.

\textbf{Step 2: Bounding $\P(G_{1, i}), \P(G_{2, i} \mid G_{1, i})$, and $\P(G_{3, i} \mid G_{1, i})$ on $D_m$.} \qquad We have
\begin{equation*}
\P(G_{1, i}) = \frac{1}{\sqrt{4\pi(|x_i^-| - 1)/|x_i^-|}} \int_{\ell_i(x_i^- + 1) - 1}^{\ell_i(x_i^- + 1)} \exp\lf(- \frac{(u - \mu_i)^2}{4(|x_i^-| - 1)/|x_i^-|}\rg) du,
\end{equation*}
where 
$$
\mu_i = \frac{|x_i^-| - 1}{|x_i^-|}\fA_i(x_i^-) + \frac{1}{|x_i^-|}[\T f_i(0) + 2 (k-i + 1)\zeta_m].
$$
Now, on $D_m$ we have $|\fA_i(x_i^-) + (x_i^-)^2| \le \zeta_m$, so using that $x_i^-= -\sqrt{\T f_i(0)} - (k + 2 - i)$ and $x_i^- \le -2$ we can check that
$$
|\mu_i + (x_i^- + 1)^2| \le c_k \zeta_m
$$
for a $k$-dependent constant $c_k > 0$. The same bound holds with $\ell_i(x_i^- + 1)$ in place of $\mu_i$. Therefore
\begin{equation}
\label{E:G1-bd}
\P(G_{1, i}) \ge e^{-c_k \zeta_m^2} = e^{-c_k m^2 \log^{4/3} (3k E_{t, f})}.
\end{equation}
Next, $\P(G_{2, i} \mid G_{1, i})$ is bounded below by the probability that a Brownian bridge from $(x_i^- + 1,0)$ to $(0, 0)$ has maximum absolute value at most $1$. Therefore
\begin{equation}
\label{E:G2-bd}
\P(G_{2, i} \mid G_{1, i}) \ge e^{- c x_i^-}
\end{equation}
for an absolute constant $c > 0$. We turn our attention to $G_{3, i}$. We will condition on the value of $u = B_i(x_i^- + 1)$ given $u \in [\ell_i(x_i^- + 1) - 1, \ell_i(x_i^- + 1)]$; the bounds on $u$ are implied by the event $G_{1, i}$. Let $L_u$ be the linear function with $L_u(x_i^-) = \fA_i(x_i^-)$ and $L_u(x_i^- + 1) = u$. Since $\fA_i(x_i^-) \le - (x_i^-)^2 + \zeta_m$, by \eqref{E:first-li-constraint} and \eqref{E:second-li-constraint} we have $\ell_{i-1}(s) - \zeta_m \ge L_u(s)$. Therefore for all large enough $m$, we have
$$
\P(B_i(s) < \ell_{i-1}(s) - \zeta_m \text{ for all } s \in [x_i^-, x_i^- + 1]\mid u = B_i(x_i^- + 1) ) \ge 1 - (2m)^{-1}
$$
for all $u \in [\ell_i(x_i^- + 1) - 1, \ell_i(x_i^- + 1)]$. Moreover, $\ell_i(x_i^- + 1) - 1 \ge \fA_i(x_i^- + 1)$ by \eqref{E:first-li-constraint} on the event $D_m$. Therefore on $D_m$, 
\begin{align*}
\P(B_i(s) > \fA_{i+1}(s) \text{ for all } s \in [x_i^-, x_i^- + 1] &\mid u = B_i(x_i^- + 1) ) \\
&\ge \P_{x_i^-, x_i^- + 1}(\fA_i(x_i^-), \fA_i(x_i^- + 1), \fA_{i+1}) \ge 1/m,
\end{align*}
where the final inequality follows since $D_m$ implies the event $C^{i, -}_{1/m}$.
Combining these two bounds by a union bound gives that
\begin{equation}
\label{E:G3i-bd}
\P(G_{3, i} \mid G_{1, i}) \ge 1/(2m).
\end{equation}
The bounds \eqref{E:G1-bd}, \eqref{E:G2-bd}, and \eqref{E:G3i-bd} then yield \eqref{E:0-event}.

\textbf{Step 3: The proof of \eqref{E:t-event}.} \qquad Define
$$
\tilde \fA_i(s) = \fA_i(t - s) + t^2 - 2st, \qquad \tilde U = \bigcup_{i=1}^k \{i\} \X [-x_i^+ + t, -x_i^- + t],
$$
and note that $\tilde \fA$ is another parabolic Airy line ensemble. Now, since Brownian bridge law commutes with affine shifts,
\begin{align*}
\P_{t, \bx^+}(\T f(t) + 2\zeta_m\iota, \fA^k(\bx^+), \fA|_{U^c}) = \P_{-\bx^+ + t, 0}(\tilde \fA^k(-\bx^+ + t), Z(f) + 2\zeta_m\iota, \tilde \fA^k|_{\tilde U^c}),
\end{align*}
where $Z(f) = \T f(t) + t^2$.
Via this transformation, the proof of \eqref{E:t-event} follows in the same way as the proof of \eqref{E:0-event} with $Z$ in place of $Y$, and $-x_i^+ + t$ in place of $x_i^-$. Indeed, the only inputs we required from $\fA$ also hold for $\tilde \fA$ on $D$. These are parabolic shape bounds implied by the event $A_m$, and the bound $\P_{x_i^-, x_i^- + 1}(\fA_i(x_i^-), \fA_i(x_i^- + 1), \fA_{i+1}) \ge 1/m$; the analogue 
$$
\P_{-x_i^+ + t, -x_i^- + t + 1}(\tilde \fA_i(-x_i^+ + t), \tilde \fA_i(-x_i^+ + t + 1), \tilde \fA_{i+1}) \ge 1/m
$$
is implied by the event $C^{i, +}_{1/m}$.
\end{proof}

We can now prove the upper bound in Theorem \ref{T:tetris-theorem}. In fact, we get a slightly stronger estimate on the error than what is claimed in the introduction.
\begin{prop}
	\label{P:tetris-lower}
We have
	\begin{equation*}
	X_{k, t}(f) \ge \exp (-\S(f) + O_k(t E_{t, f})) = \exp(-\S(f) + O_k(t^3 + t \S(f)^{2/3})).
	\end{equation*}
\end{prop}

\begin{proof} 		
	\textbf{Step 1: A conditional density formula.} \qquad
	Let $\cF$ be the $\sig$-algebra generated by $\fA$ on the set $U^c$. 	We first use the Brownian Gibbs property to give a formula for the conditional expectation
	\begin{equation}
	\label{E:rnd}
		\E_\cF(X_{k, t}(f)).
	\end{equation}
	Let $\nu_{\bu^-, \bu^+}$ be the law of $k$ independent Brownian bridges $B_1, \dots, B_k$, where $B_i$ goes from $(x_i^-, u^-_i)$ to $(x_i^+, u^+_i)$, and let $\tilde \nu_{\bu^-, \bu^+}$ be the pushforward of $\nu_{\bu^-, \bu^+}$ onto $\R^k \X \sC^k_0([0, t])$ under the map
	\begin{equation}
	\label{E:pushing-map}
	(g_1, \dots, g_k) \mapsto [(g_1(0), \dots, g_k(0)); (g_1|_{[0, t]} - g_1(0), \dots, g_k|_{[0, t]} - g_k(0))].
	\end{equation}
	Then
	\begin{equation}
	\label{E:dtilde-nu}
	\begin{split}
	\frac{d \tilde \nu_{\bu^-, \bu^+}}{d\by d \mu_{k, t}}(\by, f) &= \lf(\frac{(x_i^+ + |x_i^-|)}{(x_i^+ - t) |x_i^-|}\rg)^{k/2} \\ 
	&\X \prod_{i=1}^k  \exp \lf(- \frac{(y_i - u_i^-)^2}{4|x_i^-|} - \frac{(u_i^+ - y_i - f_i(t))^2}{4(x_i^+ - t)}  + \frac{(u_i^+ - u_i^-)^2}{4(x_i^+ + |x_i^-|)}\rg).
	\end{split}
	\end{equation}
	Now, for a function $h:\II{1, k+1} \X [x_1^-, x_1^+] \smin U \to \R$, define the conditional measure $\eta_{\bu^-, \bu^+, h}$ on $\sC(U)$ by
	$$
	\eta_{\bu^-, \bu^+, h}(A) = \frac{\nu_{\bu^-, \bu^+}(A \cap \NI(\bx^-, \bx^+, h))}{\P_{\bx^-, \bx^+}(\bu^-, \bu^+, h)},
	$$
	and let $\tilde \eta_{\bu^-, \bu^+, h}$ be the pushforward of $\eta_{\bu^-, \bu^+, h}$ under the map \eqref{E:pushing-map}.
 Then
	\begin{equation}
	\label{E:dtilde-eta}
	\frac{d \tilde \eta_{\bu^-, \bu^+, h}}{d \tilde \nu_{\bu^-, \bu^+}}(\by, f) = \frac{\P_{\bx^-, 0}(\bu^-, \by, h) \P_{t, \bx^+}(\by + f(t), \bu^+, h)\indic(f_1 + y_1 > \dots > f_k + y_k > h_{k+1})}{\P_{\bx^-, \bx^+}(\bu^-, \bu^+, h)}.
	\end{equation}
	Now, the conditional expectation \eqref{E:rnd} equals
	\begin{equation}
	\label{E:rn-integrated}
	\int_{\R^k}  \frac{d \tilde \eta_{\fA^k(\bx^-), \fA^k(\bx^+), \fA|_{U^c}}}{d \by d\mu^k}(\by, f) d \by = \int_{\R^k}  \frac{d \tilde \eta_{\fA^k(\bx^-), \fA^k(\bx^+), \fA|_{U^c}}}{d \tilde \nu_{\fA^k(\bx^-), \fA^k(\bx^+)}} \frac{d \tilde \nu_{\fA^k(\bx^-), \fA^k(\bx^+)}}{d \by d\mu^k}(\by, f) d \by
	\end{equation}
Let $m > 0$ be large enough so that Lemmas \ref{L:good-event-lemma} and \ref{L:avoidance-lines-lemma} hold, and let $\tilde D_m$ be the $\cF$-measurable event where $\P_\cF(D_m) > 0$. 
We will use \eqref{E:dtilde-eta} and \eqref{E:dtilde-nu} to bound \eqref{E:rn-integrated} below on the event $\tilde D_m$.

\textbf{Step 2: Bounding \eqref{E:dtilde-eta} on $\tilde D_m$.} \qquad 
Let $\zeta_m = m \log^{2/3}(3k E_{t, f})$ be as in Lemma \ref{L:avoidance-lines-lemma}, and let $W$ be the set of all $\by \in \R^k$ satisfying
\begin{equation}
\label{E:ykk}
\begin{split}
y_i &\in (\T f_i(0) + (k - i + 1)\zeta_m,  \T f_i(0) + (k - i + 1)\zeta_m+ 1).
\end{split}
\end{equation}
On the set $\tilde D_m$ (which implies $\fA_{k+1} \le \zeta_m$) if $\by \in W$ we have
\begin{equation}
\label{E:fyfy}
\indic(f_1 + y_1 > f_2 + y_2 > \dots > f_k + y_k > \fA_{k+1}) = 1.
\end{equation}
Next, for any $\bu^-, \bu^+, h$, the avoidance probabilities
$
\P_{\bx^-, 0}(\bu^-, \by, h)$ and $\P_{t, \bx^+}(\by + f(t), \bu^+, h)
$
are monotone increasing if we add $\bz \in [0, \infty)_>$ to the coordinate $\by$. Therefore if we let $\by \in W$, and set $\iota = (k, k-1, \dots, 1)$, then
\begin{align}
\nonumber
&\P_{\bx^-, 0}(\fA^k(\bx^-), \by, \fA|_{U^c}) \P_{t, \bx^+}(\by + f(t), \fA^k(\bx^+), \fA|_{U^c}) \\
\nonumber
\ge\; &\P_{\bx^-, 0}(\fA^k(\bx^-), \T f(0) + 2\zeta_m \iota, \fA|_{U^c}) \P_{y, \bx^+}(\T f(0) + 2\zeta_m \iota, \fA^k(\bx^+), \fA|_{U^c})\\
\label{E:eckm2}
\ge\; &e^{-c_k m^2 \sqrt{E_{t, f}}}.
\end{align}
where the second inequality holds on $\tilde D_m$ by Lemma \ref{L:avoidance-lines-lemma}.

Moreover, $\fA_{k+1}(x) \ge - x^2 - \zeta_m$ for $x \in [x_1^-, x_1^+]$ on $\tilde D_m$. Therefore for any $\bu^\pm$, letting $p(x) = - x^2$ we have
\begin{align}
\nonumber
\P_{\bx^-, \bx^+}(\bu^-, \bu^+, \fA|_{U^c}) &\le \prod_{i=1}^k \P_{x^-_i, x_i^+}(u_i^-, u_i^+, \fA_{k+1})
\le \prod_{i=1}^k \P_{x^-_i, x_i^+}(u_i^-, u_i^+, -p - \zeta_m).
\end{align}
We bound each of the terms in this product when $|u_i^\pm + (x_i^\pm)^2| \le \zeta_m$. First suppose that $x_i^+ - x_i^- \ge 2 \sqrt{2 \zeta_m}$. Then we can apply Lemma \ref{L:parabola-lemma-A}.1 to get that the $i$th term in the product is bounded above by 
$$
\exp \lf(c(x_i^+ - x_i^-) \log (x_i^+ - x_i^-) -J(u_i^- + (x_i^-)^2 + \zeta_m, u_i^+ + (x_i^+)^2 + \zeta_m, x_i^+ - x_i^-)\rg).
$$
Since $|u_i^\pm - (x_i^\pm)^2| \le \zeta_m$, simple algebra shows that this is bounded above by
\begin{align*}
&\exp \lf(c (x_i^+ - x_i^-) \log (x_i^+ - x_i^-) + \frac{(u_i^- - u_i^+)^2}{4(x_i^+ - x_i^-)} - \frac{(x_i^+)^3 + |x_i^-|^3}{3} - 2 \zeta_m (x_i^+ - x_i^-)\  \rg) \\
&\le \exp \lf(\frac{(u_i^- - u_i^+)^2}{4(x_i^+ - x_i^-)} - \frac{1}{3} \lf([\T f_i(0)]^{3/2} + [\T f_i(t)]^{3/2}\rg) + c_{k, m} E_{t, f} \rg)
&\end{align*}
where $c_{k, m}$ is a $k, m$-dependent constant.
On the other hand, if $x_i^+ - x_i^- \le 2 \sqrt{2 \zeta_m}$, then the same bound holds trivially as long as $c_{k, m}$ is large enough, since $\T f_i(0) \le |x_i^-|^2$ and $\T f_i(t) \le |x_i^+|^2$.

Putting this computation together with \eqref{E:dtilde-eta}, \eqref{E:fyfy}, and \eqref{E:eckm2} implies that for $\by \in W$ we have
\begin{equation}
\label{E:first-goal-bd}
\begin{split}
&\frac{d \tilde \eta_{\fA^k(\bx^-), \fA^k(\bx^+), \fA|_{U^c}}}{d \tilde \nu_{\fA^k(\bx^-), \fA^k(\bx^+)}} \ge \\
& \exp\lf(  \sum_{i=1}^k \lf[\frac{1}{3}\lf([\T f_i(0)]^{3/2} + [\T f_i(t)]^{3/2}\rg)  - \frac{(\fA_i(x_i^+) - \fA_i(x_i^-))^2}{4(x_i^+ + |x_i^-|)} \rg] - c_{k, m} E_{t, f}\rg).
\end{split}
\end{equation}
\textbf{Step 3: Bounding \eqref{E:dtilde-nu} below on $\tilde D_m$.} \qquad We turn our attention to bounding \eqref{E:dtilde-nu} below on $\tilde D_m$ for $\by \in W,$ and $\bu^\pm = \fA^k(\bx^\pm)$. 

First, noting that $|x_i^-|, x_i^+ \le \sqrt{E_{t, f}}$ and $x_i^+ \ge 2t$, the factor prior to the exponential in \eqref{E:dtilde-nu} is bounded below by
$$
\exp(-c k \log(E_{t, f}))
$$
for an absolute constant $c > 0$. Next, on $\tilde D_m$ we have that $u_i^\pm + (x_i^\pm)^2 \in [-\zeta_m, \zeta_m]$ and for $\by \in W$ we have $y_i \ge \zeta_m$ for all $i$, so $y_i - u_i^+ \ge 0$. Moreover, using that $x_i^+ =  \sqrt{\T f_i(t) + t^2} + t + (k + 2 - i)$ and that $y_i + f_i(t) - \T f_i(t) \in [2 \zeta_m, 3k \zeta_m]$ for $i \in \II{1, k}, \by \in W$
we have
\begin{align*}
\frac{(u_i^+ - y_i - f_i(t))^2}{4(x_i^+ - t)} 
&\le \frac{(2 \T f_i(t) + c_k t\sqrt{E_{t, f}})^2}{4\sqrt{\T f_i(t) + t^2}} \\
&\le [ \T f_i(t)]^{3/2} + c_k t E_{t, f}.
\end{align*}
Similarly, 
$$
\frac{(y_i - u_i^-)^2}{4 |x_i^-|} \le [ \T f_i(0)]^{3/2} + c_k t E_{t, f}
$$
Combining all three of these bounds, on $\tilde D_m$ for $\by \in W,$ and $\bu^\pm = \fA^k(\bx^\pm)$, \eqref{E:dtilde-nu} is bounded below by
$$
\prod_{i=1}^k \exp \lf(- [\T f_i(0)]^{3/2} - [\T f_i(t)]^{3/2} + \frac{(\fA_i(x_i^+) - \fA_i(x_i^-))^2}{4(x_i^+ + |x_i^-|)} - c_k t E_{t, f}\rg).
$$
Putting this together with \eqref{E:first-goal-bd} and using that the Lebesgue measure of $W$ is $1$ gives that on $\tilde D_m$, \eqref{E:rnd} is bounded below by 
$$
\exp \lf(- \sum_{i=1}^k \frac{2}{3}\lf([\T f_i(0)]^{3/2} + [\T f_i(t)]^{3/2}\rg) - c_k t E_{t, f}\rg).
$$
The proposition follows since $\P(\tilde D_m) \ge \P(D_m) \ge 1/2$ by Lemma \ref{L:good-event-lemma}.
\end{proof}

\section{Proofs of Corollaries and Remark \ref{R:RN-context}}
\label{S:proofs-corollaries}

Corollaries \ref{C:bounded-above} and \ref{C:schilder-Airy} have easy proofs.

\begin{proof}[Proof of Corollary \ref{C:bounded-above}]
By Theorem \ref{T:tetris-theorem}, letting $u_f = \T f_1(0) + \T f_1(t)$ we have
$$
X_{k, t}(f) \le \exp \lf(c_k t^3 + c_k \sqrt{t}u_f^{5/4} - \frac{1}{3}u_f^{3/2} \rg),
$$
which is bounded above by $e^{O_k(t^3)}$ for all $f$. 
	\end{proof}

\begin{proof}[Proof of Corollary \ref{C:schilder-Airy}]
The usual Schilder's theorem is Corollary \ref{C:schilder-Airy} with $\nu_{k, t}$ replaced by the Brownian measure $\mu_{k, t}$. The version for $\nu_{k, t}$ follows since for any Borel set $A$ and any $\ep > 0$, we have
$$
\log \mu^\ep_{k, t}(A) - c_k t^3 - d_k \log^{3/4} [\mu^\ep_{k, t}(A)^{-1}] \le \log \nu^\ep_{k, t}(A) \le \log \mu^\ep_{k, t}(A) + c_k t^3
$$
by Corollaries \ref{C:bounded-above} and \ref{C:X-bounded-below}.
\end{proof}

We need to work harder to prove Corollary \ref{C:X-bounded-below}. We will need the following elementary lemma about Brownian motion.

\begin{lemma}
	\label{L:Brownian-sup-inf-bound}
	Let $B \in \sC_0([0, t])$ be a Brownian motion, and let
	$$
	S = S(B) = \sup_{s \in [0, t]} B(s), \qquad I = I(B) = - \inf_{s \in [0, t]} B(s).
	$$
	Then $(S, I, B(t))$ has a Lebesgue density on the set $\{(a, b, x) \in [0, \infty)^2 \X \R : - b \le x \le a\}$ given by $h(a, b, x) + h(a, b, -x)$, where
	$$
	h(a, b, x) = \lf[ \frac{(2a + 2b - x)^2}{8\sqrt{\pi} t^{5/2}} - \frac{1}{2 \sqrt{\pi} t^{3/2}}\rg]\exp \lf( - \frac{(2a + 2b - x)^2}{4t}\rg)
	$$
\end{lemma}

\begin{proof}
	This follows from a reflection principle argument. Fix $a \ge 0, b \ge 0$ and let 
	$$
	\tau_a = \inf \{s \in [0, t] : B(s) = a \}, \qquad \tau_b = \inf \{s \in [0, t] : B(s) = -b\},
	$$
	and let $\tau_1 = \tau_a \wedge \tau_b, \tau_2 = \tau_a \vee \tau_b$. Let $\tilde B$ be the unique continuous process satisfying 
	$$
	\tilde B|_{[0, \tau_1]} = B|_{[0, \tau_1]}, \quad \tilde B|_{[\tau_1, \tau_2]} - \tilde B(\tau_1) = -B|_{[\tau_1, \tau_2]} + B(\tau_1), \quad B|_{[\tau_2, t]} - \tilde B(\tau_2) = B|_{[\tau_2, t]} - B(\tau_2).
	$$
	Then by the strong Markov property, $\tilde B$ is another Brownian motion, and for $x \in [-b, a]$ we have
	$$
	\{S \ge a, I \ge b, B(t) = x \} = \{\tilde B(t) = 2a + 2b + x \} \cup \{- \tilde B(t) = 2a + 2b - x \}.
	$$
	Therefore
	$$
	\P(S \ge a, I \ge b, B(t) \in dx) = \frac{1}{\sqrt{4 \pi t}} \lf[\exp \lf( - \frac{(2a + 2b - x)^2}{4t}\rg) + \exp \lf( - \frac{(2a + 2b + x)^2}{4t}\rg)\rg] dx.
	$$
	Differentiating in both $a$ and $b$ gives the joint density.
\end{proof}

We also require a straightforward optimization lemma.
\begin{lemma}
\label{L:opt-lemma}
Define $W^k = \{\bu = (\ba, \bb, \bx) \in [0, \infty)^{2k} \X \R^k : - b_i \le x_i \le a_i \text{ for all } i\}$, and define $g:W^k \to \R$ by
\begin{align*}
g(\bu) &= \frac{2}{3} \sum_{i=1}^k \lf(Y_i(\ba, \bb)^{3/2} +  (Y_i(\ba, \bb) + x_i)^{3/2} \rg), \qquad Y_i(\ba, \bb) = \sum_{j=i+1}^k a_j + \sum_{j=i}^k b_j.
\end{align*}
Now, let
$$
D_\beta =\{\bu = (\ba, \bb, \bx) \in W^k : \sum_{i=1}^k(2a_i + 2b_i - |x_i|)^2 \le \beta^2\}. 
$$
Then
$$
\max \{g(\bu) : \bu \in D_\beta\} = \beta^{3/2} \max \{\Theta(\ba) : \ba \in \R^k, \|\ba\|_2^2 = 1\} = \beta^{3/2} \al_k
$$
where $\Theta:[0, \infty)^k \to \R, \al_k$ are as in Corollary \ref{C:X-bounded-below}.
\end{lemma}

\begin{proof}
First, observe that 
\begin{align*}
\beta^{3/2} \max \{\Theta(\ba) : \bx \in \R^k, \|\ba\|^2 = 1\} &=  \max \{\Theta(\ba) : \ba \in [0, \infty)^k, \|\bx\|_2^2 = \beta\} \\
&= \max \{g(\ba, 0, \ba) : \ba \in [0, \infty)^k, \|\ba\|_2^2 = \beta\} \\
&= \max \{g(\ba, 0, \ba) : (\ba, 0, \ba) \in D_\beta\}, 
\end{align*}
so it suffices to prove that the maximum of $g$ is obtained at a point of the form $(\ba, 0, \ba)$. 
For $i \in \II{1, k}$, observe that the map from $W^k \to W^k$ sending
$$
(a_i, b_i, x_i) \mapsto (a_i - x_i, b_i + x_i, - x_i)
$$
and fixing all other coordinates of $\bu = (\ba, \bb, \bx)$
is an involution on $D_\beta$ that fixes $g$. Therefore the maximum of $g$ on $D_\beta$ is attained at a point where all $x_i \ge 0$. Next, when $x_i \ge 0$, we have
$
[2\del_{a_i} - \del_{b_i} + 2\del_{x_i}]g \ge 0.
$
Therefore letting $\bfe_i = (0, \dots, 1, \dots 0) \in \R^k$ be the $i$th coordinate vector, the maximum of $g$ on $D_\beta$ is attained at a point $(\ba, \bb, \bu)$ where all $x_i \ge 0$ and 
\begin{equation}
\label{E:baep}
(\ba + 2 \ep \bfe_i, \bb - \ep \bfe_i, \bx + 2 \ep \bfe_i) \notin D_\beta
\end{equation}
for all $\ep > 0, i \in \II{1, k}$. Now, if $x_i \ge 0$, then $[2\del_{a_i} - \del_{b_i} + 2\del_{x_i}^+](2a_i + 2b_i - |x_i|) = 0$, where $\del_{x_i}^+$ denotes the right-hand derivative. Therefore the only way that \eqref{E:baep} can hold is if 
$
(\ba + 2 \ep \bfe_i, \bb - \ep \bfe_i, \bx + 2 \ep \bfe_i) \notin W^k
$
for all $\ep > 0, i \in \II{1, k}$, which implies that $\bb = 0$. Finally, if $(\ba, 0, \bx) \in D_\beta$, then $(\ba, 0, \ba)$ is also in $D_\beta$, and furthermore $g(\ba, 0, \bx) \le g(\ba, 0, \ba)$. Hence the maximum of $g$ is obtained at a point of the form $(\ba, 0, \ba)$.
\end{proof}

\begin{proof}[Proof of Corollary \ref{C:X-bounded-below}] See Figure \ref{fig:optimal} for a sketch of the types of functions that maximize $\S(f)$ relative to the probability that Brownian motion is close to $f$ (which is controlled by the Dirichlet energy of $f$). Throughout we will use the notation from the statement of Lemma \ref{L:opt-lemma}. Let $B$ be a $k$-dimensional Brownian motion in $\sC_0^k([0, t])$, and define $S = (S(B_1), \dots, S(B_k)), I = (I(B_1), \dots, I(B_k))$ with $S(B_i), I(B_i)$ as in Lemma \ref{L:Brownian-sup-inf-bound}. By Lemma \ref{L:Brownian-sup-inf-bound}, $(S, I, B(t))$ has a joint density $f(\ba, \bb, \bx)$ on $W^k$ which is bounded above by
$$
c_k u^{2k} t^{-5k/2}e^{-u^2/4t}, \qquad \text{ where } u = u(\ba, \bb, \bx) = \sqrt{\sum_{i=1}^k(2a_i + 2b_i - |x_i|)^2}.
$$
Now, shells in $W^k$ of the form $D_{v +\ep} \smin D_v$ with $\ep \le 1$ have volume bounded above by $c_k \ep v^{k-1}$, so
$$
\int_{W^k \smin D_v} f(\ba, \bb, \bx) d \ba d\bb d\bx \le c_k t^{-5k/2} \int_v^{\infty} u^{3k+1} e^{-u^2/4t} du \le c_k v^{3k} e^{-v^2/4t}.
$$
In particular, if $A$ is any set with $\P(B \in A) \ge \ep$, then on a set $A' \sset A$ with $\P(B \in A') \ge \ep/2$, we have
$$
(S, I, B(t)) \in D_{h(\ep)}, \qquad \text{ where } \quad h(\ep) = [4t \log (\ep^{-1})]^{1/2} + c_k \sqrt{t}.
$$
To lower bound $\nu_{k, t}(A)$, using the notation of Lemma \ref{L:opt-lemma} we have
$$
\T B_i(0) \le Y_i(S, I), \qquad T B_i(t) \le Y_i(S, I) + B_i(t).
$$
Therefore by Theorem \ref{T:tetris-theorem}, we have
$$
\nu_k(A) \ge \ep \exp \lf(- \sup\{ g(\ba, \bb, \bx) + e(\ba, \bb, \bx) : (\ba, \bb, \bx) \in D_{h(\ep)} \}\rg), 
$$
where $g$ is as in Lemma \ref{L:opt-lemma} and
\begin{align*}
e(\ba, \bb, \bx) &= c_k (t^3 + \sqrt{t}(Y_1(\ba, \bb) + |\bx_1(t)|)^{5/4}).
\end{align*}
First, since $|x_i| \le a_i \vee b_i$ for all $i$, we have 
$$
\sqrt{\sum_{i=1}^k(2a_i + 2b_i - |x_i|)^2} \ge \sqrt{\sum_{i=1}^k(a_i + b_i)^2} \ge \frac{1}{k} \sum_{i=1}^k (a_i + b_i). 
$$
Therefore on $D_{h(\ep)}$,
$$
e(\ba, \bb, \bx) \le c_k (t^3 + \sqrt{t} [h(\ep)]^{5/4}).
$$
Now, we maximized $g$ over $D_{h(\ep)}$ in Lemma \ref{L:opt-lemma}. This gives that
$$
\nu_k(A) \ge \ep \exp \lf(- [h(\ep)]^{3/2} \al_k - c_k (t^3 + \sqrt{t} [h(\ep)]^{5/4})\rg),
$$
which gives the desired lower bound in the corollary after simplification.

For the upper bound, let $\by \in [0, \infty)^k$ be any argmax of $\Theta$ on the unit sphere. For $\beta > 0$, define let $A_\beta \sset \sC^k_0([0, t])$ be the set of functions where
$$
\{(-1)^{i} f_i(t) \ge \beta y_i, i \in \II{1, k} \}.
$$
We can find $\be(\ep)$ such that $\P(B \in A_{\be(\ep)}) = \ep$. To bound the value of $\be(\ep)$ observe that for $\be > 2\sqrt{t}$ we have
$$
\P(B \in A_\be) \ge \exp\lf( -c_k \log (\beta^2/(4t)) - \frac{\be^2}{4t} \rg).
$$
Hence $\be(\ep) \ge [4t \log (\ep^{-1})]^{1/2} - c_k \sqrt{t}$. Now, for $f \in A_\be$ we have
\begin{align*}
\T f(0) \ge \be \Big(y_i\indic(i \notin 2 \Z) + \sum_{j = i + 1}^k y_j \Big), \qquad \T f(t) \ge \be \Big(y_i\indic(i \in 2 \Z) + \sum_{j = i + 1}^k y_j\Big).
\end{align*}
Hence for $f \in A_\be$, by Theorem \ref{T:tetris-theorem} we have
$$
X_{t, k}(f) \le \exp \lf(-\be^{3/2} \Theta(\by) - O_k(t^3 + \sqrt{t} \be^{5/4}) \rg),
$$
and so using that $\Theta(\by) = \al_k$ we get
$$
\nu_{k, t}(A_{\be(\ep)}) \le \ep \exp \lf(-\be(\ep)^{3/2} \al_k - O_k(t^3 + \sqrt{t} \be(\ep)^{5/4} \rg),
$$
which gives the desired upper bound after simplifying the expression and grouping errors.
\end{proof}
\begin{figure}
	\centering
	\includegraphics[width=6cm]{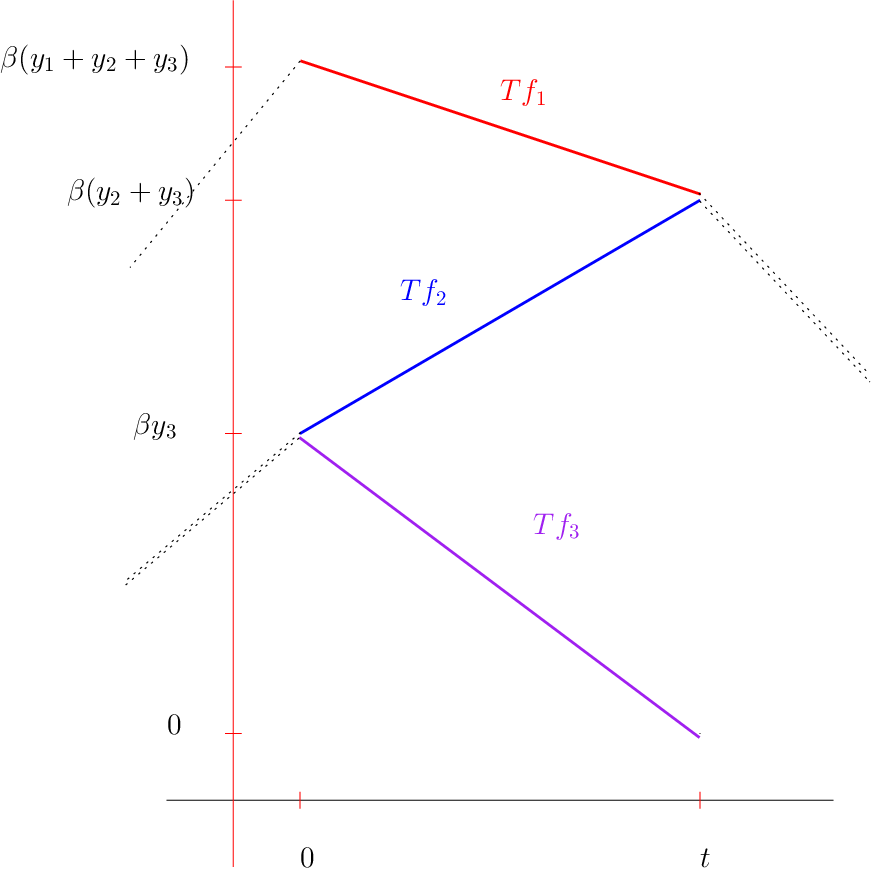}
	\caption{A sketch of $\T f$ for a function $f \in \sC^k_0([0, t])$ that maximizes $\S(f)$ relative to the probability that Brownian motion is close to $f$. Here $(y_1, y_2, y_3)$ is a nonnegative maximizer of the function $\Theta$ on the sphere and $\be > 0$.}
	\label{fig:optimal}
\end{figure}

Finally, we prove Remark \ref{R:RN-context}.

\begin{prop}
	\label{P:RN-context-example}
Let $\hat \cA_1 = \cA_1-\cA_1(0)$, the recentered Airy line ensemble. Then we can find a sequence of sets $D_n \sset \sC_0([0, 1])$ with $\mu_{1, 1}(D_n) \to 0$ such that
$$
\P(\hat \cA_1 \in D_n) \ge \mu_{1, 1}(D_n) e^{ c \log^{1/2}(\mu_{1, 1}(D_n)^{-1})}
$$
for all $n \in \N$.
\end{prop}

It is not difficult to use Theorem \ref{T:tetris-theorem} to find the optimal value of $c$ and show that the $\log^{1/2}$-decay in the exponent is sharp, but we do not pursue this as it is a rather peripheral point.

\begin{proof}
Let $\hat \fA_1(r) = \fA_1(r)-\fA_1(0) = \fA_1(r) - r^2 - \cA_1(0)$ and let $B:[0, t] \to \R$ be a Brownian motion. Fix $m > 0$ and consider the class of functions $A_{m, s} \sset \sC_0([0, 1])$ such that
$$
f(0) = 0, \qquad f(1) \in [-1, s], \qquad \inf f \ge - 1, \qquad f(1/2) \ge m.
$$
Equation \eqref{E:Xktf} implies that $X_{1, 1}(f)$ is bounded below by an absolute constant on $A_{m, 1}$, and $\hat \fA_1 \in A_{m, 1}$ implies $\hat \cA_1 \in A_{m + 1/4, 2}$. Combining this observation with standard estimates on Brownian motion, we get that
\begin{align*}
\P(\hat \cA_1 \in A_{m + 1/4, 2}) &\ge c \P(B \in A_{m, 1}) \ge c'e^{m/4} \P(B \in A_{m + 1/4, 2}).
\end{align*}
where $c, c',$ are absolute constants. Since $\P(B \in A_{m + 1/4, 2}) \ge d e^{- c'' m^2}$, taking $D_n = A_{n + 1/4, 2}$ gives the result.
\end{proof}

\section{Appendix: Proof of Lemma \ref{L:parabola-lemma-A}}

\label{S:appendix}

As discussed after the statement of Lemma \ref{L:parabola-lemma-A}, we may assume that $\al = 0$, since Brownian bridges commute with linear shifts.
\begin{proof}[Proof of Lemma \ref{L:parabola-lemma-A}.1]
	First, the bound is trivial for small $\la$ if $c > 0$ is large enough. Therefore throughout the proof we may assume $\la > 1000$.
	
	\textbf{Step 1: A Brownian bridge estimate.} \qquad 
	Let $B$ be a Brownian bridge from $(0, x)$ to $(\la,  y - \la^2)$. Now, let $\zeta = \la/\fl{\la}, m = \fl{\la}$ and define the mesh
	$$
	\Pi = \{\pi_0, \dots, \pi_{m}\} := \{0, i \zeta, \dots, \fl{\la} \zeta\}.
	$$
	Since $t \le \sqrt{x} + 1$ we have
	\begin{align}
	\label{E:P0la}
\P_{0, \la}(x, y + f(\la), f, [t, \la]) \le \P(B(n) > f(n) \text{ for all } n \in \Pi, n \ge \sqrt{x} + 1).
	\end{align}
	Now, letting $\pi_j$ be the smallest element of $\Pi \cap [ \sqrt{x} + 1, \la]$ we have that 
	$$
	\P(B(n) > f(n) \quad \text{ for all } n \in \Pi, n < \sqrt{x} + 1 \mid B(\pi_j) > f(\pi_j)) \ge e^{-c \la}
	$$ 
	for an absolute constant $c > 0$. To see this, observe that if $L:\R\to \R$ is the linear function satisfying $L(0) = x$ and $L(\pi_j) = B(\pi_j)$ for some $B(\pi_j) > f(\pi_j)$ then $L(s) + c' >  -s^2$ for an absolute constant $c' > 0$, and the probability that $B(n) \ge L(n) + c'$ for $n \in \Pi, 0 < n < \sqrt{x} + 1$ is easily bounded below by $e^{-c \la}$. Noting also that $B(0) = x > f(0) = 0$ yields the above display. Therefore \eqref{E:P0la} is bounded above by
	\begin{equation}
	\label{E:ecla}
e^{c \la} \P(B(n) > f(n) \text{ for all } n \in \Pi).
	\end{equation}
	\textbf{Step 2: Converting to a Dirichlet energy bound.} \qquad Consider the vector $\Delta B \in \R^{m}$ given by
	$
	\Delta B_i = B(\pi_i) - B(\pi_{i-1}),
	$
	and let $E$ be the minimal value of $\frac{1}{4 \zeta} \|\Delta B\|_2^2$ among all possible choices of $B$ with $B(n) > f(n)$ for all $n \in \Pi$. The vector $\Delta B = \sqrt{2 \zeta} N$, where $N$ is a standard normal vector on $\R^m$ conditioned on the event $\sum_{i=1}^m N_i = (y - \la^2 - x)/ \sqrt{2 \zeta}$, and so \eqref{E:ecla} is bounded above by
	$
	e^{c \la} \P(\|N\|^2_2 \ge 2E).
	$
	Next, by rotational invariance of the standard normal distribution, $\|N\|_2^2$ has the same distribution as if we conditioned $N$ to lie in the hyperplane $\{\bx \in \R^m : x_1 = (y - \la^2 - x)/ \sqrt{2 m \zeta} \}$. Therefore 
	$$
	\|N\|_2^2 = (y - \la^2 - x)^2/ (2 m \zeta) + \chi_{m-1},
	$$
	where $\chi_{m-1}$ is a $\chi$-squared random variable with $m-1$ degrees of freedom. Therefore letting $I = 2E - (y - \la^2 - x)^2/ (2 m \zeta)$, we have
	$$
	\P(\|N\|^2_2 > 2E) = \P(\chi_{m-1} \ge I) = \frac{1}{2^{(m-1)/2} \Ga((m-1)/2)} \int_{I}^\infty x^{(m-1)/2-1} e^{-x/2} dx.
	$$
	Making the substitution $y = x/2$, the right hand side equals
	\begin{align*}
	\frac{1}{\Ga((m-1)/2)} \int_{I/2}^\infty y^{(m-1)/2-1} e^{-y} dy.
	\end{align*}
	If $I \ge 2$, then this is bounded above by
	\begin{align*}
	\frac{1}{\Ga((m-1)/2)} \int_{I/2}^\infty y^{\cl{(m-1)/2}-1} e^{-y} dy
	&= \frac{(\cl{(m-1)/2}-1)!}{\Ga((m-1)/2)} e^{-I/2} \sum_{k=0}^{\cl{(m-1)/2} - 1} \frac{(I/2)^k}{k!} \\
	&\le e^{-I/2} (I/2 + 9)^{m/2}.
	\end{align*}
	If $I \le 2$, then $\P(\chi_{n-1} \ge I) \le e^{-I/2} (I/2 + 9)^{m/2}$ trivially. Putting everything together and using that $m \zeta = \la$ we get that 
	\begin{align}
	\label{E:almost-bd}
	\P_{0, \la}(x, y + f(\la), f, [t, \la]) \le \exp \lf(c \la - E + \frac{(y - \la^2 - x)^2}{4 \la} + c\la \log (E + 9) \rg).
	\end{align}
	\textbf{Step 2: Estimating the Dirichlet energy $E$.} \qquad
At this point, the problem is purely deterministic. Let $f_\Pi:[0, \la] \to \R$ be the function given by linearly interpolating between the values of $f$ on the mesh $\Pi$. We can then equivalently write $E$ as the minimum value of the Dirichlet energy
$$
E(g) = \frac{1}{4} \int_0^\la |g'(s)|^2 ds
$$
among all functions $g:[0, \la] \to \R$ such that $g \ge f_\Pi$ and $g(0) = x, g(\la) = y - \la^2$. The Dirichlet energy minimizer $g_\Pi$ is the smallest concave function with these properties. Since $f_\Pi$ itself is concave by the concavity of $f$, $g_\Pi$ has the following form:
\begin{itemize}[nosep]
	\item There are values $a \le b \in \Pi$ such that $g_\Pi$ is linear on $[0, a]$ and $[b, \la]$. 
	\item $g_\Pi(z) = f(z)$ for all $z \in (a, b)$.
\end{itemize}
The fact that $g_\Pi$ dominates $f_\Pi$ forces $a \le \sqrt{x}$.  Also, the fact that $g_\Pi$ is concave forces
$$
\frac{(-a^2 - x)}{a} \ge \frac{-(a+\zeta)^2 + a^2}{\zeta},
$$
which implies that $a > \sqrt{x} - \zeta$. Hence $a = \fl{\sqrt{x}}_\Pi$, the largest element of $\Pi$ that is less than or equal to $\sqrt{x}$. Similarly $b = \cl{\la - \sqrt{y}}_\Pi$, the smallest element of $\Pi$ that is greater than or equal to $\la - \sqrt{y}$. The function $g_\Pi$ can be thought of as a discretized version of the function $g$ introduced after the statement of Lemma \ref{L:parabola-lemma-A} whose energy equals $E(x, y, \la)$. Moreover, a computation shows that $|E(x, y, \la) - E(g_\Pi)| \le c \la$. Finally, since $\sqrt{x} < \la/2, \sqrt{y} < \la/2$, we have that $E(x, y, \la) \le 2 \la^3$, so \eqref{E:almost-bd} is bounded above by
$$
\exp \Big(c \la \log (1 + \la) - E(x, y, \la) + \frac{(y - \la^2 - x)^2}{4 \la} \Big),
$$
as desired.
\end{proof}

\begin{proof}[Proof of Lemma \ref{L:parabola-lemma-A}.2] First, the bound holds trivially for small $\la$, so we may assume $\la > 1000$. Also, to simplify notation we write $u = x -a$ in the proof. We have
	\begin{align}
	\nonumber
	\P(&B(s) > f(s) \text{ for all } s \in [t, \la], B(t) \le u)\\
	\nonumber
	&= \P(B(s) > f(s) \text{ for all } s \in [t, \la] \mid B(t) \le u)\P(B(t) \le u) \\
	\nonumber
	&\le \P_{t, \la}(u, f(\la) + y, f)\P(B(t) \le u) \\
	\label{E:Ptla}
	&\le c\P_{t, \la}(u, f(\la) + y, f)\exp \lf(-\frac{a^2}{4t} - \frac{(u - y - \la^2)^2}{4(\la - t)} + \frac{(x - y - \la^2)^2}{4\la} \rg).
	\end{align}
	Now, since Brownian bridge law commutes with affine shifts,
	$$
	\P_{t, \la}(u, f(\la) + y, f) = \P_{0, \la - t}(u + t^2, f(\la - t) + y, f).
	$$
	Bounding this with Lemma \ref{L:parabola-lemma-A}.1, we get that \eqref{E:Ptla} is bounded above by
	$$
	\exp \lf(-\frac{a^2}{4t} - E(u + t^2, y, \la - t) + \frac{(u+ t^2 - y - (\la-t)^2)^2}{4(\la - t)} - \frac{(u - y - \la^2)^2}{4(\la - t)} + \frac{(x - y - \la^2)^2}{4\la} + c\la \log(1 + \la) \rg).
	$$
	Now,
	$$
	\frac{(u+ t^2 - y - (\la-t)^2)^2}{4(\la - t)} - \frac{(u- y - \la^2)^2}{4(\la - t)} \le - \lambda^2 t + O(t^3 + t x),
	$$
	where we have bounded the error term by using that $t \le \sqrt{x} + 1 \le 2\la/3$ and $\sqrt{x}, \sqrt{y} \le \la/2$. We also have that
	$$
	E(x, y, \la) - E(u + t^2, y, \la - t) \le \frac{2}{3} x^{3/2} - \frac{2}{3}u^{3/2} + \la^2 t + \frac{t^3}{3}
	$$
	which combined with the previous computations gives the desired result.
\end{proof}

\begin{proof}[Proof of Lemma \ref{L:parabola-lemma-A}.3]
For $x, y > 0$ with $\sqrt{x} < \la/2, \sqrt{y} < \la/2$, let $g_{x, y}$ be the function defined immediately after Lemma \ref{L:parabola-lemma-A}, and for every $i$, define $h_i:[0, \la] \to \R$ by
$$
h_i(s) = g_{x_i - (k + 1-i)/2, y_i - (k + 1-i)/2}(s) + (k + 1-i)/2.
$$
Then $h_i(0) = x_i, h_i(\la) = y_i - \la^2$. Moreover, since $g_{x', y'} \ge g_{x, y} \ge f$ for any $x' \ge x$ and $y' \ge y$ and
$$
x_k, y_k \ge 1, \quad x_i - x_{i+1} \ge 1, \quad y_i - y_{i+1} \le 1,
$$
we have that $h_k - f \ge 1/2$ and $h_i - h_{i+1} \ge 1/2$ for all $i \le k-1$. Therefore letting $B$ denote a $k$-tuple of independent Brownian bridges from $(0, \bx)$ to $(\la, \by - \la^2)$, we have that
\begin{equation}
\label{E:P-product}
\begin{split}
\P_{0, \la}(\bx, \by + f(\la), f) &\ge \P(\|B_i - h_i\|_\infty < 1/4 \text{ for all } i \in \II{1, k}) \\
&= \prod_{i=1}^k \P(\|B_i - h_i\|_\infty < 1/4).
\end{split}
\end{equation}
It just remains to bound $\P(\|B_i - h_i\|_\infty < 1/4)$ below for each $i$. We use the Cameron-Martin theorem for Brownian bridges. Indeed, letting $\mu_i$ denote the law of $B_i$ and $\nu_i$ denote the law of $B_i + h_i$, the Cameron-Martin theorem yields that
$$
\frac{d \mu_i}{d \nu_i}(W) = \exp \lf( \frac{1}{4} \int_0^\la [h_i'(s)]^2 ds - \frac{1}{2} \int_0^\la h_i' dW + \frac{(x_i + \la^2 - y_i)^2}{4\la}\rg).
$$
Note that the first term in the exponent above equals $E(x_i - (k+1 - i)/2, y_i - (k + 1-i)/2, \la)$.
Now, standard estimates on the maximum absolute value of a Brownian bridge imply that 
$$
\nu_i(w \in \sC([0, \la]): \|w - h_i\|_\infty < 1/4) \ge e^{-c \la}.
$$
On the other hand, for a Brownian sample path $W$ with $W(0)= h_i(0), W(\la) = h_i(\la)$, by integration by parts and stochastic integration by parts we have that
\begin{align*}
\lf| \int_0^\la h_i^{\prime} dW - \int_0^\la (h_i^{\prime}(s))^2 ds \rg| &= \lf| \int_0^\la h_i''(s) W(s) ds - \int_0^\la h_i''(s) h_i(s) ds \rg| \\
&\le \int_0^\la |h_i''(s)| |W(s) - h_i(s)| ds 
\end{align*}
This is bounded above by $\la/2$ when $\|W - h_i\|_\infty < 1/4$ since $|h_i''| \le 2$. Therefore $\P(\|B_i - h_i\|_\infty < 1/4)$ is bounded below by
$$
\exp \lf(- c \la - E(x_i - (k+1 - i)/2, y_i - (k + 1-i)/2, \la) + \frac{(x_i + \la^2 - y_i)^2}{4\la}  \rg).
$$
We can complete the proof by appealing to \eqref{E:P-product} and recognizing that 
\[
E(x_i - (k+1 - i)/2, y_i - (k + 1-i)/2, \la) \le E(x_i, y_i, \la) + c k (\sqrt{x} + \sqrt{y}) \le E(x_i, y_i, \la) + c k \la. \qedhere
\]
\end{proof}

	\bibliographystyle{dcu}
	\bibliography{bibliography}
	
\end{document}